\documentclass[reqno,12pt]{amsart}
\usepackage[english]{babel}

\usepackage{amsfonts,amsmath,amsthm,color,graphicx,xfrac,physics,mathrsfs,dirtytalk,mathtools,amssymb,bm,enumitem,a4wide}
\usepackage[colorlinks,citecolor=green,linkcolor=red]{hyperref}

\newcommand\Q{\mathbb Q}
\newcommand\R{\mathbb R}
\newcommand{\RR}{\R}
\newcommand\Z{\mathbb Z}
\newcommand{\ZZ}{\Z}
\newcommand\N{\mathbb N}
\newcommand{\NN}{\N}

\newcommand{\LL}{\mathscr L}
\newcommand\calH{\mathscr H}
\newcommand{\HH}{\calH}

\newcommand\Vout{V_\mathrm{out}}
\newcommand\Vin{V_\mathrm{in}}
\newcommand\Vmid{V_\mathrm{mid}}
\newcommand\Qout{Q_\mathrm{out}}
\newcommand\Qin{Q_\mathrm{in}}
\newcommand\overlineQin{\overline Q_\mathrm{in}}
\newcommand\Qmid{Q_\mathrm{mid}}

\newcommand{\FF}{\mathcal{F}}

\newcommand\rad{\mathrm{rad}}
\newcommand\cone{\mathrm{cone}}

\newcommand\conv{\mathrm{conv\,}}
\newcommand\ext{\mathrm{ext\,}}
\newcommand\diam{\mathrm{diam}}
\newcommand\diag{\mathrm{diag}}
\newcommand\dist{\mathrm{dist}}

\newcommand\SO{{SO}}
\newcommand{\supp}{{\mathrm {supp\,}}}

\newtheorem{theorem}{Theorem}
\newtheorem{lemma}[theorem]{Lemma}
\newtheorem{proposition}[theorem]{Proposition}
\newtheorem{remark}[theorem]{Remark}

\newtheorem{definition}[theorem]{Definition}
\newtheorem{thm}[theorem]{Theorem}
\newtheorem{lem}[theorem]{Lemma}
\newtheorem{prop}[theorem]{Proposition}

\theoremstyle{definition}
\newtheorem{defn}[theorem]{Definition}

\theoremstyle{remark}
\newtheorem{rem}[theorem]{Remark}

\numberwithin{theorem}{section}
\numberwithin{equation}{section}

\newcommand{\DIFF}{{\mathrm {D}}}

\newcommand{\BV}{{\mathrm {BV}}}

\newcommand{\rmc}{{\mathrm {c}}}

\newcommand{\rmloc}{{\mathrm {loc}}}

\DeclareMathOperator*{\argmin}{arg\,min}

\newcommand{\defeq}{\vcentcolon=}
\newcommand{\eqdef}{=\vcentcolon}

\def\Xint#1{\mathchoice
	{\XXint\displaystyle\textstyle{#1}}%
	{\XXint\textstyle\scriptstyle{#1}}%
	{\XXint\scriptstyle\scriptscriptstyle{#1}}%
	{\XXint\scriptscriptstyle\scriptscriptstyle{#1}}%
	\!\int}
\def\XXint#1#2#3{{\setbox0=\hbox{$#1{#2#3}{\int}$}
		\vcenter{\hbox{$#2#3$}}\kern-.5\wd0}}

\def\dashint{\Xint-}

\newcommand{\mres}{\mathbin{\vrule height 1.6ex depth 0pt width 0.13ex\vrule height 0.13ex depth 0pt width 1.3ex}}

\newcommand{\fr}{\penalty-20\null\hfill$\blacksquare$}     

\let\oldchi=\chi
\renewcommand{\chi}{\text{\raisebox{\depth}{\(\oldchi\)}}}

\newcommand\eps{\varepsilon}
\let\epsilon\eps

\begin{document}
		\title[Hessian-Schatten variation]{Functions with bounded Hessian-Schatten variation: density, variational and extremality properties}

	\author[L.\ Ambrosio]{Luigi Ambrosio}
\address{L.~Ambrosio: Scuola Normale Superiore, Piazza dei Cavalieri 7, 56126 Pisa} 
\email{\tt luigi.ambrosio@sns.it}

	\author[C.\ Brena]{Camillo Brena}
	\address{C.~Brena: Scuola Normale Superiore, Piazza dei Cavalieri 7, 56126 Pisa} 
	\email{\tt camillo.brena@sns.it}
	
	\author[S.\ Conti]{Sergio Conti}
\address{S.~Conti:  Institut f\"ur Angewandte Mathematik,
	Universit\"at Bonn, 53115 Bonn} 
\email{\tt sergio.conti@uni-bonn.de}

	\begin{abstract}
        In this paper we analyze in detail a few questions related to the theory of functions with bounded
        $p$-Hessian--Schatten total variation, which are relevant in connection with the theory of inverse problems and machine learning.
We prove an optimal density result, relative to the 
        $p$-Hessian--Schatten total variation, of continuous piecewise linear (CPWL)
        functions in any space dimension $d$,
        using a construction based on a mesh
        whose local orientation is adapted to the function to be approximated.
        We show that not all extremal functions  with respect to the $p$-Hessian--Schatten total variation are CPWL. Finally, we prove existence of minimizers of certain relevant functionals involving the $p$-Hessian--Schatten total variation in the critical dimension $d=2$.
	\end{abstract}
	\maketitle
	\tableofcontents

\tableofcontents
\section*{Introduction}
Broadly speaking, the goal of  an inverse problem is to reconstruct an unknown signal of interest from a collection of (possibly noisy)   observations. Linear inverse problems, in particular,  are prevalent in various areas of signal processing. They are defined via the specification of three principal components: 
\begin{itemize}
\item[$\bullet$]a hypothesis space $\mathcal{S}$ from which we aim to reconstruct the unknown signal $f^*\in\mathcal{S}$,
\item [$\bullet$]a linear forward operator $\nu:\mathcal{S}\rightarrow \mathbb{R}^N$ that models the data acquisition process,
\item[$\bullet$] the observed data that is stored in an array ${y}\in\mathbb{R}^N$, with the implicit assumption that ${y}\approx {\nu}(f^*)$. 
\end{itemize}The task is then to (approximately) reconstruct the unknown signal $f^*$ from the observed data ${y}$.  From a variational perspective, the problem  can be formulated as a minimization of the form 
\begin{equation}\label{eq:inv_pb}
	f^* \in \argmin_{f\in\mathcal{S}} 
	\lambda\mathcal E\left({\nu}(f),{y}\right) +  \mathcal{R}(f),
\end{equation}
where
\begin{itemize}
	\item [$\bullet$]$\mathcal E:\mathbb{R}^N\times\mathbb{R}^N\rightarrow\mathbb{R}$ is a convex loss function that measures the data discrepancy,
	\item[$\bullet$] $\mathcal{R}:\mathcal{S}\rightarrow\mathbb{R}$ is the regularization functional that enforces prior knowledge and regularity on the reconstructed signal,
	\item [$\bullet$] $\lambda>0$ is a tunable parameter that adjusts the two terms. 
\end{itemize}


In general, regularization (obtained by the presence of $\mathcal R$) enhances the stability of the problem and alleviates its inherent ill-posedness.  Also, the presence of $\mathcal R$ leads to a key theoretical result, the so called ``representer theorem'', that provides a parametric form for optimal solutions of~\eqref{eq:inv_pb} and has been recently extended to cover generic convex optimization problems over Banach spaces \cite{Boyer2018,bredies2020sparsity,Unser2020b,unser2021convex}. In simple terms (and under suitable assumptions), this abstract results characterizes the solution set of~\eqref{eq:inv_pb} in terms of the extreme points of the unit ball of the regularization functional \begin{equation}\label{cdsa}
 \{f\in\mathcal{S}: \mathcal{R}(f)\leq 1\}.
 \end{equation} Hence, the original problem can be translated in finding the extreme points of the unit ball appearing in~\eqref{cdsa}.


\smallskip

In this paper, we are going to study problems arising from a particular, yet general, choice of the items appearing in the functional in~\eqref{eq:inv_pb}. In particular,
\begin{enumerate}[label={\textit{\alph{enumi}}})]
	\item the hypothesis space are the functions $f:\Omega\rightarrow\RR$ with bounded $p$-Hessian--Schatten variation (see item~\ref{itemee}), for some $\Omega\subseteq\RR^d$ open. The space coincides indeed with Demengel's space (\cite{Deme}) of functions with bounded Hessian, 
	which has been introduced to study models of plastic deformations of solids and has proven useful also in the context of image processing, but the norm we adopt is specific and allows for optimal approximation
	results by continuous and piecewise affine functions when $p=1$; 
	\item 	  \label{itemee}the regularizing term is the $p$-Hessian--Schatten variation $|\DIFF_p^2 \,\cdot\,|(\Omega)$, that coincides with the relaxation of the functional  (here and after $|\,\cdot\,|_p$ denotes the $p$-Schatten norm),
	$$
	|\DIFF_p^2 f|(\Omega)\defeq \int_{\Omega} |\nabla ^2 f|_p\dd\LL^d\qquad\text{for every }f\in C^2(\Omega);
	$$
	This is a variant of the classical second-order total variation (\cite{Aziznejad2021HSTV}). It
 has been inspired by  \cite{hinterberger2006variational,bergounioux2010second,knoll2011second,lefki2013HS,lefki2013Poisson} and used in \cite{campos2021HTV,pourya2022delaunay};
	\item  in the critical case $d=2$ we consider as linear forward operator the evaluation functional at certain points $x_1,\ldots,x_N\in \RR^2$, with observed data $(y_1,\ldots,y_N)\in\RR^N$; 
		\item still in the critical case, the error term is taken to be an $\ell^q$ norm, i.e.\ $$\mathcal E (f)\defeq\Vert (f(x_i)-y_i)_{i=1,\ldots,N}\Vert_{\ell^q}.$$
 	\item 	 the tunable parameter is $\lambda\in (0,\infty]$, where by convention $\lambda=\infty$ imposes a perfect fit with the data.
\end{enumerate}

\smallskip

In view of the discussion above, it is evident that some questions arise as natural.
\begin{enumerate}[label=\roman*)]
	\item\label{intro1} The description of the extremal points of the ball (cf.~\eqref{cdsa})
\begin{equation}\label{campsmkc}
		\{f:\Omega\rightarrow\RR:|\DIFF^2_p f|(\Omega)\le 1\}
\end{equation}
modulo additive affine functions (since the Hessian--Schatten seminorm is invariant under 
	the addition of affine functions, this factorization is necessary).
	A reasonable description of these extremal points was given in \cite{ambrosio2022linear}, under the assumption that a certain density conjecture holds true. Namely, it has been proved that if $\rm CPWL$ functions are dense in energy in the space of functions with bounded Hessian--Schatten variation, then all extremal points, which obviously are on the sphere, are found in the closure of the $\rm CPWL$  extremal points (and this last set is rather manageable, see \cite{ambrosio2022linear}). Here and below, a $\rm CPWL$ (Continuous and PieceWise Linear) function is a piecewise affine function, affine on certain simplexes.
	 In Section~\ref{sectcpwl}  we give a positive answer to the just mentioned conjecture, proved only in the two-dimensional case in  \cite{ambrosio2022linear} with a different, more constructive, strategy.
	 As any CPWL function can be exactly represented by a neural network with rectified linear unit (ReLU) activation functions \cite{arora2016understanding}, our result (Theorem~\ref{corollary}) in particular
implies approximability of any function whose Hessian has bounded total variation
by means of neural networks with ReLU activation functions, with convergence of the $1$-Hessian-Schatten norm.
	\item \label{intro2} Again with respect to the  extremal points of the set described in~\eqref{campsmkc}, one may wonder whether all the extremal points are $\rm CPWL$. By a delicate measure-theoretic analysis, in Section~\ref{sectcones} we show that the answer is negative: functions whose graphs are cut cones are extremal,
	modulo affine functions, and these functions are not $\rm CPWL$ if $d\geq 2$. In connection with this negative answer, as for compact convex sets exposed points are dense in the class of extreme points, it would be interesting to know whether cut cones are also exposed, namely if there exist linear continuous functionals attaining their minimum, when restricted to the closed unit ball of the Hessian-Schatten seminorm, only at a cut cone.
	\item \label{intro3} In the two-dimensional case, one may wonder whether the functional~\eqref{eq:inv_pb} admits minimizers, with the choice of error and regularizing term described above. In Section~\ref{sectionsolutions} we give a positive answer,  for a large set of choices of the parameters $\lambda$, $p$ and $q$.
\end{enumerate}

\smallskip

Now we pass to a more detailed description of the content of the paper. Namely, we examinate separately the answers to items~\ref{intro1},~\ref{intro2} and~\ref{intro3} above and we sketch their proofs.

\subsection*{Density of CPWL functions} In Section~\ref{sectcpwl} we address the problem of density in energy $|\DIFF^2_1\,\cdot\,|(\Omega)$ of $\rm CPWL$ functions in the set of functions with bounded Hessian--Schatten variation. Our main result is Theorem~\ref{theodensity}, stated for $C^2$ targets, and then it follows the localized version Theorem~\ref{corollary} for targets with finite $p$-Hessian--Schatten variation. The proof of Theorem~\ref{theodensity} heavily relies on a fine study of triangulations of $\RR^d$ and consists morally of three parts.
\medskip
\\\textbf{Part 1} is Section~\ref{subsecgenprop}  and deals with general properties of triangulations (considered as couples of sets, the set of vertices and the set of elements), the most important ones  being the Delaunay, non degeneracy and uniformity properties (items~\ref{deftriangdel},~\ref{deftriangreg} and~\ref{deftriangunif} of Definition~\ref{deftriang}).  Roughly speaking, the Delaunay property states that given an element of the triangulation, no vertex of the triangulation lies inside the circumsphere of the given element. It entails regularity properties, among them, the fact that angles in the elements are not too small. This leads to the non degeneracy property, crucial to estimate geometric quantities related to an element in terms of the volume of the given element. Finally, uniformity states that the vertices of the triangulation look like a rotation of a rescaling of the lattice $\ZZ^d$. The main results are Lemma~\ref{lemmaconstrE}, that allows us to gain a Delaunay triangulation starting from a {uniform} set of vertices and Lemma~\ref{lemmabrzd} which studies Delaunay triangulation whose vertices locally coincide with a rotation of a rescaling of the lattice $\ZZ^d$.
\medskip
\\\textbf{Part 2} is Section~\ref{secconstrtriang} and aims at constructing a \say{good} triangulation (in the sense of \textbf{Part 1})  that locally follows a prescribed orientation. The outcome is Theorem~\ref{theogridrot}  and the main difficulty in its proof relies in \say{gluing} the various sub-triangulations to allow for the variable orientation (see Figure~\ref{fig1}).
\medskip
\\\textbf{Part 3} is the proof of the density result, Section~\ref{sectproof}.  We exploit the outcome of \textbf{Part 2} to build a triangulation that locally follows the orientation  given by the Hessian of $w$, $\nabla ^2 w$, in the sense that is given by an orthonormal basis of eigenvectors for $\nabla ^2 w$. Then we take $u$, the affine interpolation for $w$ with respect to this triangulation, which will be a good approximation. The contribution of the Hessian--Schatten variation of $u$ on regions in which the orientation of the triangulation is constant (and hence adapted to the Hessian of $w$) is estimated thanks to the good choice of the orientation, whereas the contribution around the boundaries of these regions, i.e.\ where the gluing took place, comes from the regularity properties of the triangulation and the smallness of these regions.
\subsection*{Extremality of cones} In Section~\ref{sectcones}, we prove that functions whose graphs are cut cones are extremal with respect to the Hessian--Schatten total variation seminorm. Namely, we prove that functions defined as
	$$f^\cone(x)\defeq(1-|x|)_+$$
	are extremal modulo affine functions, in the sense that if for some $\lambda\in(0,1)$ $$f^\cone=\lambda f_1+(1-\lambda) f_2$$ with $$|\DIFF_p^2 f_1|(\RR^d)=|\DIFF_p^2 f_2|(\RR^d)=|\DIFF_p^2 f^\cone|(\RR^d),$$
	for some $p\in[1,\infty)$, then $f_1$ and $f_2$ are equal to $f^\cone$, up to affine functions (Theorem~\ref{coneexthm}). 	
	\bigskip
	
	Our strategy is as follows. First, we set $f_i^\rad$ to be the radial symmetrization of $f_i$, for $i=1,2$. As $f^\cone$ is radial, a simple computation yields that still 
	$$f^\cone=\lambda f_1^\rad+(1-\lambda) f_2^\rad$$ 
	and
	 $$|\DIFF_p^2 f_1^\rad|(\RR^d)=|\DIFF_p^2 f_2^\rad|(\RR^d)=|\DIFF_p^2 f^\cone|(\RR^d).$$
	 This implies with not much effort that $f_i^\rad=f^\cone$, up to affine terms, thanks to the explicit computation of Hessian--Schatten total variation of radial functions (Proposition~\ref{htvradial}).

	 The bulk of the proof is then to prove that whenever we have $f$ such that $f^\rad=f^\cone$ and $|\DIFF_p^2 f|(\RR^d)=|\DIFF_p^2 f^\cone|(\RR^d)$, then $f$ equals to $f^\cone$, up to affine terms. In other words, in the case $f^\rad=f^\cone$, we have rigidity of the property  that $|\DIFF^2_p f^\rad|(\RR^d)\le |\DIFF^2_p f|(\RR^d)$
stated in Lemma~\ref{radialbetter}. 
	\medskip\\\textbf{Case $p=1$} is dealt in Proposition~\ref{propconerad}. 
	 For its proof, a key remark is the fact that, if $\boldsymbol{\Delta}$ denotes the distributional Laplacian,  then $\int_{B_1}\boldsymbol{\Delta} (f(U\,\cdot\,))$ is independent of $U\in SO(\RR^d)$. Hence, by $f^\rad=f^\cone$, we have that
	 	$$
	 	\int_{B_1}\boldsymbol{\Delta} f=	 	\int_{B_1}\boldsymbol{\Delta} f^\cone=-|\DIFF_1^2 f^\cone|(B_1)
		=-|\DIFF_1^2 f|(B_1),
	 	$$
	 	where the second inequality is obtained by explicit computation (or by concavity of $f^\cone$ in $B_1$). This then implies that (at the right hand side there is the total variation of the matrix valued measure $\DIFF\nabla f$ with respect to the $1$-Schatten norm) 
	 	$$
	 \int_{B_1}\dd\tr(\DIFF\nabla f)=-	\int_{B_1}\dd|\DIFF\nabla f|_1,
	 	$$
	 	so that $\tr(\DIFF\nabla f)=-|\DIFF\nabla f|_1$  almost everywhere, which implies that the eigenvalues of $\DIFF\nabla f$ are all negative, almost everywhere (Lemma~\ref{lemmaconvex}), by rigidity in the inequality $|\Tr(A)|\le |A|_1$. Then, by Lemma~\ref{lemmaconvex1}, it follows that $f$ has a continuous concave representative in $B_1$. Finally we exploit concavity to obtain the pointwise bound $f\ge f^\cone$ in $B_1$, which, combined with the integral equality $f^\rad=f^\cone$, implies the claim.
	 	\medskip\\\textbf{Case $p\in (1,\infty)$} is dealt in Proposition~\ref{propconeradp}, where we reduce ourselves to the case $p=1$, namely we show that the information $|\DIFF_p^2 f|(\RR^d)=|\DIFF_p^2 f^\cone|(\RR^d)$, coupled with $f^\rad=f^\cone$, self improves to $|\DIFF_1^2 f|(\RR^d)=|\DIFF_1^2 f^\cone|(\RR^d)$, whence we can use what proved in the \textbf{Case $p=1$}. This reduction is done treating separately the absolutely continuous and singular part of $|\DIFF_p^2 f|$. The former is treated exploiting the strict convexity of the $p$-Schatten norm together with the  scaling property of the map $p\mapsto |\DIFF_p^2 f^\cone|$, whereas the latter is treated by Alberti's rank 1 Theorem (\cite{alberti_1993}), in conjunction with the fact that the $p$-Schatten norm of rank $1$ matrices is independent of $p$.
	 
\subsection*{Solutions to the minimization problem} In Section~\ref{sectionsolutions} we restrict ourselves to the two dimensional Euclidean space. Indeed, we want to exploit the continuity of functions with bounded Hessian--Schatten variation in dimension $2$ (\cite{ambrosio2022linear}, see Proposition~\ref{sobo}) to have a meaningful evaluation functional and define, for $\Omega\subseteq\RR^2$ open (cf.~\eqref{eq:inv_pb}), 
$\FF_\lambda:L^1_\rmloc(\Omega)\rightarrow [0,\infty]$ by
\begin{equation}\label{defFF}
	\FF_\lambda(f)= |\DIFF_1^2 f|(\Omega)+\lambda \Vert(f(x_i)-y_i)_{i=1,\ldots,N}\Vert_{\ell^1},
\end{equation}
where $x_1,\ldots,x_N\in\Omega$ are distinct points and $y_1,\ldots,y_N\in\RR$. 
Also, we are adopting the convention that $\infty\,\cdot\,0=0$, hence, if $\lambda=\infty$, we have 
$\FF_{\infty}:L^1_\rmloc(\Omega)\rightarrow [0,\infty]$,
\begin{equation}\notag
	\FF_{\infty}(f)=\begin{cases}
		|\DIFF^2_1 f|(\Omega)\qquad&\text{if $f(x_i)=y_i$ for $i=1,\ldots,N$,}\\
		\infty\qquad&\text{otherwise}.
	\end{cases}
\end{equation}
Notice that $\FF_\lambda$ is the sum of the regularizing term $|\DIFF^2_1 f|$ and the weighted (by $\lambda$) error term $\lambda \Vert(f(x_i)-y_i)_{i=1,\ldots,N}\Vert_{\ell^1}$ and that $\FF_\lambda$ can be seen as a relaxed version of 
$\FF_{\infty}$. 

\bigskip

In Section~\ref{sectionsolutions}, we will consider slightly more general functionals, see~\eqref{deffunct}, but for the sake of clarity we reduce ourselves to a particular case in this introduction. Our aim is to prove existence of minimizers of $\FF_\lambda$ (Theorem~\ref{csaaa}). Notice that in higher ($\ge 3$) dimension, $\FF_\lambda$ is not well defined (by the lack of continuity), and, even if we try to define it imposing continuity on its domain, minimizers do not exist in general, as the infimum of $\FF_\lambda$ is always zero. To see this last claim, simply exploit the scaling property of the Hessian--Schatten total variation (or use Proposition~\ref{htvradial}) for functions of the kind $x\mapsto y_i(1-|x-x_i|/r)_+$ as $r\searrow 0$.
\bigskip

We sketch now the proof of the existence of minimizers of $\FF_\lambda$. There are two key steps. We denote $\lambda_c\defeq 4\pi$, the \say{critical} value for $\lambda$.
\medskip
\\\textbf{Step 1}.
First we prove existence of minimizers of $\FF_{\lambda}$, for $\lambda\in [0,\lambda_c]$. This is done via  the direct method of calculus of variations, after we prove relative  compactness of minimizing sequences and semicontinuity of this functional. Compactness, proved in Proposition~\ref{compacntess},  is mostly due to the estimates of \cite{ambrosio2022linear}, see Proposition~\ref{sobo}. Semicontinuity is then proved in Lemma~\ref{lowersemicont} and here the choice of $\lambda\in [0,\lambda_c]$ plays a role. The key idea is  that, given a point $x_i$ and a converging sequence $f_k\rightarrow f$, either $|\DIFF^2_1 f_k|$ concentrates at $x_i$ or it does not. In the former case (Lemma~\ref{excess}), as a part of $|\DIFF_1^2 f_k|$ concentrates at $x_i$ (and $|\DIFF_1^2 f|(x_i)=0$, being points of codimension $2$), we experience a drop in the regularizing term of the functional, and this drop is enough to offset the lack of convergence of the evaluation term $f_k(x_i)$ in the error term. In the latter case (Lemma~\ref{excess} again), we have instead convergence of $k\mapsto f_k(x_i)$.  
\medskip
\\ \textbf{Step 2}.
We prove the existence of minimizers of $\FF^\lambda$, for $\lambda\in [\lambda_c,\infty]$. By \textbf{Step 1}, we can take a minimizer $f$ of $\FF_{\lambda_c}$. Then we modify $f$ to obtain $\tilde f$ satisfying
$$
|\DIFF^2_1 \tilde f|(\Omega)\le |\DIFF_1^2 f|(\Omega)+\lambda_c \Vert (f(x_i)-y_i)_i\Vert_{\ell^1}\qquad\text{and}\qquad\tilde f(x_i)=y_i\text{ for }i=1,\ldots,N.
$$
Such modifications is obtained adding to $f$ a suitable linear combination of  \say{cut-cones}, namely functions $x\mapsto y_i(1-|x-x_i|/\bar r)_+$ for $\bar r$ small enough.
As $\tilde f$ has a perfect fit with the data, for any $\lambda$,
$$
\FF_\lambda(\tilde f)=\FF_{\lambda_c}(\tilde f)\le \FF_{\lambda_c}(f),
$$
where the inequality is due to the construction of $\tilde f$.
Now, as $\FF_\lambda\ge \FF_{\lambda_c}$ (here the choice $\lambda\in [\lambda_c,\infty]$ plays a role) and as $f$ is a minimizer of $\FF_{\lambda_c}$, we see that $\tilde f$ is a minimizer of $\FF_\lambda$.

\medskip
Therefore, putting together what seen in  \textbf{Step 1} and  in \textbf{Step 2} we have that for every $\lambda\in [0,\infty]$ there exists a minimizer of $\FF_\lambda$.

\section{Preliminaries}
In this short section we first recall basic facts about Hessian--Schatten seminorms and then in 
Section~\ref{sectexplicit} we add an explicit formula to compute Hessian--Schatten variations of radial functions.
\subsection{Schatten norms}
We recall basic facts about Schatten norms, see \cite{ambrosio2022linear} and the references therein.
\begin{defn}[Schatten norm]
	Let $p\in[1,\infty]$. If $M\in\RR^{d\times d}$ and $s_1(M),\ldots,s_d(M)\ge 0$ denote the singular values of $M$ (counted with their multiplicity), we define the Schatten $p$-norm of $M$  
	by $$|M|_p\defeq\Vert(s_1(M),\ldots,s_d(M))\Vert_{\ell^p}.$$
\end{defn}
We recall that the scalar product between $M,\,N\in\RR^{d\times d}$ is defined by $$M\,\cdot\, N\defeq \tr(M^t N)=\sum_{i,\,j=1,\ldots,d} M_{i,j}N_{i,j}$$ and induces the Hilbert--Schmidt norm.  Next,  we enumerate several properties of the Schatten norms that shall be used throughout the paper
\begin{prop}\label{zuppa}
	The family of Schatten norms satisfies the following properties.
	\begin{enumerate}[label=\roman*)]
		\item \label{zuppa1}If $M\in\RR^{d\times d}$ is symmetric, then its singular values $s_1(M),\ldots,s_d(M)$ are equal to $|\lambda_1(M)|,\ldots,|\lambda_d(M)|$, where $\lambda_1(M),\ldots,\lambda_d(M)$ denote the eigenvalues of $M$ (counted with their multiplicity). Hence $|M|_p=\Vert(\lambda_1(M),\ldots,\lambda_d(M))\Vert_{\ell^p}$.
		\item \label{zuppa2} If $M\in\RR^{d\times d}$ and $N\in O(\RR^d)$, then $|M N|_p=|N M|_p=|M|_p$.
		\item \label{zuppa3} If $M,\,N\in\RR^{d\times d}$, then $|M N|_p\le |M|_p |N|_p$.
		\item \label{zuppa4} If $M\in\RR^{d\times d}$, then $|M|_p=\sup_N M\,\cdot\, N$, where the supremum is taken among all $N\in\RR^{d\times d}$ with $|N|_{p^*}\le 1$, for $p^*$ the conjugate exponent of $p$, defined by $1/p+1/p^*=1$.
		\item\label{zuppa5} If $M$ has rank $1$, then $|M|_p$ coincides with the Hilbert-Schmidt norm of $M$ for every $p\in [1,\infty]$. 
		\item \label{zuppa6} If $p\in(1,\infty)$, then the Schatten $p$-norm is strictly convex.
		\item \label{zuppa7} If $M\in\RR^{d\times d}$, then $|M|_p\le C|M|_q$, where  $C=C(d,p,q)$ depends only on $d$, $p$ and $q$. 
	\end{enumerate}
\end{prop}

\begin{defn}[$L^r$-Schatten norm]
	Let $p,\,r\in[1,\infty]$ and let $M\in C_\rmc(\RR^d)^{d\times d}$. We define the Schatten $(p,r)$-norm of $M$ by
	$$
	\Vert M\Vert_{p,r}\defeq \Vert |M|_p\Vert_{L^r(\RR^d)}.
	$$
\end{defn}

\subsubsection{Poincar\'e inequalities} We recall basic facts about Poincaré inequalities. 
\begin{defn}\label{defnPoin}
	Let $A\subseteq\RR^d$ be a domain. We say that $A$ supports Poincar\'e inequalities if for every $q\in[1,d)$ there exists a constant $C=C(A,q)$ depending on $A$ and $q$ such that
	$$
	\bigg(\dashint_A \Big|f-\dashint_A f\Big|^{q^*}\dd\LL^d\bigg)^{1/q^*}\le C \bigg(\dashint_A|\nabla  f|^q\dd\LL^d\bigg)^{1/q}\qquad\text{for every }f\in W^{1,q}(A),
	$$
	where $1/{q^*}=1/q-1/d$.
\end{defn}

%

\subsection{Hessian--Schatten total variation}
For this section fix $\Omega\subseteq\RR^d$ open and $p\in[1,\infty]$. We let $p^*$ denote the conjugate exponent of $p$. Now we recall the definition of Hessian--Schatten total variation and some basic properties, see \cite{ambrosio2022linear} and the references therein.

\begin{defn}[Hessian--Schatten variation]\label{def:H-Svar}
Let $f\in L^1_\rmloc(\Omega)$. For every $A\subseteq\Omega$ open we define
	\begin{equation}\label{defd2}
		|\DIFF^2_p f|(A)\defeq \sup_F \int_A \sum_{i,\,j=1,\ldots,d} f \partial_i\partial_j F_{i,j}\dd\LL^d,
	\end{equation}
where the supremum runs among all $F\in C_\rmc^\infty(A)^{d\times d}$ with $\Vert F\Vert_{p^*,\infty}\le 1$.
We say that $f$ has bounded $p$-Hessian--Schatten variation in $\Omega$ if $|\DIFF^2_p f|(\Omega)<\infty$.
\end{defn}
\begin{rem}
	If $f$ has bounded $p$-Hessian--Schatten variation in  $\Omega$, then the set function defined in~\eqref{defd2} is the restriction to open sets of a finite Borel measure, that we still call $|\DIFF^2_p f|$. This can be proved with a classical argument, building upon \cite{DGL} (see also \cite[Theorem 1.53]{AFP00}).
	
By its very definition, the $p$-Hessian--Schatten variation is lower semicontinuous with respect to convergence in distributions.\fr
\end{rem}

For any couple $p,\,q\in [1,\infty]$, $f$ has bounded $p$-Hessian--Schatten variation if and only if $f$ has bounded $q$-Hessian--Schatten variation and moreover $$C^{-1}|\DIFF^2_p f|\le |\DIFF^2_q f|\le C|\DIFF^2_p f| $$
for some constant $C=C(d,p,q)$ depending only on $d$, $p$ and $q$.  This is due to equivalence of matrix norms.

The next proposition connects Definition~\ref{def:H-Svar} with Demengel's space of functions with bounded Hessian \cite{Deme}, namely
Sobolev functions whose partial derivatives are functions of bounded variation. We shall use $\DIFF$ to denote the
distributional derivative, to keep the distinction with $\nabla$ notation (used also for gradients of Sobolev
functions).

\begin{prop}\label{hessiananddiffgrad}
	Let $f\in L^1_{\rmloc}(\Omega)$. Then the following are equivalent:
	\begin{itemize}
		\item $f$ has bounded Hessian--Schatten variation in $\Omega$,
		\item $f\in W^{1,1}_\rmloc(\Omega)$ and $\nabla f\in \BV_\rmloc(\Omega;\RR^d)$ with $|\DIFF\nabla f|(\Omega)<\infty$.
	\end{itemize}  If this is the case, then, as measures,
	\begin{equation}\notag
		|\DIFF^2_p f|=\bigg|\dv{\DIFF\nabla f}{|\DIFF\nabla f|}\bigg|_p |\DIFF\nabla f|.
	\end{equation}
	In particular, there exists a constant $C=C(d,p)$ depending only on $d$ and $p$ such that 
	$$
	C^{-1}|\DIFF \nabla f|\le |\DIFF^2_p f|\le C|\DIFF \nabla f|
	$$
	as measures.
\end{prop}

\begin{prop}\label{myersserrin}
	Let $f\in L^1_\rmloc(\Omega)$. Then, for every $A\subseteq\Omega$ open, it holds
	$$
	|\DIFF^2_p f|(A)=\inf\left\{\liminf_k \int_A |\nabla^2 f_k|_p\dd\LL^d\right\}
	$$
	where the infimum is taken among all sequences $(f_k)\subseteq C^\infty(A)$ such that $f_k\rightarrow f$ in $L^1_\rmloc(A)$.
	If moreover $f\in L^1(A)$, the convergence in $L^1_\rmloc(A)$ above can be replaced by convergence in $L^1(A)$.
\end{prop}

In the statement of the next lemma and in the sequel we denote by $B_\epsilon(A)$ the open
$\epsilon$-neighbourhood of $A\subseteq\RR^d$.

\begin{lem}\label{mollif}
	Let $f\in L^1_\rmloc(\Omega)$ with bounded Hessian--Schatten variation in $\Omega$. Let also $A\subseteq\RR^d$ open and $\epsilon>0$ with $B_\epsilon(A)\subseteq\Omega$. Then, if $\rho\in C_\rmc(\RR^d)$ is a convolution kernel with $\supp\rho\subseteq B_\epsilon(0)$, it holds
	$$
	|\DIFF_p^2 (\rho\ast f)|(A)\le|\DIFF_p^2 f|(B_\epsilon(A)).  
	$$ 
\end{lem}

In the same spirit of Lemma~\ref{mollif}, we have the following lemma.
\begin{lem}\label{radialbetter}
	Let $f\in L^1_\rmloc(\Omega)$ with bounded Hessian--Schatten variation in $\Omega$. Assume that $A\subseteq\Omega$ is open and invariant under the action of $SO(\RR^d)$. For any $U\in SO(\RR^d)$ the
	function $f_U\defeq f(U\,\cdot\,)$ satisfies $|\DIFF^2_p f_U|(A)\le |\DIFF^2_p f|(A).$
	In particular, setting
	$$
	f^\rad\defeq \dashint_{SO(\RR^d)} f_U\dd\mu_d(U),
	$$
	where $\mu_d$ is the Haar measure on $SO(\RR^d)$, by convexity one has
	$$|\DIFF^2_p f^\rad|(A)\le |\DIFF^2_p f|(A).$$
\end{lem}
\begin{proof}
	The proof is very similar to the one of Lemma~\ref{mollif} above i.e.\  \cite[Lemma 12]{ambrosio2022linear}, but we sketch it anyway for the reader's convenience and for future reference.
	
We take any $F\in C_\rmc^\infty(A)^{n\times n}$ with $\Vert F\Vert_{p^*,\infty}\le 1$ and we set  $G\defeq U F(U^t\,\cdot\,) U^t$. A straightforward computation shows that $$\sum_{i,j}\partial_i\partial_j G_{i,j}(x)=\sum_{i,\,j}(\partial_i\partial_j F_{i,j})(U^tx)$$ and that $G\in C_\rmc^\infty(A)^{n\times n}$ with $\Vert G\Vert_{p^*,\infty}\le 1$.
Then  we compute, by a change of variables,
\begin{equation}\notag
	\begin{split}
		\int_A \sum_{i,\,j} f_U\partial_i\partial_j F_{i,j}\dd\LL^d&=\int_A f (x) \sum_{i,\,j}(\partial_i\partial_j F_{i,j})(U^t x)\dd\LL^d(x)\\&=\int_A f (x) \sum_{i,\,j}(\partial_i\partial_j G_{i,j})(x)\dd\LL^d(x).
\end{split}
\end{equation}
In particular, 
$$
\bigg|	\int_A \sum_{i,j} f_U\partial_i\partial_j F_{i,j}\dd\LL^d(x)\bigg|\le |\DIFF^2_p f|(A).
$$

Now, by Fubini's Theorem
\begin{align*}
		\int_A \sum_{i,j} f^\rad\partial_i\partial_j F_{i,j}\dd\LL^d&=\int_{SO(\RR^d)}\int_A f_U \sum_{i,j}\partial_i\partial_j F_{i,j}\dd\LL^d\dd\mu_d(U)\\&\le \int_{SO(\RR^d)}  |\DIFF^2_p f|(A)\dd\mu_d(U)= |\DIFF^2_p f|(A),
\end{align*}
whence the claim as $F$ was arbitrary.
\end{proof}

\begin{prop}[Sobolev embedding]\label{sobo}
	Let $f\in L^1_\rmloc(\Omega)$ with bounded Hessian--Schatten variation in $\Omega$. 
	Then
	\begin{alignat*}{3}
		&f\in L^{d/(d-2)}_{\rmloc}(\Omega)\cap W_\rmloc^{1,d/(d-1)}(\Omega)\qquad&&\text{if }d\ge 3,\\
		&f\in L^\infty_{\rmloc}(\Omega)\cap W^{1,2}_\rmloc(\Omega)\qquad&&\text{if }d= 2,\\
		&f\in L^\infty_{\rmloc}(\Omega)\cap W^{1,\infty}_\rmloc(\Omega)\qquad&&\text{if }d= 1\end{alignat*}
	and, if $d=2$, $f$ has a continuous representative.
		
	More explicitly, for every $A\subseteq\Omega$ bounded domain that supports Poincar\'e inequalities and $r\in[1,\infty)$, there exist $C=C(A,r)$ and an affine map $g=g(A,f)$ such that, setting $\tilde f\defeq f-g$, it holds that
	\begin{alignat*}{3}
	&\Vert \tilde f\Vert_{L^{d/(d-2)}(A)}+\Vert \nabla\tilde f\Vert_{L^{d/(d-1)}(A)}\le C|\DIFF^2 f|(A)\qquad&&\text{if }d\ge 3,\\
		&\Vert \tilde f\Vert_{L^r(A)}+\Vert \nabla \tilde f\Vert_{L^2(A)}\le C|\DIFF^2 f|(A)\qquad&&\text{if }d= 2,\\
		&\Vert \tilde f\Vert_{L^\infty(A)}+\Vert \nabla \tilde f\Vert_{L^\infty(A)}\le C|\DIFF^2 f|(A)\qquad&&\text{if }d= 1.
	\end{alignat*}
\end{prop}

\begin{lem}[Rigidity]\label{rigidity}
	Let $f,\,g\in L^1_\rmloc(\Omega)$ with bounded Hessian--Schatten variation in $\Omega$ and assume that 
	$$
	|\DIFF^2_p (f+g)|(\Omega)=|\DIFF^2_p f|(\Omega)+|\DIFF^2_p g|(\Omega).
	$$
	Then 
$$|\DIFF^2_p (f+g)|=|\DIFF^2_p f|+|\DIFF^2_p g|$$ 
as measures on $\Omega$.
\end{lem}

\subsection{Hessian--Schatten variation of radial functions}\label{sectexplicit}
The following result is new and aims at computing the Hessian--Schatten variation of radial functions. This will be needed in Section~\ref{sectcones} and Section~\ref{sectionsolutions}. Notice also that, as expected, the contribution involving the singular part of $|\DIFF g'|$ in~\eqref{csdcasdca} below does not depend on $p$. 

In the proof we shall use the
auxiliary function $F:(0,R)\times \RR^2\rightarrow[0,\infty)$
$$F(s,(v_1,v_2))\defeq d\omega_d\Vert (s v_2,v_1,\ldots,v_1)\Vert_{\ell^p}s^{d-2},$$
where $v_1$ is repeated $d-1$ times and $\omega_d:=\LL^d(B_1)$
($d$ will be the dimension of the Euclidean ambient space). 
Notice that $F$ is continuous, convex and $1$-homogeneous with respect to the $(v_1,v_2)$ variable. 
Therefore, for intervals $(r_1,r_2)\subseteq (0,R)$, the functional 
$$
\Phi_{(r_1,r_2)}(\mu):=\int_{(r_1,r_2)}F\bigg(s,\dv{\mu}{|\mu|}\bigg)\dd|\mu|=
\int_{(r_1,r_2)}F\bigg(s,\dv{\mu}{\lambda}\bigg)\dd\lambda
\qquad\text{whenever $|\mu|\ll\lambda$},
$$
defined on $\RR^2$-valued measures $\mu$ makes sense and is convex. Furthermore, 
Reshetnyak lower semicontinuity Theorem (e.g.\ \cite[Theorem 2.38]{AFP00}) grants
its lower semicontinuity with respect to weak convergence in duality with $C_{\rmc}((r_1,r_2))$.

\begin{prop}\label{htvradial}
	Let $d\ge 2$ and let $g\in L^1_{\rmloc}((0,R))\rightarrow\RR$ be such that $\int_0^r s^{d-1}|g(s)|\dd s<\infty$ for every $r\in (0,R)$. Define $f(\,\cdot\,)\defeq g(|\,\cdot\,|)\in L^1_\rmloc(B_R(0))$. 
	
	Assume that $f$ has bounded Hessian--Schatten total variation in $B_R(0)$.
	Then $g\in W^{1,1}_\rmloc((0,R))$ and $g'\in \BV_\rmloc((0,R))$. Write the  decomposition $\DIFF g'=\DIFF^s g'+g''\LL^1$, where $\DIFF^s g'\perp \LL^1$.
	Then, for every $r\in (0,R]$ and $p\in [1,\infty]$, one has
\begin{equation}\label{csdcasdca}
|\DIFF^2_p f|(B_r(0))= d\omega_d
\bigg(\int_{(0,r) }s^{d-1}\dd{|\DIFF^s g'|( s)}+\int_0^r \Vert(s g''(s),g'(s),\ldots,g'(s))\Vert_{\ell^p} s^{d-2}\dd s\bigg).
\end{equation}
	
Conversely, assume that $g\in W^{1,1}_\rmloc((0,R))$ and $g'\in \BV_\rmloc((0,R))$, and, with the same notation above, that 
	$$
	\int_{(0,R)} s^{d-1}\dd{|\DIFF^s g'|( s)}+\int_0^R\Vert(s g''(s),g'(s),\ldots,g'(s))\Vert_{\ell^p} s^{d-2}\dd s<\infty.
	$$
	Then $f$ has bounded Hessian--Schatten total variation in $B_R(0)$ and the Hessian--Schatten variation of $f$ is computed as above.
\end{prop}
\begin{proof}
	Let $r\in (0,R)$.
	Let $\rho_k$ be radial Friedrich mollifiers for $\RR^d$ and define $f_k\defeq \rho_k\ast f$. As $f_k$ is still radial, we write $f_k(\,\cdot\,)=g_k(|\,\cdot\,|)$, where $g_k\in C^\infty((0,r))$.
	As $f_k\rightarrow f\in L^1 (B_r(0))$, $g_k\rightarrow g$ in $L^1_\rmloc ((0,r))$. Now we compute, on $B_r(0)$,
	\begin{equation}\notag
		\nabla^2 f_k(x)=g_k''(|x|)\frac{x\otimes x}{|x|^2}+g_k'(|x|)\frac{|x|^2
			{\rm Id}-x\otimes x}{|x|^3}.
	\end{equation}
	Notice that the eigenvalues of the matrix appearing at the right hand side of the equation above are
	$g_k''(|x|)$ with multiplicity 1 and $g_k'(|x|)/|x|$ with multiplicity $d-1$, the eigenvectors being 
	$x$ and a basis of $x^\perp$.
	Therefore, by Proposition~\ref{hessiananddiffgrad}, on $B_r(0)$ one has
	\begin{equation}\label{eq:luigi1}
	|\DIFF^2_p f_k|=|x|^{-1}\big\Vert \big(|x|g_k''(|x|),g_k'(|x|),\ldots,{g_k'(|x|)}\big)\big\Vert_{\ell^p}\LL^d
	\geq g_k''(|x|)\LL^d.
	\end{equation}
	As $|\DIFF_p^2 f_k|(B_r(0))$ is uniformly bounded by Lemma~\ref{mollif}, we obtain the claimed membership for $g$, letting eventually $r\nearrow R$. 
	
	For the purpose of proving the inequality $\geq$ in \eqref{csdcasdca}.
It is enough to compute $|\DIFF^2_p f|(A_{r_1,r_2})$, where we define the open annulus
	$$
	A_{r_1,r_2}\defeq B_{r_2}(0)\setminus \bar{B}_{r_1}(0)
	$$
	for $[r_1,r_2]\subseteq (0,R)$. Also, there is no loss of generality in assuming that $r_1$ and $r_2$ are such that $|\DIFF g'|(\{r_1\})=|\DIFF g'|(\{r_2\})=0$, as well as $|\DIFF\nabla f|(\partial A_{r_1,r_2})=0$, hence we will tacitly assume this condition in what follows.
	
	From \eqref{eq:luigi1}, with the notation $\mu_g\defeq (g'\LL^1,\DIFF g')$, we get
	\begin{equation}\notag
	|\DIFF^2_p f_k|(A_{r_1,r_2})=\int_{A_{r_1,r_2}}|\DIFF^2_pf_k|(x)\dd\LL^d(x)=\Phi_{(r_1,r_2)}(\mu_{g_k}).
	\end{equation}
	Now notice that Lemma~\ref{mollif} and our choice of radii grant $|\DIFF^2_p f|(A_{r_1,r_2})=\lim_k |\DIFF^2_p f_k|(A_{r_1,r_2})$,
	so that the lower semicontinuity of $\Phi$ together with the weak* convergence of $\mu_{g_k}$ to $\mu_g$
	grants
	\begin{align*}
	|\DIFF^2_p f|(A_{r_1,r_2})&\ge \Phi_{(r_1,r_2)}(\mu_g)\\&=d\omega_d
	\bigg(\int_{(r_1,r_2)} s^{d-1}\dd{|\DIFF^s g'|( s)}+\int_{r_1}^{r_2} \Vert(s g''(s),g'(s),\dots,g'(s))\Vert_{\ell^p} s^{d-2}\dd s\bigg).
	\end{align*}
	Letting $r_1\to 0$ and $r_2\to r$ provides the inequality $\geq$ in \eqref{csdcasdca}.
	
	Now we prove the converse implication and inequality. This time we denote by $(\rho_k)$ a sequence of Friedrich mollifiers on $\RR$ and we call $g_k\defeq\rho_k\ast g$, then $f_k(\,\cdot\,)\defeq g_k(|\,\cdot\,|)$. Notice that, with our choice of the radii, $|\mu_{g_k}|((r_1,r_2))$ converges to $|\mu_g|((r_1,r_2))$ as $k\to\infty$, therefore invoking
	Reshetnyak continuity Theorem (e.g.\ \cite[Theorem 2.39]{AFP00}) we get
	\begin{align*}
	|\DIFF^2_p f|(A_{r_1,r_2})&\leq\liminf_k |\DIFF^2_p f_k|(A_{r_1,r_2})=
	\liminf_k \Phi_{(r_1,r_2)}(\mu_{g_k})\\
	&=\Phi_{(r_1,r_2)}(\mu_g)\leq\Phi_{(0,R)}(\mu_g)\\
	&= d\omega_d\biggl(\int_{(0,R)} s^{d-1}\dd{|\DIFF^s g'|( s)}
	+\int_0^R\Vert(s g''(s),g'(s),\ldots,g'(s))\Vert_{\ell^p} s^{d-2}\dd s\biggr).
	\end{align*}
	Letting $r_1\to 0$ and $r_2\to R$ gives that $f$ has bounded Hessian--Schatten total variation in $B_R(0)\setminus\{0\}$. To conclude, obtaining also the converse inequality in \eqref{csdcasdca}, we need
	just to apply the classical Lemma~\ref{removal} below to $f$ and to the partial derivatives of $f$, taking into
	account the mutual absolute continuity of $|\DIFF^2_p f|$ and $|\DIFF\nabla f|$ (Proposition \ref{hessiananddiffgrad}).
	\end{proof}
	
\begin{lemma}\label{removal}
Let $B_R(0)\subseteq\RR^d$, $d\geq 2$ and let $h\in W^{1,1}(B_R(0)\setminus\{0\})$ (resp.\ $h\in \BV(B_R(0)\setminus\{0\})$). 
Then $h\in W^{1,1}(B_R(0))$ (resp.\ $h\in \BV(B_R(0))$ and $|\DIFF h|(\{0\})=0$).
\end{lemma}
\begin{proof} By a truncation argument, we can assume with no loss of generality that $h$ is bounded.
Then, the approximation of $h$ by the functions $h_k=h(1-\psi_k)\in W^{1,1}(B_R(0))$ (resp.\ $\BV(B_R(0))$),
where $\psi_k\in C^1_{\rmc}(B_{1/k}(0))$ satisfy $|\nabla\psi_k|\leq 2 k$, $0\leq\psi_k\leq 1$ and $\psi_k=1$
in a neighbourhood of $0$,
together with Leibniz rule, provides the result.
\end{proof}

\section{Density of CPWL functions}\label{sectcpwl}
We recall the definition of continuous piecewise linear ($\rm CPWL$) functions. In view of this definition we state that a simplex in $\R^d$ is the convex hull of $d+1$ points (called vertices of the simplex) that do not lie on an hyperplane, and a face of a simplex is the convex hull of a subset of its vertices.
\begin{defn}
	Let $\Omega\subseteq\RR^d$ open and let $f\in C(\Omega)$. We say that $f$ is $\rm CPWL$ (or $f\in{\rm CPWL}(\Omega)$)
	if there exists a decomposition of $\RR^d$ in $d$-dimensional simplexes $\{P_k\}_{k\in\NN}$, such that
	\begin{enumerate}[label=\roman*)]
		\item $P_k\cap P_h$ is either empty or a common face of $P_k$ and $P_h$, for every $h\ne k$;
		\item for every $k$, the restriction of $f$ to $P_k\cap \Omega$ is affine;
		\item the decomposition is locally finite, in the sense that for every ball $B$, only finitely many $P_k$ intersect $B$.
	\end{enumerate}
\end{defn}
The main theorem of this section is the following density result.
\begin{theorem}\label{theodensity}
	For any  $w\in C^2(\R^d)$ there exists a sequence $(u_j)\subseteq{\rm CPWL}(\RR^d)$ with $u_j\rightarrow w$ in the $L^\infty_\rmloc(\RR^d)$ topology and such that	for any bounded open set  $\Omega\subseteq\R^d$ with $\LL^d(\partial\Omega)=0$,
	\begin{equation}\notag
		\lim_{j\to\infty} |\DIFF^2_1u_j|(\Omega)\to 
	 |\DIFF^2_1 w|(\Omega).
	\end{equation}
\end{theorem}
Recall that, as explained in \cite[Remark 22]{ambrosio2022linear}, because of lower semicontinuity the exponent $p=1$ is the only meaningful exponent in a density result as above, namely this sharp approximation by ${\rm CPWL}$ functions is not possible for the energy
$|\DIFF^2_p f|$ when $p>1$.

We defer the proof of Theorem~\ref{theodensity} to Section~\ref{sectproof}, after having studied properties of \say{good} triangulations in Section~\ref{subsecgenprop} and Section~\ref{secconstrtriang}. Namely, we aim to construct triangulations of $\R^d$ which locally follow a prescribed orientation. The general scheme is illustrated in Figure~\ref{fig3}.
In each of the large squares it coincides with a rotation of a triangulation of $\eps\Z^d$; the difficulty resides in the interpolation region between different squares. 
In Section~\ref{subsecgenprop} we discuss standard material on general properties of triangulations. 
In Section~\ref{secconstrtriang} we present the specific construction, the key result is Theorem~\ref{theogridrot}.
This is then used to prove density in Theorem~\ref{theodensity}.

\bigskip
First, we start with a brief discussion around the result of Theorem~\ref{theodensity}.
We recall the following extension result, \cite[Lemma 17]{ambrosio2022linear}. Its last claim is immediate, once one takes into account also Proposition~\ref{sobo}.
\begin{lem}\label{extension}
	Let $\Omega\defeq(0,1)^d\subseteq\RR^d$ and let $f\in L^1_\rmloc(\Omega)$ with bounded Hessian--Schatten variation in $\Omega$. Then there exist an open neighbourhood  
	$\tilde\Omega$ of $\bar\Omega$ and $\tilde f\in L^1_\rmloc(\tilde \Omega)$ with bounded Hessian--Schatten variation in  $\tilde \Omega$ such that
\begin{equation}\label{cdsasc}
		|\DIFF^2_1 \tilde f|(\partial \Omega)=0
\end{equation}
and
$$
\tilde f=f\qquad\text{a.e.\ on $\Omega$}.
$$
In particular, $f\in L^1(\Omega)$.
\end{lem}

The following result gives a positive answer to \cite[Conjecture 1]{ambrosio2022linear}, partially
proved in the two-dimensional case in \cite[Theorem 21]{ambrosio2022linear}. The proof is based on
Theorem~\ref{theodensity} and a diagonal argument.

\begin{theorem}\label{corollary}
	Let $\Omega\defeq(0,1)^d\subseteq\RR^d$. Then $\rm CPWL$ functions are dense with respect to the energy $|\DIFF^2_1\,\cdot\,|(\Omega)$ in the space
	$$
	\{f\in L^1_\rmloc(\Omega):f\text{ has bounded Hessian--Schatten variation in }\Omega\}
	$$
	with respect to the $L^1(\Omega)$ topology.
	Namely, for any $f\in L^1_\rmloc(\Omega)$ with bounded Hessian--Schatten variation in $\Omega$, there exists $\{f_k\}_k\subseteq{\rm CPWL}(\Omega)$ with $f_k\rightarrow f$ in $L^1(\Omega)$  and $|\DIFF_1^2 f_k|(\Omega)\rightarrow|\DIFF_1^2 f|(\Omega)$.
\end{theorem}
\begin{proof}
	Take $f$ as in the statement, and let $\tilde{f}$ be given by Lemma~\ref{extension}. By  using smooth cut-off functions, there is no loss of generality in assuming that $\tilde{f}$ is compactly supported in $\tilde \Omega$, hence, in particular, $\tilde f\in L^1(\RR^d)$. Also, we see that we can assume that $\LL^d(\partial\tilde\Omega)=0$.
	
	 Now we take $(\tilde f_k)\subseteq C_\rmc^\infty(\RR^d)$ be mollifications of $\tilde f$ by means of compactly supported mollifiers, notice that $\tilde{f}_k\rightarrow\tilde f$ in $L^1(\RR^d)$ and $|\DIFF^2_1 f_k|(\tilde \Omega)=|\DIFF^2_1 f_k|(\RR^d)\rightarrow|\DIFF^2_1 \tilde f|(\RR^d) =|\DIFF^2_1 \tilde f|(\tilde\Omega)$, thanks to Proposition~\ref{mollif} and lower semicontinuity. Now, for any $k$, take $(\tilde f_{k,h})\subseteq{\rm CPWL}(\RR^d)$ be given by Theorem~\ref{theodensity} for $\tilde f_k$.  With a diagonal argument, we obtain $(g_\ell)\subseteq{\rm CPWL}(\RR^d)$ with $g_\ell\rightarrow \tilde f$ in $L^1(\tilde\Omega)$ and such that $|\DIFF_1^2 g_\ell|(\tilde\Omega)\rightarrow|\DIFF^2_1 \tilde f|(\tilde\Omega)$. By lower semicontinuity, the fact that $|\DIFF_1^2 g_\ell|(\tilde\Omega)\rightarrow|\DIFF^2_1 \tilde f|(\tilde\Omega)$ and~\eqref{cdsasc}, it easily follows that 
	 $$
	 |\DIFF_1^2 g_\ell|(\Omega)\rightarrow|\DIFF^2_1 \tilde f|(\Omega)=|\DIFF^2_1  f|(\Omega).
	 $$
	Clearly,  $g_\ell\rightarrow f$ in $L^1(\Omega)$,
	so that the proof is concluded.
 \end{proof}

\begin{rem} 
Let $\Omega\defeq(0,1)^d$. As a consequence of Theorem~\ref{corollary}, the description of the extremal points of the unit ball with respect to the $|\DIFF_1^2\,\cdot\,|(\Omega)$ seminorm obtained in \cite[Theorem 25]{ambrosio2022linear} remains in place in arbitrary dimension. In a slightly imprecise way, the result states that $\rm CPWL$ extremal points are dense in $1$-Hessian--Schatten energy in the set of extremal points with respect to the $L^1(\Omega)$ topology. Notice that the description of $\rm CPWL$ extremal points is made explicit in   \cite[Proposition 23]{ambrosio2022linear}. \fr	
\end{rem}

\begin{figure}
\includegraphics[width=6cm]{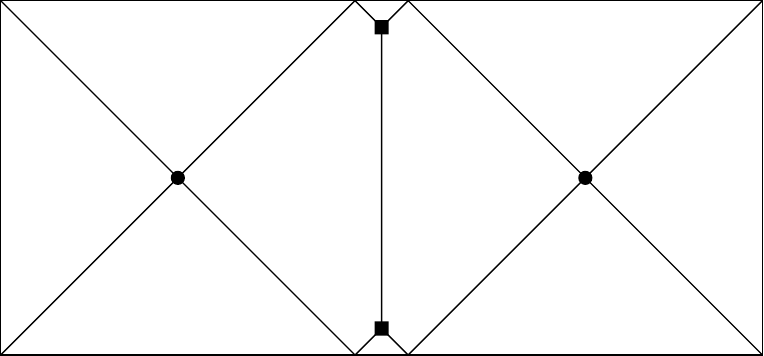}
\caption{Sketch of the function $G_h$ used in proving Remark~\ref{remnotclosed}. 
The function equals 1 on the two points marked by black dots, $-h$ on the two points marked by black squares, vanishes
outside the large rectangle, and is
affine in each of the ten polygons in the figure.
}
 \label{fignotflosed}
\end{figure} 

\begin{rem}\label{remnotclosed}
The set of extremal points is not closed with respect to the convergence considered here. 
 For example, with $d=2$, 
one can easily check that the function $g(x):=\max\{1-\Vert x\Vert_{\ell^\infty},0\}$ is extremal, but the function $G_0(x):=g(x+e_1)+g(x-e_1)$ is not. 
 Indeed, $G_0=\frac 12 (2 g(\cdot + e_1) + 2 g (\cdot -e_1))$, with
 $|\DIFF_p^2G_0|(\R^2)=|\DIFF_p^22 g(\cdot + e_1)|(\R^2)=
 |\DIFF_p^22 g(\cdot - e_1)|(\R^2)$.
For $h\in (0,1/4)$ we then define
 $G_h:\R^2\to\R$ by
 \begin{equation}\notag
  G_h(x):=\max\bigl\{1-\Vert x-(1+h)e_1\Vert_{\ell^\infty}, 1-\Vert x+(1+h)e_1\Vert_{\ell^\infty}, -\dist_{\ell^\infty}(x, \partial R_h)\bigr\}
 \end{equation}
if $x\in R_h:=[-2-h,2+h]\times [-1,1]$,
and $G_h(x)=0$ if $x\in \R^2\setminus R_h$ (see Fig.~\ref{fignotflosed}). Then each $G_h$ is CPWL, is extremal, and $G_h\to G_0$ uniformly with 
$|\DIFF_p^2G_h|(\R^2)\to
|\DIFF_p^2G_0|(\R^2)$ for any $p\in[1,\infty]$, but $G_0$ is not extremal.

Let us briefly comment on the proof of extremality of $G_h$ (the same argument implies extremality of $g$). 
If $G_h=\lambda f+(1-\lambda) f'$, with $\lambda\in (0,1)$ and
$|\DIFF_p^2f|(\R^2)=
|\DIFF_p^2f'|(\R^2)=
|\DIFF_p^2G_h|(\R^2)$, then by
Lemma~\ref{rigidity} the support of $|\DIFF_p^2f|$
is contained in the support of 
$|\DIFF_p^2G_h|$, so that $f$ (after choosing the continuous representative) is affine in each of the sets on which $G_h$ is affine. 
Adding an irrelevant affine function, we can reduce to the case that $f=0$ outside $R_h$.
Using
the fact that if two affine functions coincide on three non-collinear points then they coincide everywhere, one obtains $f=aG_h$, where $a:=f((1+h)e_1)\in\R$  (see Fig.~\ref{fignotflosed}); by equality of the norms $a=\pm 1$. Similarly, $f'=\pm G_h$, so that by $G_h=\lambda f+(1-\lambda) f'$
we obtain $G_h=f=f'$.\fr
\end{rem}

\subsection{General properties of triangulations}\label{subsecgenprop}
We define a triangulation of $\R^d$ as a pair of two sets, the first one, $V$, containing the vertices (nodes), the second one, $E$, containing the elements, which are nondegenerate compact simplexes with pairwise disjoint interior. Each simplex is the convex hull of its $d+1$ vertices. One further requires a compatibility condition that ensures that neighbouring elements share a complete face (and not a strict subset of a face). 
We remark that there is a large literature which studies this in the more general framework of simplicial complexes. For the present application the metric and regularity properties are crucial, we present in this section the few properties which are relevant here in a self-contained way.

\begin{definition}\label{deftriang}
A triangulation of $\R^d$ is a pair $(V,E)$, with
$V\subseteq\R^d$  and $E\subseteq\mathcal P(\R^d)$ such that
\begin{enumerate}[label=\rm{\roman*)}]
 \item \label{deftriangv} for every $e\in E$, $e$ has non empty interior and 
there is $v_e\subseteq \R^d$ with $\# v_e=d+1$ and $e=\conv(v_e)$; 
\item\label{deftriangV}
$V=\bigcup_{e\in E} v_e$;
\item\label{deftriangeep}
for any $e,\,e'\in E$ one has
$e\cap e'=\conv(v_e\cap v_{e'})$; 
\item $\bigcup_{e\in E}e=\R^d$.
\end{enumerate}
We introduce four regularity properties:
\begin{enumerate}[label={\rm(\alph{enumi})}]
 \item \label{deftriangdel}
The triangulation has the Delaunay property if for each $e\in E$, 
the unique open ball $B$ with $v_e\subseteq \partial B$ 
obeys $B\cap V=\emptyset$. 
\item\label{deftriangreg} The triangulation is $c_*$-non degenerate, for some $c_*>0$, if $(\diam \, e)^d\le c_*\LL^d(e)$ for all $e\in E$.
\item\label{deftriangunif} The set $V\subseteq\R^d$ is $(\bar c, \eps)$-uniform, for some $\bar c,\,\eps>0$, if $|x-y|\ge \eps/\bar c$  for all   $x\in V,\,y\in V$ with $x\ne y$ and
$B_{\bar c\eps}(q)\cap V\ne\emptyset$ for all $q\in\R^d$.
\item\label{locallyfinite} The triangulation is locally finite if, for every ball $B$, only finitely many elements of $E$ intersect $B$.
\end{enumerate}
\end{definition}
\begin{figure}
\begin{center}
 \includegraphics[width=7cm]{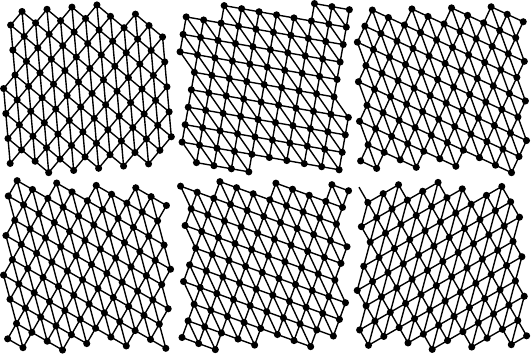} 
\end{center}
 \caption{Sketch of the desired triangulation without the interpolation region. Aim of this section is to find a suitable interpolation between the squares.}
 \label{fig3}
\end{figure}
Condition~\ref{deftriangeep} states that two distinct elements of $E$ are either disjoint or share a face of dimension between 0 and $d-1$; in particular distinct elements have disjoint interior. Notice that $\conv(\emptyset)=\emptyset$.

The Delaunay property~\ref{deftriangdel} states that the circumscribed sphere to each simplex does not contain any other vertex, and implies $\partial e\cap V=v_e$ for all $e\in E$. It can be interpreted as a statement that the vertices have been matched to form simplexes in an ``optimal'' way.

The non-degeneracy property~\ref{deftriangreg} states that simplexes are uniformly non-degenerate, so that the affine bijection that maps $e$ onto the standard simplex has a uniformly bounded condition number. It implies that 
there is $C=C(c_*,d)$ such that 
for any $e\in E$, any $x\in v_e$,
any $F\in \R^d$ one has
\begin{equation}\label{eqnondegsimpl}
 |F|\le C(c_*,d) \sum_{y\in v_e\setminus\{x\}} \frac{|F\cdot (y-x)|}{|y-x|}.
\end{equation}

The uniformity property~\ref{deftriangunif} of a set $V$ of vertices ensures (for Delaunay triangulations) that all sides of all elements have length comparable to $\eps$. Also, property~\ref{deftriangunif} immediately implies property~\ref{locallyfinite}, as it forces  $V$ to be a locally finite set.

\begin{rem}\label{remdiam}
	Let $(V,E)$ be a triangulation that has  the Delaunay property
	(property~\ref{deftriangdel}) and is $(\bar c,\eps)$-uniform (property~\ref{deftriangunif}). Then  $\diam(e)\le 2\bar c \eps$, for any $e\in E$.\fr
\end{rem}
\begin{proof}
	Take $e\in E$ and let $q\in \RR^d$ and $r\in(0,\infty)$ such that $v_e\subseteq \partial B_r(q)$. By the Delaunay property, $V\cap B_r(q)=\emptyset$, so that, by $(\bar c,\eps)$-uniformity, $\bar c\eps>r\ge \diam(e)/2$.
\end{proof}
\bigskip

We next show how given the set of vertices $V$ one can abstractly obtain a good triangulation.
The construction is standard up to a perturbation argument. As we could not find a reference with the complete result, we prove it.
 \begin{lemma}\label{lemmaconstrE}
Let $V\subseteq\R^d$ be uniform in the sense of property~\ref{deftriangunif} of Definition~\ref{deftriang}.
 Then there is $E\subseteq\mathcal P(\R^d)$ such that $(V,E)$ is a triangulation of $\R^d$ with the Delaunay property~\ref{deftriangdel}.
 \end{lemma}
\begin{proof}
	We define $f:\R^d\to[0,\infty]$ by
	\begin{equation*}
		f(x):=\begin{cases}
			|x|^2\qquad &\text{ if } x\in V,\\
			\infty\qquad&\text{ otherwise.}
		\end{cases}
	\end{equation*}
	Let $g$ be the convex envelope of $f$, which is 
	CPWL (see Lemma~\ref{lemmaconvcpwl} below). Moreover, notice that
\begin{equation*}
		g(x)=|x|^2=f(x)\qquad\text{for every }x\in V.
	\end{equation*}
	Let $q\in\R^d$, $\mu\in\R$ be such that
	\begin{equation}\label{eq:luigi99}
		A:=\{x: g(x)=\mu+2x\cdot q \}
	\end{equation}
	has nonempty interior.
 Notice that  $A$ is compact, convex  and coincides with the closure of its interior, and $g(x)> \mu+2x\cdot q$ for every $x\in\RR^d\setminus A$.
	Also, we set
	\begin{equation}\label{eqAaffx}
		w:=\{x\in V:  \mu+2x\cdot q = |x|^2 \}= A\cap V,
	\end{equation}
then,
	\begin{equation}\notag
	\mu+2x\cdot q < |x|^2 \qquad\text{for all } x\in V\setminus w.
\end{equation}

 Now we show that $\ext(A)\subseteq V$ so that $\ext(A)\subseteq w$ and hence  $A=\conv(w)$ with $\#w\ge d+1$ (as $A$ has nonempty interior). Take indeed $p\in \ext(A)$ and assume $p\notin V$. 
Then, take a minimal set of points $\{p_1,\dots,p_{k}\}\subseteq V$ such that 
$(p,g(p))\in \conv\big((p_1,f(p_1)),\dots, (p_{k},f(p_{k}))\big)$ (this is possible by~\eqref{convnoclosed} of Lemma~\ref{lemmaconvcpwl} below).
As $p\in\ext(A)$, up to reordering, we can assume that  $p_1\notin A$, hence by $g(p_1)>\mu+2p_1\cdot q$ we have that $g(p)>\mu+2p\cdot q$, a contradiction.

	The above equations can be rewritten as
	\begin{equation}\notag
		|x-q|^2=\mu+|q|^2 \qquad\text{for all }x\in w
	\end{equation}
	and
	\begin{equation}\notag
		|x-q|^2>\mu+|q|^2 \qquad\text{for all } x\in V\setminus w.
	\end{equation}
	We set $r:=\sqrt{\mu+|q|^2}$, so that these conditions are
	$w\subseteq \partial B_r(q)$ 
	and $V\cap B_r(q)=\emptyset$, 
	so that the set $w$ has the Delaunay property.
	
	  Notice then that for every $x\in V$, there is at least one set  $A$ as in \eqref{eq:luigi99}
	with nonempty interior and with  $x\in  A\cap V$ (this set was called $w$): this follows from the fact that $g$ is CPWL.
	
	Any decomposition of those elements $A$ in \eqref{eq:luigi99} with nonempty interior into non degenerate simplexes with vertices in $w$ leads to a pair $(V,E)$ with all 4 claimed properties of triangulations, except for~\ref{deftriangeep} of Definition~\ref{deftriang}. In the rest of the proof we show by a perturbation argument that a decomposition exists such that property~\ref{deftriangeep}, which relates neighbouring pieces in which $g$ is affine,  also holds.

	We first remark that property~\ref{deftriangeep} is automatically true if $g$ is non degenerate, in the sense that each $A$ is a simplex, which is the same as $\#w=d+1$  (we are going to add a few details about this in the sequel of the proof). In turn, this is true if 
	for every choice of $X:=\{x_1, \ldots, x_{d+2}\}\subseteq V$ the
	$d+2$ points $\{(x,g(x))\}_{x\in X}\in\R^{d+1}$ do not lie in a $d$-dimensional hyperplane, so that~\eqref{eqAaffx} cannot hold for all $x\in X$.

	We fix an enumeration $\varphi:V\to\N\setminus\{0,1\}$ and
	recall that $V$ is $(\bar c,\eps)$-uniform. For any $\rho\in(0,\eps\wedge 1]$ we consider
	$f_\rho:\R^d\to[0,\infty]$ defined by
	\begin{equation}\notag
		f_\rho(x):=\begin{cases}
			|x|^2+\rho^{\varphi(x)}\qquad&\text{ if } x\in V,\\
			\infty\qquad &\text{ otherwise.}
		\end{cases}
	\end{equation}
	For a given set $X:=\{x_1,\ldots, x_{d+2}\}\subseteq V$ consider the $d+2$ equations
	\begin{equation}\label{eqsystX}
		\mu+2x_i\cdot q =|x_i|^2+\rho^{\varphi(x_i)}\qquad\text{for } i=1,\ldots, d+2
	\end{equation}
	in the $d+1$ unknowns $(\mu,q)$.
	The affine map $T:\R^{d+1}\to\R^{d+2}$ defined by 
	$T_i(\mu,q):=\mu+2x_i\cdot q - |x_i|^2$
	has an image which is at most $d+1$ dimensional, hence contained in a set of the form $\{\Xi\in \R^{d+2}: \Xi\cdot\nu=a\}$ for some 
	$\nu\in S^{d+1}$, $a\in\R$ (which depend on $X$). If  the system~\eqref{eqsystX} has a solution, then
	\begin{equation}\notag
		\sum_{i=1}^{d+2} \nu_i \rho^{\varphi(x_i)}=a.
	\end{equation}
	As $|\nu|=1$ and the exponents are all distinct, this is a nontrivial polynomial equation in $\rho$, and has at most finitely many solutions.
	As there are countably many possible choices of the set $X\subseteq V$, 
	for all but countably many values of $\rho$ no such system has a solution. Therefore we can choose $\rho_j\searrow 0$ such that~\eqref{eqsystX} has no solution for any choice of $X$ with $X=\{x_1,\ldots,x_{d+2}\}\subseteq V$.

		Fix now an index $j$ and let $g_{\rho_j}$ be the convex envelope of $f_{\rho_j}$.  Notice that if $\rho_j$ is sufficiently small  (that we are going to assume from here on), then, as $V$ is discrete and $|x|^2$ is strictly convex,
	$$
	g_{\rho_j}(x)=|x|^2+\rho_j^{\varphi(x)}=f_{\rho_j}(x)\qquad\text{for every }x\in V.
	$$
	 Our choice of $\rho_j$ implies that for every $j$, for every choice of $X:=\{x_1, \ldots, x_{d+2}\}\subseteq V$ the
		$d+2$ points $\{(x,g_{\rho_j}(x))\}_{x\in X}\in\R^{d+1}$ do not lie in a $d$-dimensional hyperplane. Now pick $\mu,\,q$ such that
	\begin{equation}\notag
		A:=\{x: g_{\rho_j}(x)=\mu+2x\cdot q \}
	\end{equation}
	has nonempty interior (the function $g_{\rho_j}$ is CPWL, 
	by Lemma~\ref{lemmaconvcpwl} below).
 By non-degeneracy, arguing as above, $A=\conv(w)$, with $\#w=d+1$ and ${\rm Int}(A)\cap V=\emptyset$.
	We define $E_j$ as the family of those sets.
	
	Let us justify why $(V,E_j)$ is a triangulation of $\RR^d$. It is enough to show that property~\ref{deftriangeep} holds. Take then $e_1,e_2\in E_j$ (with vertices $w_1,w_2$), so that there exist two affine functions $L_1,L_2$ such that $g_{\rho_j}=L_i$ on $e_i $ and $g_{\rho_j}>L_i$ on $\RR^d\setminus e_i $, for $i=1,2$.  Assume that $\xi\in e_1\cap e_2$, so that $L_1(\xi)=g_{\rho_j}(\xi)=L_2(\xi)$. Take a minimal set $\{\zeta_1,\dots,\zeta_k\}\subseteq w_2$ with $\xi \in \conv(\{\zeta_1,\dots,\zeta_k\})$. As for every $a=1,\dots,k$, $L_2(\zeta_a)=g_{\rho_j}(\zeta_a)\ge L_1(\zeta_a)$, it follows that for every $a=1,\dots,k$, $g_{\rho_j}(\zeta_a)=L_1(\zeta_a)$ hence $\{\zeta_1,\dots,\zeta_k\}\subseteq w_1\cap w_2$.

	The conditions
	\begin{equation}\notag
		\mu+2x\cdot q = |x|^2 +\rho_j^{\varphi(x)}\ge |x|^2 \qquad\text{for all }x\in w
	\end{equation}
	and
	\begin{equation}\notag
		\mu+2x\cdot q\le |x|^2 +\rho_j^{\varphi(x)}\le |x|^2+\rho_j^2\qquad \text{for all } x\in V
	\end{equation}
	lead to
	\begin{equation}\notag
		|x-q|^2\le \mu+|q|^2\qquad \text{for all }x\in w
	\end{equation}
	and
	\begin{equation}\notag
		\rho_j^2+|x-q|^2\ge\mu+|q|^2\qquad \text{for all } x\in V.
	\end{equation}
	Therefore $w\subseteq \overline B_r(q)$, 
	and either $r\le \rho_j$ or $V\cap B_{r-\rho_j}(q)=\emptyset$, where
	$r:=\sqrt{\mu+|q|^2}$. 
	By uniformity of the grid, necessarily $r-\rho_j< \bar c \eps$, which gives $\diam(A)\le 2r<
	2\bar c\eps+2\rho_j\le
	2(\bar c+1)\eps$.
	
	For any $x\in V$, the possible choices of $e$ with $x\in v_e$ are restricted by $\diam(e)< 2(\bar c+1)\eps$, which implies $v_e\subseteq V\cap B_{2(\bar c+1)\eps}(x)$. As the grid is uniform, the latter set is finite, with a bound depending only on $\bar c$. Therefore for any $x\in V$ we can choose a subsequence of $\rho_j$ such that the set
	\begin{equation}\notag
		\{e\in E_j: x\in v_e\}
	\end{equation}
	is, after finitely many steps, constant. 
	As there are countably many $x\in V$, we can choose a common diagonal subsequence. Along this sequence, for any bounded set $K$ the set $\{e\in E_j: e\subseteq K\}$ is, after finitely many steps, constant.  Property~\ref{deftriangeep} holds for $E_j$, and therefore for those sets. Therefore we obtain a common set $E$ with all desired properties. We remark that indeed the Delaunay property follows from  the construction of $E$ and the discussion of the first part of the proof: indeed, if $e\in E$, it is easy to see that there exists an affine function coinciding with $g$ on $e$.
\end{proof}

We next present the result on the regularity of convex envelopes used above.
\begin{lemma}\label{lemmaconvcpwl}
 Let $V\subseteq\R^d$ be a uniform set of vertices, in the sense of 
 item~\ref{deftriangunif} of Definition~\ref{deftriang}. Let $f:V\to[0,\infty)$ be superlinear, in the sense that
 \begin{equation}\label{eqasssuperlinear}
  \lim_{x\in V,\  |x|\to\infty} \frac{f(x)}{|x|}=\infty.
 \end{equation}
Let $g:\R^d\to[0,\infty)$ be the convex envelope of $f$ ($f$ is extended by $\infty$ to $\R^d\setminus V$). Then $g$ is $\rm CPWL$.
Moreover,
\begin{equation}\label{convnoclosed}
	\{(x,g(x)):x\in\RR^d\}\subseteq\conv(\{(x,f(x)):x\in V\})
\end{equation}
(notice that we are \emph{not} taking the closure of the convex hull at the right hand side).
\end{lemma}
\begin{remark}
It is easy to verify what follows.
 \begin{enumerate}[label=\roman*)]
  \item 
The fact that $V$ is uniform implies that $g$ is real-valued.
\item The assumption of superlinearity is necessary. Indeed, consider $d=2$, $V=\Z^2$, $f(x)=|x|$. 
Obviously $g(x)\ge |x|$. For any $x\in\Q^2$ there is $n\in\N\setminus\{0\}$ such that 
$xn\in \Z^2$, which implies 
$g(x)\le (1-\frac1n) f(0)+\frac1n f(xn)=|x|$,
so that $g(x)=|x|$ on $\Q^2$.
As $g$ is a real-valued convex function, it is continuous. We conclude $g(x)=|x|$ on $\R^2$, which is not $\rm CPWL$.
 \end{enumerate}
\end{remark}
\begin{proof}[Proof of Lemma~\ref{lemmaconvcpwl}]
For $r\in(0,\infty)$, we write
 \begin{equation}\notag
  f_r(x):=\begin{cases} f(x) & \text{ if } x\in V\cap B_r,\\
           \infty & \text{ otherwise,}
          \end{cases}
 \end{equation}
and let $g_r\geq g$ be the convex envelope of $f_r$. 
Since $V$ is uniform, any set $V\cap B_r$ is finite, and therefore 
$g_r$ is $\rm CPWL$ on $\conv(V\cap B_r)$, and infinity outside. If $r\ge \bar c \eps$, with
$\bar c,\,\eps>0$ the constants from item~\ref{deftriangunif} of Definition~\ref{deftriang}, the set $V\cap B_r$ is nonempty.

We shall show below that for any $r>0$ there is $R>0$ such that $g=g_R$ on $B_{r/4}$. This implies that $g$ is CPWL on $B_{r/4}$ for any $r$, and therefore the assertion. The choice of $R$ (which depends on $f$ and $r$) is done in \eqref{eqchoiceR} below.

For $r\ge \bar c\eps$ we define
 $\alpha_r:=\max f(V\cap [-r,r]^d)$. 
We first prove that if
$R/\sqrt d> r\ge 4\bar c \epsilon$ then
\begin{equation}\label{eqgub}
 g_R(x)\le \alpha_{r} \text{ for all } x\in B_{r/2}.
\end{equation}
To see this, 
let $q_1,\ldots, q_{2^d}$ denote the vertices of the cube $[-1,1]^d$. 
By uniformity of $V$, for each $i$ we can pick 
$p_i\in V\cap B_{\bar c \epsilon}((r-\bar c \epsilon)q_i)$. 
One checks that
$B_{r/2}\subseteq (r-2\bar c \epsilon)[-1,1]^2\subseteq \conv(\{p_1,\dots, p_{2^d}\})$. 
As $p_i\in V\cap [-r,r]^d\subseteq V\cap B_R$, we  have
$g_R(p_i)\le f(p_i)\le \alpha_{r}$ for all $i$, and therefore
$g_R\le \alpha_{r}$ on $B_{r/2}$, which proves \eqref{eqgub}.

We next show that, if $R$ is chosen sufficiently large, then $g_R=g$ on $B_{r/4}$.
By convexity, \eqref{eqgub}, and $g_R\ge0$ we obtain $\mathrm{Lip}(g_R;B_{r/4})\le 4\alpha_r/r$. 
As $g_R$ is CPWL in $B_{r/4}$, for any $y\in B_{r/4}$ there is an affine function $a:\R^d\to\R$ such that $y\in T_a:=\{g_R=a\}\cap B_{r/4}$ and $T_a$ has  nonempty interior. The Lipschitz bound on $g_R$ then carries over to $a$, and we obtain $|\nabla a|\le 4\alpha_{r}/r$.
By convexity of $g_R$, we have $a\le g_R$, so that $a\le f$ on $V\cap B_R$.
In order to obtain the same inequality outside $B_R$, we consider any $x$ with
$|x|\ge R\ge r$. Then, recalling $y\in T_a\subseteq B_{r/4}$,
\begin{equation}\notag
 a(x)\le a(y)+|\nabla a|\,|x-y|
 \le \alpha_r + \frac{4\alpha_{r} }r\Big(|x|+\frac r4\Big)
 \le \frac{6\alpha_{r}}{r} |x|.
\end{equation}
Finally, by \eqref{eqasssuperlinear} we can choose $R> \sqrt d r$ such that
\begin{equation}\label{eqchoiceR}
 f(x)\ge \frac{6\alpha_r}{r} |x| \hskip1cm\text{ for all } x\in V\setminus B_R.
\end{equation}
Therefore $a\le f$ everywhere, which implies
$a\le g \le g_R$, and in turn
$g=g_R$ on $T_a$ and therefore on $B_{r/4}$.

We prove now \eqref{convnoclosed}. Take $x\in\RR^d$, so that, by what proved above, $g(x)=g_R(x)$ for some $R>0$. 
	Now notice that the epigraph of $g_R$ coincides with the convex hull of the epigraph of $f_R$ (here we are using that the convex hull of the epigraph of $f_R$ is closed), so that the conclusion is easily achieved.
\end{proof}

We next investigate in more detail Delaunay triangulations such that $V$ locally  coincides with $\Z^d$ (possibly up to translations and rotations).
We show in Lemma~\ref{lemmabrzd} below that the elements necessarily are the ``natural'' ones. Before we recall some basic properties of $\Z^d$, where, as usual, for $F\in\R^{d\times d}$, $A\subseteq\R^d$, $p\in\R^d$, we set $p+FA:=\{p+Fa: a\in A\}$.
\begin{remark}\label{remarkface}
	The following hold.
	\begin{enumerate}[label=\roman*)]
		\item\label{remarkfaced} Let $R\in SO(\RR^d)$ and let $\epsilon\in (0,\infty)$. Then $\dist(x,\epsilon R \Z^d)\le \epsilon\sqrt d/2$ for any $x\in\RR^d$.
		\item\label{remarkfacev} If $v\subseteq\Z^d$, $\#v=d+1$, then
		either $v$ is contained in a $(d-1)$-dimensional affine subspace, or
		\begin{equation}\notag
			\LL^d(\conv v)\ge \frac{1}{d!}.
		\end{equation}
	\item\label{remarkfaces} If $w\subseteq\Z^d$, $\#w=d$, then
	either $w$ is contained in a $(d-2)$-dimensional affine subspace, or
	\begin{equation}\label{eqwzd}
		\calH^{d-1}(\conv w)\ge \frac{1}{(d-1)!}.
	\end{equation}
	\end{enumerate}
\end{remark}
\begin{proof}
To prove the first item, we can change coordinates to assume that $R=\rm Id$, and then, by scaling, we see that we can assume $\epsilon=1$. For each $i=1,\dots,d$ we select $z_i\in\Z$ with $|x_i-z_i|\le\frac12$, so that $z\in\Z^d$ and $$|x-z|=\Big({\sum\nolimits_{i=1}^d (x_i-z_i)^2}\Big)^{1/2}\le \sqrt d/2.$$

For the second one, by translation we can assume $0\in v$. The volume of the simplex $\conv v$ is given by $1/d!$ times the absolute value of the determinant of the matrix whose columns are the 
 vectors of $v\setminus\{0\}$. 
As each component of each vector is integer, the determinant is an integer. Hence it is either 0, or at least 1.

The proof of the third item is similar. Again, assume $0\in w$. At least one $e_i$ is not contained in the linear space generated by $w$. We apply the first assertion to $v:=w\cup\{e_i\}$, and obtain that the volume of $T:=\conv v$ is either zero or at least $1/d!$. Since the volume of $T$ is also given by $1/d$ times the area of $\conv w$ times the distance of $e_i$ to the space generated by $w$, which is at most 1 since $0\in w$, we obtain~\eqref{eqwzd}.
 \end{proof}

\begin{lemma}\label{lemmabrzd}
 Let $(V,E)$ be a triangulation of $\R^d$ with the Delaunay property and let $B_r(q)$ be a ball such that 
 $ V\cap B_r(q)= \eps R \Z^d\cap B_r(q)$, 
 for some $\eps>0$ and $R\in\SO(\RR^d)$.
 If $e\in E$ is such that
 $e\cap B_{r-\sqrt d \eps}(q)\ne\emptyset$,
then there is a unique $y\in \eps R (\Z+\frac12)^d$ such that
 $v_e\subseteq y+\eps R \{-\frac12,\frac12\}^d$, characterized by $v_e\subseteq \partial B_{\sqrt{d}/2}(y)$. 
\end{lemma}
We remark that  the assumption
$e\cap B_{r-\sqrt d \eps}(q)\ne\emptyset$ implies $r>\sqrt d \eps$.
\begin{proof}
By scaling and a change of coordinates it suffices to consider the case $\eps=1$, $R=\mathrm{Id}$.
Let $e$ be as in the statement, and let $B_\rho(y)$ be such that $v_e\subseteq\partial B_\rho(y)$. By the Delaunay property, using also the assumption in force here, \begin{equation}\label{csdacas}
	B_\rho(y)\cap \Z^d\cap B_r(q)\subseteq B_\rho(y)\cap V=\emptyset;
\end{equation}
by $e\cap B_{r-\sqrt d}(q)\ne\emptyset$ and $e\subseteq\overline B_\rho(y)$
we have
\begin{equation}\label{vcdsona}
	|q-y|<r-\sqrt d+\rho\qquad \text{(and $r>\sqrt d$)}.
\end{equation}
We want to show now that $\rho=\sqrt{d}/2$. 

First, we assume (by contradiction) that 
$\rho>\sqrt d/2$. We show that this possibility cannot occur. We define $\rho':=\min\{\rho,r, (r+\rho-|q-y|)/2\}$.
Condition \eqref{vcdsona} implies $\rho'>\sqrt d/2$ and the definition of $\rho'$ gives
$$|q-y|\le- 2\rho'+r+\rho=( r-\rho')+(\rho-\rho'),$$ so that
there exists $y'\in \overline B_{r-\rho'}(q)\cap \overline  B_{\rho-\rho'}(y)$ (we adopt the convention that $\overline B_0(x)=\{x\}$). The point $y'$ obeys then
$B_{\rho'}(y')\subseteq B_r(q)\cap B_\rho(y)$
%
and therefore, recalling \eqref{csdacas},
$B_{\rho'}(y')\cap \Z^d=\emptyset$, which contradicts $\rho'>\sqrt d/2$ (Remark \ref{remarkface}(\ref{remarkfaced}).   
 
Hence $\rho\le \sqrt d/2$, so that, using also \eqref{vcdsona}, $\overline B_\rho(y)\subseteq B_r(q)$, and therefore, recalling \eqref{csdacas}, $B_\rho(y)\cap\Z^d=\emptyset$
 and 
$v_e\subseteq\Z^d$.
We define $z\in\Z^d$ by choosing for each $i$ a component $z_i\in\Z$ which minimizes
$|z_i-y_i|$, notice that $|z_i-y_i|\le 1/2$. As $B_\rho(y)\cap\Z^d=\emptyset$, we have
$|z-y|\ge\rho$.
By minimality of $z_i$, for any $x\in v_e\subseteq\Z^d$ and any $i$ we have 
$|x_i-y_i|\ge |z_i-y_i|$, which by $x\in \partial B_\rho(y)$ implies $\rho=|x-y|\ge|z-y|\ge\rho$. Therefore, equality holds  throughout and
\begin{equation}\notag
 \rho=|x-y|=|z-y| \text{ and } |x_i-y_i|=|z_i-y_i| \qquad\text{for every $i\in\{1,\ldots, d\}$ and $x\in v_e$}.
\end{equation}
Assume that there exists $i$ with $|z_i-y_i|<\frac12$, so that $|z_i-x_i|<1$ for all $x\in v_e$. As $x_i, z_i\in\Z$, this implies $x_i=z_i$ for all   $x\in v_e$, hence $v_e$ is contained in a $(d-1)$-dimensional subspace of $\R^d$. As $e$ is non degenerate (i.e.\ has non empty interior), this is impossible, hence $|z_i-y_i|=\frac12$ for all $i$. We conclude that $\rho=\sqrt d/2$ and then
$v_e\subseteq y+\{-\frac12,\frac12\}^d$, which also implies the membership of $y$ to $(\mathbb{Z}+1/2)^d$ by $v_e\subseteq\Z^d$.
\end{proof}

\subsection{Construction of the triangulation}
\label{secconstrtriang}

\newcommand\Cgrid{C_\mathrm G}
\newcommand\Cface{C_\mathrm F}
We write $Q_\ell(x):=x+(-\ell/2,\ell/2)^d$ and $Q_\ell:=Q_\ell(0)$. Notice the factor $1/2$, i.e.\ $\ell$ is the length of the edge of the open cube $Q_\ell(x)$.

Aim of this section is to prove the following (see Figure~\ref{fig1} for an illustration):
\begin{theorem}\label{theogridrot}
For any $d\ge 2$ there is $\Cgrid=\Cgrid(d)$ with the following property.

 Let $0<\eps<\delta$ with $\delta \ge \Cgrid \eps$, and let $R:\delta\Z^d\to\SO(\RR^d)$. Then there is 
 a triangulation $(V,E)$ of $\R^d$, in the sense of Definition~\ref{deftriang}, with the following properties:
 \begin{enumerate}[label=\roman*)]
  \item\label{theogridrotreg} Regularity:
The triangulation has the Delaunay property
(property~\ref{deftriangdel}), 
is $\Cgrid$-non degenerate (property~\ref{deftriangreg}), and is $(\Cgrid,\eps)$-uniform (property~\ref{deftriangunif}).
  \item\label{theogridrotor} Orientation: for each $z\in \delta \Z^d$ 
  one has $V\cap Q_{\delta-\Cgrid\eps}(z)=
  \eps R(z)\Z^d\cap Q_{\delta-\Cgrid\eps}(z)$.
\end{enumerate} 
\end{theorem}

\begin{figure}
\begin{center}
 \includegraphics[width=7cm]{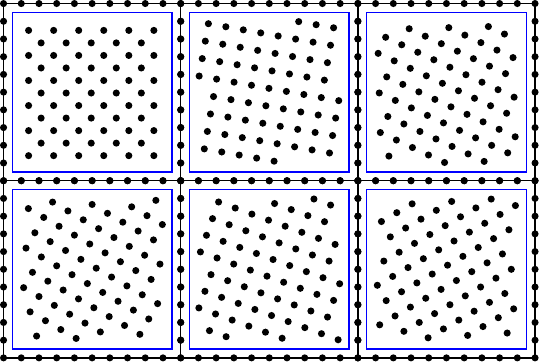} 
\end{center}
 \caption{Sketch of the set of vertices $V$ built in Theorem~\ref{theogridrot}. The blue squares indicate the irregular regions where $\Vmid$ is used.}
 \label{fig1}
\end{figure}

We start by proving 
that in a single cube we can construct a set of vertices $V$ which coincides with $\eps\Z^d$ on the boundary, with a rotation of the same lattice inside, and which is uniform and non-degenerate, in a sense made precise in the statement below. This will then be used to prove Theorem~\ref{theogridrot}.
\begin{lemma}\label{lemmaboundarylayer}
Let $z\in \R^d$, $\eps>0$, $R\in \SO(\RR^d)$, $M\in\N$ with $M\ge 6+2 d$.
Then there is $V\subseteq\R^d$
with the following properties:
\begin{enumerate}[label=\rm\roman*)]
 \item\label{lemmaboundarylayorient} Orientation: $V\setminus Q_{M\eps}(z)=\eps\Z^d\setminus Q_{M\eps}(z)$ and $V\cap Q_{(M-2)\eps}(z)=R\eps\Z^d\cap Q_{(M-2)\eps}(z)$;
 \item\label{lemmaboundarylayerunif} $(2d,\epsilon)$-uniformity: 
 for any $q\in\R^d$ we have 
 $B_{2d\eps}(q)\cap V\ne\emptyset$; for any
 $x\ne y\in V$ we have $|x-y|\ge\eps/(2d)$;
 \item\label{lemmaboundarylayernond} Non-degeneracy:
 There is $C'=C'(d)$ such that if
 $v\subseteq V$, $\#v=d+1$, $v$ is not contained
 in a $(d-1)$-dimensional affine subspace, and there is
 a ball $B_r(y)$ with 
 $v\subseteq \partial B_{r}(y)$, 
 $B_r(y)\cap V=\emptyset$,
 then $\LL^d(\conv v)\ge \eps^d/C'$.
 \end{enumerate}
\end{lemma}
\begin{figure}
\begin{center}
 \includegraphics[width=7cm]{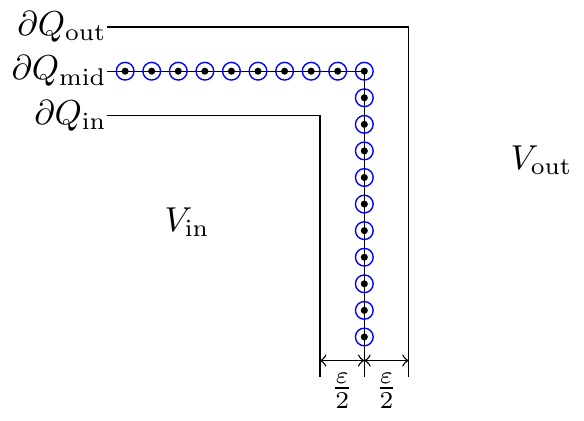}
\end{center}
\caption{Sketch of the boundary region as considered in Lemma~\ref{lemmaboundarylayer}.}
 \label{fig2}
\end{figure}

\begin{proof}We divide the proof in several steps.
	\medskip
	\\\textbf{Step 1}: general setting.
To simplify notation we denote by
$\Qout:=Q_{M\eps}(z)$ the outer cube, by $\Qin:=Q_{(M-2)\eps}(z)$ the inner cube, and by $\Qmid:=Q_{(M-1)\eps}(z)$ the intermediate one (see Figure~\ref{fig2}).
 We set $\Vout:=
 \eps\Z^d\setminus \Qout$; $\Vin:= R\eps\Z^d\cap\overline\Qin$, and shall construct below a finite set
 $\Vmid\subseteq Q_{(M-\frac12)\eps}(z)\setminus Q_{(M-\frac32)\eps}(z)$
 such that $$V:=\Vin\cup\Vout\cup\Vmid$$ has the desired properties.
The property~\ref{lemmaboundarylayorient} is true for any choice of $\Vmid$. Next we deal with~\ref{lemmaboundarylayerunif}, and leave the more delicate treatment of~\ref{lemmaboundarylayernond} at the end.
 
We show that for any $q\in \R^d$ one has $B_{2d\eps}(q)\cap (\Vin\cup \Vout)\ne \emptyset$.
Consider first the case $q\in \Qmid$. 
Let $q'$ be the point of $\overline Q_{(M-2-\sqrt d)\eps}(z)$ closest to $q$. This implies\begin{equation}\label{csadascs}
	 |q-q'|
\le\frac12\sqrt d(1+\sqrt d)\eps 
\end{equation}and
$B_{\sqrt d\eps/2}(q')\subseteq \Qin$. 
By  Remark~\ref{remarkface}, we can take $p\in R\eps \Z^d\cap 
\overline B_{\sqrt d\eps/2}(q')\subseteq\Vin$. 
Since by \eqref{csadascs}
$$
2 d\eps > |q-q'|+\sqrt d\eps/2
$$ we have
$p\in \overline B_{\sqrt d\eps/2}(q')\subseteq B_{2d\eps}(q)$, and the first assertion in~\ref{lemmaboundarylayerunif} is proved  in this case.
In the case $q\not\in\Qmid$ we argue similarly, projecting onto $\R^d\setminus Q_{(M+\sqrt d)\eps}(z)$,  with $\R^d\setminus \Qout$ instead of $\overlineQin$.
Therefore the first assertion in~\ref{lemmaboundarylayerunif} is true for any choice of $\Vmid$.

 
 It remains to choose $\Vmid$ so that the property $|x-y|\ge \eps/(2d)$ for all $x\ne y\in V$ (i.e.\ the second assertion in~\ref{lemmaboundarylayerunif}) is preserved, and~\ref{lemmaboundarylayernond}  holds. 
 In order to understand the strategy (cf.\ \ref{lemmaboundarylayernond}), consider a set 
 $v$ and a ball $B_r(y)$ such that
 \begin{equation}\label{eqvproppf}
\text{$v\subseteq V$ with $\#v=d+1$, 
 $v\subseteq\partial B_r(y)$, $V\cap B_r(y)=\emptyset$.   }
 \end{equation}
The construction strategy of $\Vmid$ then will ensure that:
 \begin{enumerate}[label=(\alph*)]
  \item\label{enumpropina} 
sets  $v$ as in~\eqref{eqvproppf} cannot contain elements of both $\Vin$ and $\Vout$;
  \item\label{enumpropinb} for any choice of $v$ as in~\eqref{eqvproppf}, with additionally $v\subseteq \Vin\cup \Vmid$
  or $v\subseteq \Vout\cup \Vmid$,
  is either contained in a $(d-1)$-dimensional affine subspace or obeys $\LL^d(\conv v)\ge \eps^d/C'$.
 \end{enumerate}
\medskip
\textbf{Step 2}: construction of $U_\eps$. We show here that there is a finite set $U_\eps\subseteq \partial\Qmid$ such that if the set $\Vmid$ is constructed picking exactly one point  $z$ of each $B_{\eps/(4d)}(u)$, for $u\in U_\eps$, then~\ref{enumpropina} 
and the second assertion in~\ref{lemmaboundarylayerunif} hold. The specific choice of the points $z$
will be done in Step 3 to ensure~\ref{enumpropinb} of (and hence~\ref{lemmaboundarylayernond}, by~\ref{enumpropina}).
 

We let $U_\eps:=\partial\Qmid\cap(\frac{1}{d}\eps\Z^d+p)$, where $p:=z-\frac{M-1}2\eps\sum_i e_i$ is a vertex of $\Qmid$. 
The shift $p$ is chosen so that the set is nonempty; we recall that $\Qmid$ is a cube of side length $(M-1)\eps\in\eps\Z$, but the centre $z$ is a generic point in $\R^d$.

Assume now that $\Vmid$ is chosen so that
it contains exactly one point of each $ B_{\eps/(4d)}(u)$, for $u\in U_\eps$.
We claim that then $V$ satisfies also the second assertion in~\ref{lemmaboundarylayerunif}. Let indeed $x,y\in V$,
$x\ne y$. 
If both are in $\Vin$, or both in $\Vout$, then $|x-y|\ge\eps$. 
If both are in $\Vmid$, then 
there are $u_x\ne u_y\in U_\eps$ with 
$|u_x-x|+|u_y-y|\le \eps/(2d)$. 
As
$u_x-u_y\in 
\frac{1}{d}\eps\Z^d\setminus\{0\}$, we obtain
$$|x-y|\ge |u_x-u_y|-|u_x-x|-|u_y-y|\ge 
\eps/(2d).$$
In the other cases, we use $$\dist(\Vout,\Vmid)\ge \dist(\partial\Qout,\partial\Qmid)-\eps/(4d)=\epsilon/2-\epsilon/(4 d)\ge \epsilon/4$$
and similarly $\dist(\Vin,\Vmid)\ge\eps/4$ to conclude.
This proves the second assertion in~\ref{lemmaboundarylayerunif}.

We finally check that~\ref{enumpropina} holds.
Let $v\subseteq V$ be as in~\eqref{eqvproppf}. 
Assume by contradiction that $v$
contains elements of both $\Vin$ and $\Vout$, then the sphere $\partial B_r(y)$ 
intersects both $\partial \Qout$ and $\partial \Qin$. We show that there exists $x'\in \partial \Qmid$ such that
$B_{\eps/2}(x')\subseteq B_r(y)$. 
Assume first $y\in \Qmid$. Let $y'\in  \partial B_r(y)\cap\partial\Qout$, and choose $x'\in[y,y']\cap \partial\Qmid$. Then
$|x'-y'|\ge\eps/2$, so that
$$|x'-y|=|y-y'|-|x'-y'|\le  r-\eps/2$$ 
and
$B_{\eps/2}(x')\subseteq B_r(y)$. If instead $y\not\in \Qmid$, we select $y'\in \partial B_r(y)\cap \partial\Qin$, and proceed analogously. 
Let $x$ be the point in $U_\eps$ closest to $x'$. 
As every component $x_i$ is the element of
$\frac{1}{d}\eps\Z+p_i$ 
 closest 
to $x_i'$,  we have
$|x-x'|\le \sqrt d\eps/(2d)=\eps/(2\sqrt d)$. 
As $\frac12> \frac1{4d}+\frac{1}{2\sqrt d}$, 
we obtain
$B_{\eps/(4d)}(x)\subseteq B_{\eps/2}(x')\subseteq B_r(y)$.
As $x\in U_\eps$, there is a point of $\Vmid$ in $B_{\eps/(4d)}(x)$,
which contradicts the condition $V\cap B_r(y)=\emptyset$ stated in~\eqref{eqvproppf}.
Therefore this cannot happen, and hence~\ref{enumpropina} holds.
\medskip
\\\textbf{Step 3:} choice of the elements of $\Vmid$.
We write 
$\{u_1,\ldots, u_J\}:=U_\eps$ and
iteratively for every $j$ 
 pick a point  $z_j\in B_{\eps/(4d)}(u_j)$ which ensures~\ref{enumpropinb}. 
 We collect in $\Vmid^j:=\{z_1, \ldots, z_j\}$ the points chosen in the first $j$ steps, and at the end we will use $\Vmid:=\Vmid^J$. Fix 
 \begin{equation}\label{eq:luigi81}
\ell:= 1+2d,
 \end{equation} the reason for this specific choice will be clear later.

An admissible 
set of vertices at stage $j$ is a set $v$ with $\#v=d+1$ such that
there is
 $q\in \partial \Qmid$ with $v\subseteq B_{\ell\eps}(q)$,  $\LL^d(\conv v)>0$, and either $v\subseteq \Vmid^j\cup \Vin$
 or $v\subseteq \Vmid^j\cup \Vout$.

An admissible face at stage $j$ is a set $w$ with $\#w=d$ such that
there is
 $q\in \partial \Qmid$ with $w\subseteq B_{\ell\eps}(q)$, 
 $\calH^{d-1}(\conv w)>0$, and either $w\subseteq \Vmid^j\cup \Vin$
 or $w\subseteq \Vmid^j\cup \Vout$.
 We denote by $N_w:=\#(w\cap \Vmid^j)$  the number of items of $w$ in $\Vmid^j$, clearly $N_w\le d$.
  
We intend to 
show that there are $\alpha,\,\beta,\,\gamma,\,{C_F}>0$ (depending only on $d$) such that we can 
choose $z_j\in B_{\eps/(4d)}(u_j)$ iteratively with the following two properties:
\begin{enumerate}[label=\roman*)]
\item If $v$ is an admissible set of vertices at stage $j$, then
\begin{equation}\label{eqestldv}
 \LL^d(\conv v) \ge \beta\eps^{d}.
\end{equation}
 \item If $w$ is an admissible face at stage $j$, then
\begin{equation}\label{eqesthdm1w}
 \calH^{d-1}(\conv w) \ge \frac{\alpha^{N_w}}{\Cface}\eps^{d-1}.
\end{equation}
\end{enumerate}
The key to the choice of $z_j$, which eventually leads to~\eqref{eqestldv} at stage $j$ building upon \eqref{eqesthdm1w} at stage $j-1$, is the following geometric observation.
If $v$ is an admissible set of vertices at stage $j$, 
and it contains the point $z_j$, then $w:=v\setminus\{z_j\}$ is  an admissible face at stage $j-1$ and for any $q\in w$ we have
\begin{equation}\label{eqvolvw}
 \LL^d(\conv v ) = \frac1d|(z_j-q)\cdot\nu_w| \calH^{d-1}(\conv w)
\end{equation}
where $\nu_w$ is a unit normal to the affine space generated by $w$. The factor $\calH^{d-1}(\conv w)$ will be estimated via~\eqref{eqesthdm1w} at stage $j-1$, the choice of $z_j$ needs to ensure that the first factor is not too small, for any possible choice of $w$.

Now we start choosing $z_1,\ldots,z_J$. As stated before, we proceed by iteration. Assume that we have already chosen $z_1,\ldots,z_{j-1}$, we want to choose $z_j$ (if $j=1$ we use $\Vmid^0=\emptyset$). Let $w$ be an admissible face at stage $j-1$ such that $w\subseteq B_{(2\ell+1/(4d))\eps}(u_j)$.  If no such face exists, choose $z_j\defeq u_j$.
Since no two points in $V$ are at distance smaller than $\eps/(2d)$ (by~\ref{lemmaboundarylayerunif}), the number of possible choices of $w$ is bounded by a number $K$ which depends only on $d$.
Let $w_1, \ldots, w_K$ be these possible choices.
We choose  $z_j$ such that
\begin{equation}\label{eqchoicezj}
 |(z_j-p_k)\cdot\nu_{w_k}|\ge \gamma\eps
\end{equation}
for all $k=1,\ldots, K$ and an arbitrary choice of $p_k\in w_k$ (the condition does not depend on the choice of $p_k$, as $\nu_{w_k}$ is orthogonal to $p_k-p_k'$ for any $p_k$, $p_k'\in w_k$). We show now why we can choose such $z_j$.
We observe that 
\begin{equation}\notag
 \LL^d\big(\{z\in B_{\eps/(4d)}(u_j): 
 |(z-p_k)\cdot\nu_{w_k}|< \gamma\eps\}\big) \le 
 2\gamma\eps \left(\frac{\eps}{2d}\right)^{d-1}
 =\gamma 2^{2-d}d^{1-d}\eps^d
\end{equation}
and thus the total volume of these sets is controlled by
$ K \gamma 2^{2-d}d^{1-d}\eps^d$. Then we choose $\gamma$ such that this expression equals 
$\frac 12\LL^d(B_{\eps/(4d)}(u_j))$ and hence we have a suitable $z_j$. Continuing in this way, we have thus constructed $\Vmid^J$.

It remains to show by induction that the points we constructed have the properties~\eqref{eqestldv}
and~\eqref{eqesthdm1w}.
Assume first $j=0$, and recall $\Vmid^0=\emptyset$, so that $N_w=0$.
By Remark~\ref{remarkface},~\eqref{eqestldv}
and~\eqref{eqesthdm1w} hold  provided $\Cface\ge (d-1)!$
and $\beta\le 1/d!$.
 Assume now that~\eqref{eqestldv} 
 and~\eqref{eqesthdm1w} 
 hold at stage $j-1$, we are going to prove that they hold also at stage $j$. 

Let $v$ be an admissible set of vertices at stage $j$.
 If $z_j\not\in v$, then $v$ was already admissible at stage $j-1$, hence~\eqref{eqestldv} holds.  Then we assume that $z_j\in v$, so that  $w:=v\setminus\{z_j\}$ is an admissible face at stage $j-1$ and 
 $v\subseteq B_{\ell\epsilon}(q)\subseteq B_{2\ell\eps}(z_j)\subseteq B_{(2\ell+1/(4d))\eps}(u_j)$, where $q\in\partial \Qmid$ is given by the admissibility of $v$.
 In particular, 
 $w\subseteq B_{(2\ell+1/(4d))\eps}(u_j)$, so that~\eqref{eqchoicezj}  holds for $w$ in place of $w_k$.
By~\eqref{eqesthdm1w} at stage $j-1$,~\eqref{eqvolvw},~\eqref{eqchoicezj} 
and $N_w\le d$
we have, provided $\alpha\le1$,
\begin{equation}\notag
 \LL^d(\conv v )= \frac{1}{d} |(z_j-p)\cdot\nu_w| \calH^{d-1}(\conv w)
 \ge \frac{\gamma \alpha^{d}}{\Cface d} \eps^d
\end{equation}
for any $p\in w$,
so that setting $\beta:= \min\{\gamma \alpha^{d} /(\Cface d),1/d!\}$ we obtain~\eqref{eqestldv}.

Let $w$ be an admissible face at stage $j$. As above, by the inductive assumption it suffices to consider the case $z_j\in w$.
Assume $w\subseteq\Vmid^j\cup \Vin$, the other case is analogous and will not be treated. Being $w$ admissible, $w\subseteq B_{\ell\epsilon}(q)$, for some $q\in\partial\Qmid$. 
Let $q'$ be the point of $\partial Q_{(M-4-\sqrt d)\eps}(z)$ closest to $q$, so that $|q-q'|\le \sqrt d (3+\sqrt d)\eps/2$, and choose $p_*\in \eps R \Z^d\cap \overline B_{\eps\sqrt d/2}(q')\subseteq \overline Q_{(M-4)\eps}(z)
$ (Remark~\ref{remarkface}).
By the choice of $\ell$  made in \eqref{eq:luigi81}, we get
$$|p_*-q|\le |p_*-q'|+|q'-q|\le (\sqrt d+3\sqrt d + d )\eps/2< (\ell-1)\eps.$$
Then the $2d$ points $p_*\pm \eps Re_i$ are all in 
$B_{\ell\eps}(q)\cap \Vin$, and at least one of them is not in the affine space generated by $w\setminus \{z_j\}$. Denote it by $p$, and set
\begin{equation}\notag
 \hat w:= \bigl(w \setminus \{z_j\}\bigr)\cup \{p\}.
\end{equation}
Then $\hat w$ is an admissible face at stage $j-1$, with $N_{\hat w}=N_{w}-1$
and $\calH^{d-1}(\conv \hat w)\ne0$, so that~\eqref{eqesthdm1w} holds for $\hat w$. 
Further, $\hat w\subseteq B_{\ell\eps}(q)\subseteq B_{2\ell\eps}(z_j)\subseteq B_{(2\ell+1/(4d))\eps}(u_j)$ implies that $\hat w$ is one of the faces $w_1,\ldots, w_K$ considered  for~\eqref{eqchoicezj}, so that the choice of $z_j$ implies that~\eqref{eqchoicezj} holds for $\hat w$.

We compute the volume of the simplex with vertices in $\hat w\cup\{z_j\}=w\cup\{p\}$ in two different ways:
\begin{equation}\notag
 |(z_j-p)\cdot \nu_{\hat w}| \calH^{d-1}(\conv \hat w)=
 |(z_j-p)\cdot \nu_{w}| \calH^{d-1}(\conv w).
\end{equation}
By~\eqref{eqchoicezj} and~\eqref{eqesthdm1w} for $\hat w$, recalling that $z_j,p\in B_{\ell\eps}(q)$ implies
$|z_j-p|\le 2\ell\eps$, we obtain
\begin{equation}\notag
 \calH^{d-1}(\conv w) \ge \frac{1}{2\ell\eps}|(z_j-p)\cdot \nu_{\hat w}| \calH^{d-1}(\conv \hat w)\ge
  \frac{\gamma}{2\ell}\alpha^{N_{\hat w}}\eps^{d-1}/\Cface
\end{equation}
which concludes the proof of~\eqref{eqesthdm1w} with $\alpha:=\min\{1, \gamma/(2\ell)\}$. 
\end{proof}

At this point we conclude the proof of Theorem~\ref{theogridrot}.
\begin{proof}[Proof of Theorem~\ref{theogridrot}]
 Set $$\ell:=2 d\qquad\text{and}\qquad M:=\lfloor\delta/\eps\rfloor-4\ell,$$
 so that $Q_{M\eps}\subseteq Q_\delta$, with
\begin{equation}\label{cvsdas}
	 \dist(Q_{M\eps},\partial Q_\delta)\ge 2\ell\eps.
\end{equation}
 We first select a background lattice,
 \begin{equation}\notag
  V^0:= \eps \Z^d \setminus \bigcup_{z\in \delta\Z^d} 
  Q_{M\eps}(z).
 \end{equation}
For each $z\in \delta\Z^d$, if $\Cgrid\ge 7+2d+4\ell$ we can use (by $M\ge C_G-1-4\ell$)  
 Lemma~\ref{lemmaboundarylayer} to obtain a set $V_z$ such that
 $V_z\cap Q_{(M-2)\eps}(z)=R(z)\eps \Z^d\cap Q_{(M-2)\eps}(z)$,
 and $V_z\setminus Q_{M\eps}(z)=
 \eps\Z^d\setminus Q_{M\eps}(z)$. 
 We then set 
 \begin{equation}\notag
  V:= V^0\cup \bigcup_{z\in\delta\Z^d} 
  (V_z\cap Q_\delta(z)) =
  V^0\cup \bigcup_{z\in\delta\Z^d} 
  (V_z\cap Q_{M\eps}(z)).
 \end{equation}
 This set obviously has the orientation property
 stated in~\ref{theogridrotor}, provided that $\Cgrid\ge4\ell+3$.

 We show that for any $x\ne y\in V$, one has $|x-y|\ge \eps/\ell$. Indeed, if there is $z\in\delta\Z^d$ with $x,y\in V_z$ then item~\ref{lemmaboundarylayerunif} of Lemma~\ref{lemmaboundarylayer}
 implies $|x-y|\ge\eps/ \ell$. If $x,y\in V^0$ then $|x-y|\ge\eps$. We are left with the case
 $x\in Q_{M\eps}(z)$ and $y\in Q_{M\eps}(z')$ for some $z\ne z'\in\delta\Z^d$, which implies $|x-y|\ge 2\dist(Q_{M\eps},\partial Q_\delta)\ge 4\ell\eps\ge \epsilon/\ell$, by \eqref{cvsdas}. 
 
We next similarly show that for any $q\in\R^d$ one has $V\cap B_{\ell\eps}(q)\ne\emptyset$.
If 
there is $z\in\delta\Z^d$ such that 
$q\in Q_{(M+2\ell)\eps}(z)$ then $B_{\ell\eps}(q)\subseteq Q_\delta(z)$, and the required property follows from item~\ref{lemmaboundarylayerunif} of Lemma~\ref{lemmaboundarylayer}, {since
$V\supseteq V_z\cap Q_\delta(z)$.} If not, then 
$B_{\ell\eps}(q)$ does not intersect any $Q_{M\eps}(z)$, 
so that $B_{\ell\eps}(q)\cap V^0=
B_{\ell\eps}(q)\cap\eps\Z^d$, which is nonempty  by Remark~\ref{remarkface}.

This proves that the set $V$ is $(\ell,\eps)$-uniform, in the sense of Property~\ref{deftriangunif} of Definition~\ref{deftriang}. By 
Lemma~\ref{lemmaconstrE} there is a set  $E$ so that $(V,E)$ is a triangulation with the Delaunay property. 

It only remains to show that $(V,E)$ is non-degenerate. Let $e\in E$ be a simplex, and 
let $\partial B_r(q)\supseteq v_e$ be its circumscribed sphere. By the Delaunay property
$B_r(q)\cap V=\emptyset$, by the $(\ell,\epsilon)$-uniformity proven above this implies $r< \ell\eps$. 
If there is $z\in\delta\Z^d$ such that 
$q\in Q_{(M+2\ell)\eps}(z)$ 
then $v_e\subseteq V_z$, and item~\ref{lemmaboundarylayernond} of  Lemma~\ref{lemmaboundarylayer}
implies $\LL^d(e)\ge \eps^d/C'$. 
Otherwise $v_e\subseteq V^0\subseteq\eps\Z^d$, and since $\LL^d(e)>0$ by Remark~\ref{remarkface} we obtain 
$\LL^d(e)\ge \eps^d/d!$.
This concludes the proof, with $\Cgrid:=\max\{ 7+2d+4\ell,4\ell+3, C', d!\}$.
\end{proof}

\subsection{Proof of the main result}\label{sectproof}

We now recall how one can use a triangulation to define continuous, piecewise affine approximations.

\begin{lemma}\label{lemmap1interp}
	Let $(V,E)$ be a triangulation of $\R^d$. For any $w:V\to\R$ there is a unique $u\in C^0(\R^d)$ which coincides with $w$ on $V$ and is affine on each $e\in E$. 
	
	If the triangulation is $c_*$-non degenerate, and if moreover $w$ is obtained as the restriction to $V$ of a $C^2(\R^d)$ function that we still denote $w$, then the function $u$ obtained above obeys
	\begin{equation}\label{eqconstrdulip}
		\|\nabla u\|_{L^\infty(e)}\le C \|\nabla w\|_{L^\infty(e)}
	\end{equation}
	and
	\begin{equation}\label{eqconstrdudw}
		\|\nabla w-\nabla u\|_{L^\infty(e)}\le C\diam(e) \|\nabla^2 w\|_{L^\infty(e)}
	\end{equation}
	for all $e\in E$, with $C$ depending on $c_*$ and $d$.
\end{lemma}
\begin{proof}
	For each $e\in E$ one defines $u_e:e\to\R$ by $u_e=w$ on $v_e$ and as the affine interpolation in the rest of $e=\conv(v_e)$. 
	To prove existence of $u$ we only need to check that $u_e=u_{e'}$ on $e\cap e'$, for any pair $e\ne e'\in E$.  Assume $e\cap e'\ne\emptyset$. Then $e\cap e'=\conv(v_e\cap v_{e'})$. 
	As $u_e=u_{e'}$ on $v_e\cap v_{e'}$, and both are affine in 
	$\conv(v_e\cap v_{e'})$, they coincide on $e\cap e'$. This concludes the proof of the first assertion.
	
	To prove the two estimates, we 
	focus on an element $e\in E$  and let $G$ be the constant gradient of $u$ on $e$.
	For any pair  $x,\,y\in v_e$,
	\begin{equation}\label{eqestoneel}
		\begin{split}
			G(y-x)=u(y)-u(x)
			&=w(y)-w(x)=\int_0^1 \nabla w(x+t(y-x)) (y-x) dt,
		\end{split}
	\end{equation}
	which implies
	\begin{equation}\notag
		|G(y-x)|\le \|\nabla w\|_{L^\infty(e)} |y-x|.
	\end{equation}
	With~\eqref{eqnondegsimpl} we obtain~\eqref{eqconstrdulip}.
	
	To prove the last estimate, 
	we pick any $\xi\in e$ and rewrite~\eqref{eqestoneel} as
	\begin{equation*}\begin{split}
			(G-\nabla w(\xi))(y-x)
			&=\int_0^1 \left(\nabla w(x+t(y-x))-\nabla w(\xi)\right) (y-x) dt.
		\end{split}
	\end{equation*}
	By the mean-value theorem
	$|\nabla w(\eta)-\nabla w(\xi)|\le \diam(e)\|\nabla^2w\|_{L^\infty(e)}$ for any $\eta\in e$, so that
	\begin{equation*}
		|(G-\nabla w(\xi))(y-x)|\le \diam(e)\|\nabla ^2w\|_{L^\infty(e)} |y-x|.
	\end{equation*}
	With~\eqref{eqnondegsimpl} we obtain~\eqref{eqconstrdudw}.
\end{proof}

We are ready to prove our main result, Theorem~\ref{theodensity}.
\begin{proof}[Proof of Theorem~\ref{theodensity}]
	Before entering into the proof of the theorem, we stress that we are going to use the fact that for a piecewise affine function $u_j$,
	\begin{equation}\label{rankone}
		|\DIFF^2_1 u_j|=|\DIFF\nabla u_j|.
	\end{equation} This follows from the fact that $u_j$ is piecewise affine, hence the distributional derivative of $\DIFF\nabla u_j$ is only of jump type, so that the density of $\DIFF\nabla u_j$ with respect to $|\DIFF\nabla u_j|$ 
is a rank 1 matrix, and hence we can use item~\ref{zuppa5} of Proposition~\ref{zuppa} in conjunction with Proposition~\ref{hessiananddiffgrad}.

Fix two sequences $\delta_j\to0$, $\eps_j\to0$, with $\delta_j>0$, $\eps_j>0$, and $\eps_j/\delta_j\to0$. 
	For each $j$ and each $z\in\delta_j\Z^d$ we select a matrix $R_z\in \SO(\RR^d)$ such that $R_z^t\nabla ^2w(z)R_z$ is diagonal,
and let $(V_j,E_j)$ be the grid constructed in Theorem~\ref{theogridrot} with these parameters.
We define  $u_j$ as the piecewise affine interpolation of $w$, constructed as in Lemma~\ref{lemmap1interp}. 
This concludes the construction.
 
 In order to prove convergence and the energy bound, it suffices to work in a
large ball $B_r$, with $\Omega\subseteq B_{r/2}$.
For large $j$, we can assume $\Cgrid\eps_j\le \delta_j\le r/(2d)$.
Here and below $\Cgrid$ is the  (fixed) constant from 
Theorem~\ref{theogridrot}, we can assume $\Cgrid>2\sqrt d$. We use $C$ for a generic constant that depends only on $d$ (and $\Cgrid$) and may vary from line to line.
 By Lemma~\ref{lemmap1interp} one immediately obtains
 a uniform Lipschitz bound on $u_j$,
 \begin{equation}\notag
  \|\nabla u_j\|_{L^\infty(B_{2r})} \le  C \|\nabla w\|_{L^\infty(B_{3r})}.
 \end{equation} 
By the uniformity property of the grid, for any $x\in B_r$ and any $j$ there is $y\in V_j$ with $|x-y|\le \Cgrid \eps_j$, therefore
\begin{equation}\notag
 \|w-u_j\|_{L^\infty(B_r)} \le  \Cgrid  \eps_j (\|\nabla u_j\|_{L^\infty(B_{2r})}+\|\nabla w\|_{L^\infty(B_{2r})})\to0.
\end{equation}
This proves local uniform convergence.

Since $\nabla^2w$ is continuous, one has that
\begin{equation}\label{eqdefomegaj}
 \omega_\rho:=\sup \bigl\{|\nabla^2w(x)-\nabla^2w(y)|:
 x,\,y\in B_{2r},\, |x-y|\le \rho \sqrt d \bigr\}
\end{equation}
converges to zero as $\rho\to0$.

 The estimate of the energy is done separately in the interior of the cubes, where the grid is regular, and in the boundary regions. We start from the boundary, where the grid is irregular.
 As $\nabla w$ is continuous, equation~\eqref{eqconstrdudw} in Lemma~\ref{lemmap1interp} permits to estimate $ |[\nabla u_j]|$, the jump in $\nabla u_j$ across the boundary between two neighbouring elements $e$ and $e'$ which intersect $B_r$, and gives
\begin{equation}\notag
 |[\nabla u_j]|\le C \eps_j \|\nabla ^2w\|_{L^\infty(B_{2r})}
\qquad
 \text{in all $e$ with $e\cap B_r\ne\emptyset$},
\end{equation}
here we used also Remark \ref{remdiam}.
Using non-degeneracy and uniformity of the triangulation to control the volume of $e$, we obtain
  \begin{equation}\notag
   |\DIFF\nabla u_j|(\partial e)\le C \calH^{d-1}(\partial e) \max |[\nabla u_j]|(\partial e)\le C \LL^d(e) \|\nabla ^2w\|_{L^\infty(B_{2r})} 
  \end{equation}
  for all elements $e\in E_j$ with $e\subseteq B_{r}$.
Fix now $z\in \delta_j\Z^d$
such that $Q_{\delta_j}(z)\cap\Omega\ne\emptyset$. 
  Summing the previous condition
  over all elements $e\in E_j$ with $e\cap  \overline Q_{\delta_j}(z)
  \setminus Q_{\delta_j-4\Cgrid\eps_j}(z)\ne\emptyset$ leads to
  \begin{equation}\label{eqqdeltabrds}
  \begin{split}
   |\DIFF\nabla u_j|( \overline Q_{\delta_j}(z)\setminus Q_{\delta_j-4\Cgrid\eps_j}(z))
 &  \le C \LL^d(Q_{\delta_j+4\Cgrid\eps_j}(z)\setminus Q_{\delta_j- 8\Cgrid\eps_j}(z)) \|\nabla ^2w\|_{L^\infty(B_{2r})} \\
&  \le C((\delta_j+4\Cgrid\eps_j)^d-(\delta_j- 8\Cgrid\eps_j)^d)  \|\nabla^2w\|_{L^\infty(B_{2r})}\\&\le C \delta_j^{d-1}\eps_j\, \|\nabla^2w\|_{L^\infty(B_{2r})},
   \end{split}
  \end{equation}
provided $j$ is large enough, since $\eps_j\ll\delta_j$. Here we used that for every $e\in E_j$, $\diam(e)\le 2 C_G\epsilon_j$, being the triangulation $(V_j,E_j)$ $(C_G,\epsilon_j)$-uniform and with the Delaunay property.
 
We next estimate the energy inside $ Q_{\delta_j-3\Cgrid\eps_j}(z)$, for some $z\in \delta_j\Z^d\cap B_{r}$. 
Let $H_z:=\nabla^2w(z)$, and recall that $R_z$ was chosen so that $R_z^tH_zR_z=\diag(\lambda_1,\dots, \lambda_d)$ for some $\lambda\in\R^d$, which implies $|H_z|_1=\sum_{i=1}^d|\lambda_i|$, see items~\ref{zuppa1} and~\ref{zuppa2} of Proposition~\ref{zuppa}. In the next estimates we write briefly $\delta$ and $\eps$ for $\delta_j$ and $\eps_j$.

For any element $e\in E_j$ with
$e\cap Q_{\delta-2\Cgrid\eps}(z)\ne\emptyset$, we can select
 $p_e\in e\cap Q_{\delta-2\Cgrid\eps}(z)$. Then
$B_{\Cgrid\eps/2}(p_e){\subseteq Q_{\Cgrid\eps}(p_e)}\subseteq Q_{\delta-\Cgrid\eps}(z)$, so that the orientation property of Theorem~\ref{theogridrot} gives
$B_{\Cgrid\eps/2}(p_e)\cap V_j=B_{\Cgrid\eps/2}(p_e)\cap \eps R_z\Z^d$. Recalling $\Cgrid>2\sqrt d$,
by applying Lemma~\ref{lemmabrzd} 
with $q=p_e$, $r=\Cgrid\eps/2$, there exists
 $y\in \eps R_z(\Z+\frac12)^d$ such that $v_e\subseteq y+\eps R_z\{-\frac12,\frac12\}^d$.
Let $F_y:=\nabla w(y)$.
 For all $x\in v_e$, Taylor remainder term in integral form and~\eqref{eqdefomegaj}
 yield
 \begin{equation}\notag
  w(x)= w(y)+F_y(x-y)+\frac12 H_z(x-y)\cdot (x-y)
  +R(x)
 \end{equation}
 (this can be seen as the definition of $R(\,\cdot\,)$) with
 \begin{equation}\label{cdsascads}
  |R(x)|\le d\eps^2|\nabla^2 w(y)-H_z|+
  \int_0^1 |\nabla^2 w(x+t(y-x))-\nabla^2 w(y)| \, |y-x|^2 dt \le
  C \eps^2 \omega_\delta.
 \end{equation}
As $x-y=\sum_i \eps \gamma_i R_ze_i$, with $\gamma_i\in\{-\frac12,\frac12\}$, recalling that $R_z^tH_zR_z=\diag(\lambda_1,\ldots, \lambda_d)$ we have
\begin{equation}\notag
 H_z(x-y)\cdot (x-y)
 =\eps^2\sum_{i,\,k=1}^d\gamma_i\gamma_k e_i R_z^t H_z R_z e_k
 =\frac14 \eps^2\sum_{i=1}^d\lambda_i
 \end{equation}
which does not depend on the $\gamma_i$, and therefore is the same for all $x\in v_e$.
Hence
 \begin{equation*}\begin{split}
   w(x)
&=w(y)+F_y(x-y)+\frac18 \eps^2 \sum_{i=1}^d \lambda_i+R(x)\qquad
\text{for all }x\in v_e
.\end{split}
 \end{equation*}
 The function $u_j$ is affine on the element $e$, assume it has the form $u_j(\xi)=a_e+G_e\xi$
 for $\xi\in e$. As $u_j=w$ on $v_e$, for every pair $x,x'\in v_e$ we obtain
\begin{equation}\notag
 G_e(x-x')=u_j(x)-u_j(x')=w(x)-w(x')=F_y(x-x')+R(x)-R(x').
\end{equation}
Recalling that $e$ is a non-degenerate simplex by \eqref{eqnondegsimpl}, \eqref{cdsascads} and what just proved we obtain
\begin{equation}\label{eqgefst}
 |G_e-F_y|\le C\eps\omega_\delta.
\end{equation}

In summary, if $e\in E_j$ obeys
$e\cap Q_{\delta-2\Cgrid\eps}(z)\ne\emptyset$
then there exists $y_e\in \eps R_z(\Z+\frac12)^d$  with $v_e\subseteq y_e+\eps R_z\{-\frac12,\frac12\}^d$, and 
the vector $G_e:=\nabla {u_j}_{|e}$ obeys~\eqref{eqgefst}.

Consider now some $y\in \eps R_z(\Z+\frac12)^d$ such that
$(y+R_z Q_\eps)\cap Q_{\delta-4\Cgrid\eps}(z)\ne\emptyset$. If $e,\,e'$ are two elements with $v_e,\,v_{e'}\subseteq y+R_z \overline{Q}_\eps$, then  (by $\Cgrid>\sqrt d$) both intersect $Q_{\delta-2\Cgrid\eps}(z)$, so that
the above discussion applies and~\eqref{eqgefst} gives
 $|G_e-G_{e'}|\le C\eps\omega_\delta$, having used that the above discussion forces $y=y_e$ (since $y,\,y_e\in \eps R_z(\Z+\frac12)^d$ 
 and $y\neq y_e$ imply that $(y+R_z \overline{Q}_\eps)\cap (y_e+\eps R_z\{-\frac12,\frac12\}^d)\supseteq v_e$
 has at most dimension $d-1$) and analogously $y=y_{e'}$.
 In particular, those elements constitute 
 a decomposition of  $y+R_z Q_\eps$.
Arguing as before, summing over all pairs, 
\begin{equation}\label{eq:luigi7}
 |\DIFF\nabla u_j|(y+R_zQ_\eps)\le C \eps^{d-1} \max|G_e-G_{e'}|
 \le C \eps^d\omega_\delta.
\end{equation}
In order to estimate the contribution from the boundary of these cubes,
let $y'=y\pm \eps R_z e_i$ be the centre of one of the neighbouring small cubes. Since $\Cgrid>2\sqrt d$,
$y'+R_z Q_\eps\subseteq Q_{\delta-2\Cgrid\eps}(z)$, so that~\eqref{eqgefst} holds for any element $e''$ contained in
$y'+R_z \overline Q_\eps$ (with $e''$ in place of $e$ and $y'$ in place of $y$). 
As the common boundary has area $\eps^{d-1}$,
\begin{equation*}\begin{split}
 |\DIFF\nabla u_j|(\partial(y+R_zQ_\eps))
 &\le C \eps^d\omega_\delta
 + \sum_{y'\in y+R_z\eps\{\pm e_1,\dots, \pm e_d\}}
 \eps^{d-1} |F_y-F_{y'}| .
 \end{split}
\end{equation*}
As we did before, we represent $F_{y'}-F_y=\nabla w(y')-\nabla w(y)$ with {Taylor's theorem}
\begin{equation}\notag
 F_{y'}=F_y
 +H_z(y'-y)+R'(y',y) \qquad\text{and}\qquad |R'(y',y)|\le C \eps\omega_\delta
\end{equation}
(this can be seen as the definition of $R'(\,\cdot\,,\,\cdot\,)$) to obtain
\begin{equation}\label{eq:luigi88}\begin{split}
 |\DIFF\nabla u_j|(\partial(y+R_zQ_\eps))
 &\le C \eps^d\omega_\delta
 + \sum_{y'\in y+R_z\eps\{\pm e_1,\dots, \pm e_d\}}
 \eps^{d-1} |H_z(y'-y)|\\
 &= C \eps^d\omega_\delta
 + 2\eps^d  |H_z|_1
 \le C \eps^d\omega_\delta
 +2\int_{y+R_zQ_\eps} |\nabla^2 w|_1\dd\LL^d,
 \end{split}
\end{equation}
where we used that the $R_ze_i$ are eigenvectors of $H_z$ by the choice of $R_z$, the definition of the Schatten norm and in the final step~\eqref{eqdefomegaj}. 
Let $$
A_z:=\{y\in \eps R_z(\Z+\frac 12)^d: (y+R_zQ_\eps)\cap Q_{\delta-4\Cgrid\eps}(z)\ne\emptyset\}. 
$$
Summing over all $y\in A_z$, {taking into account \eqref{eq:luigi7} and \eqref{eq:luigi88}} and recalling that the boundaries between the cubes appear twice in the sum,
gives
\begin{equation}\notag
 |\DIFF\nabla u_j|(Q_{\delta-4\Cgrid\eps}(z))
 \le C \delta^d\omega_\delta
 +\int_{Q_\delta(z)} |\nabla^2w|_1\dd\LL^d
\end{equation}
and combining with~\eqref{eqqdeltabrds}
\begin{equation}\notag
 |\DIFF \nabla u_j|(\overline Q_{\delta}(z))
 \le C \delta^d\left(\omega_\delta+\frac{\eps}{\delta}\|\nabla^2w\|_{L^\infty(B_{2r})}\right)
 +\int_{Q_\delta(z)} |\nabla^2w|_1\dd\LL^d.
\end{equation}
Summing over all $z$ such that $Q_\delta(z)\cap \Omega\ne\emptyset$, and inserting back the indices $j$,
\begin{equation}\notag
 |\DIFF\nabla  u_j|(\Omega)
 \le C |(\Omega)_{\delta_j}| \left(\omega_{\delta_j}+\frac{\eps_j}{\delta_j}\|\nabla^2 w\|_{L^\infty(B_{2r})}\right)
 +\int_{(\Omega)_{\delta_j}} |\nabla^2w|_1 \dd\LL^d
\end{equation}
where $(\Omega)_\rho:=\{x\in\R^d: \dist(x,\Omega)\le \rho\sqrt d\}$. 
Taking the limit $j\to\infty$, and recalling that $\delta_j\to0$, $\omega_{\delta_j}\to0$ and $\eps_j/\delta_j\to0$, concludes the proof (recalling~\eqref{rankone}).
\end{proof}

\section{Extremality of cones}\label{sectcones}
In this section we consider functions of the kind \begin{equation}\label{defcone}
	f^\cone(x)\defeq(1-|x|)_+.
\end{equation}
It is clear that our forthcoming discussion will apply also to slightly different functions, e.g.\ $a(1-b|x-x_0|)_+$ for $a,b\in\RR$ with $b> 0$ and $x_0\in\RR^d$, but this will not make much difference, as one can reduce to the particular case of~\eqref{defcone} via a change of coordinates and a rescaling. Notice that, by Proposition~\ref{htvradial}, if $d\ge 2$,
\begin{equation}\label{dcass}
	|\DIFF_p^2 f^{\cone}|(B_r(0))=
	 d\omega_d\big((d-1)^{1/p-1}(r\wedge 1)^{d-1}+\chi_{(1,\infty)}(r)\big).
\end{equation}
Our aim is to investigate extremality of such kind of functions with respect to $p$-Hessian--Schatten seminorms, for $p\in[1,\infty]$. It turns out that these functions are extremal, and now we state our main result in this direction. Its proof is deferred to Section~\ref{proofmainextremal} and will follow  easily from the results of Section~\ref{convexity} and Section~\ref{extremalityradialing}, taking into account also Section~\ref{sectexplicit}.
\begin{thm}\label{coneexthm}
	Let $d\ge 2$ and let $p\in [1,\infty)$.
Let $f_1,\,f_2\in L^1_\rmloc(\RR^d)$ with bounded Hessian--Schatten variation in $\RR^d$ such that
$$
|\DIFF^2_p f_1|(\RR^d)=|\DIFF^2_p f_2|(\RR^d)=|\DIFF^2_p f^\cone|(\RR^d)
$$ 
and such that for some $\lambda\in (0,1)$, 
$$f^\cone=\lambda f_1+(1-\lambda) f_2.$$
Then $f_1$ and $f_2$ are equal to $f^\cone$, up to affine terms: there exist affine functions $L_1,L_2:\RR^d\rightarrow\RR$ such that $f_i=f^\cone+L_i$ for $i=1,\,2$.
\end{thm}
Notice that Theorem \ref{coneexthm} is stated only for $d\ge 2$. Indeed, for $d=1$,  it is easy to realize that $f^{\rm cone}$ is \emph{not} extremal, according to the meaning described in the statement of the theorem.
\smallskip

To simplify the notation, as in this section we are going to consider only balls centred at the origin, we will omit to write the centre of the ball, i.e.\ $B_r\defeq B_r(0)$.
Before going on, we recall that given $f\in L^1_\rmloc(\RR^d)$, we denote by 
$f^\rad$ the function given by Lemma~\ref{radialbetter}. As an explicit expression, notice that 
\begin{equation}\label{defrad}
	f^\rad(x)=\dashint_{\partial B_{|x|}} f(\sigma)\dd\HH^{d-1}(\sigma)\quad\text{for $\LL^d$-a.e.\ $x$.}
\end{equation}
Notice also that $f^\rad(x)=g(|x|)$ for $g(r)$ given by the right hand side of~\eqref{defrad} with $r$ in place of $|x|$.

\subsection{Convexity}\label{convexity}
We prove that if a function $f\in L^1_\rmloc(\RR^d)$ is such that $f^\rad=f^\cone$ and  such that $|\DIFF^2_p f|(\RR^d)=|\DIFF^2_p f^\cone|(\RR^d)$, then $f$ is the cone. The case $p=1$ is treated in Proposition~\ref{propconerad}, using the fact that the absolutely continuous part of $\DIFF\nabla f$ has a sign, which makes $f$ concave inside the unit ball. The case $p>1$ is treated in
Proposition~\ref{propconeradp},
using strict convexity of the $p$-Schatten norm to show that  the absolutely continuous part of $\DIFF\nabla f$ is a scalar multiple of  the absolutely continuous part of $\DIFF\nabla f^\cone$, and then scaling to reduce to the $p=1$ case.

First, we need a couple of lemmas. The first is an extension of a well known criterion to recognize convexity. 

\begin{lemma}\label{lemmaconvex1}
	Let $\Omega\subseteq\R^d$ be open and convex and let $f\in L^1_\rmloc(\Omega)$ with bounded Hessian--Schatten variation in $\Omega$. Assume that $\DIFF\nabla f\ge 0$ (as a measure with values in symmetric matrices). Then $f$ has a representative which is continuous and convex.
\end{lemma}
\begin{proof}
	The property of having a continuous representative is clearly local. 
	Since $\Omega$ is open and convex, a continuous function $g:\Omega\to\R$ is convex if and only if it is convex in a neighbourhood of any point. Therefore it suffices to prove the assertion in a neighbourhood of any point, so that we can assume $f\in W^{1,1}(\Omega)$ with $\nabla f\in\BV(\Omega;\RR^d)$, by Proposition~\ref{sobo} and Proposition~\ref{hessiananddiffgrad}.
	
	Let $x\in\Omega$, and pick $r>0$ such that $Q_{4r}(x)\subseteq\Omega$ (we write here $Q_\ell(y):=y+(-\ell,\ell)^n$).
	Fix a mollifier $\eta_\eps\in C^\infty_{\mathrm {c}}(B_\eps;[0,\infty))$, with $\eps\le r$, and define
	$f_\eps:=\eta_\eps\ast f\in C^\infty(Q_{3r}(x))$. 
	Then an immediate computation yields $\DIFF\nabla f_\eps=\eta_\eps\ast \DIFF\nabla f\ge0$ in $Q_{3r}(x)$, therefore $f_\eps$ is convex in $Q_{3r}(x)$. Further,
	$f_\eps\to f$ in $W^{1,1}(Q_{3r})$.
	It remains to  show that $f_\eps$ (possibly after passing to a subsequence) converges uniformly in $Q_{r}$, which implies the conclusion in $Q_r$ and therefore in a neighbourhood of any point of $\Omega$. 
	
	We prove now uniform convergence in $Q_r$, the argument is classical, see e.g.\ the proof of \cite[Theorem 7.6]{EG15}. Passing to a  subsequence, $f_{\eps_j}\to f$ pointwise almost everywhere. 
	Pick $\bar x\in Q_{r/2}(x)$ such that 
	the sequences $f_{\eps_j}(\bar x)$ and
	$f_{\eps_j}(y)$, for any vertex $y$ 
	of  $Q_{2 r}(\bar x)\subseteq Q_{3r}(x)$, 
	are bounded (as we can assume them to be convergent), and let $M=M_{\bar x,r}$ be the common bound.
	By convexity, $f_{\eps_j}\le M$ on $\bar Q_{2 r}(\bar x)$.
	To prove the uniform lower bound, we observe that
	for any $w\in Q_{2r}(\bar x)\setminus\{\bar x\}$ there is 
	$z\in \partial Q_{2r}(\bar x)$
	such that $\bar x$ is in the interior of the segment joining $w$ with $z$. As convexity implies monotonicity of the difference quotients,
	\begin{equation}\notag
		\frac{f_{\eps_j}(\bar x)-f_{\eps_j}(w)}
		{|\bar  x-w|} 
		\le \frac{f_{\eps_j}(z)-f_{\eps_j}(\bar x)}
		{|z-\bar x|} 
		\le \frac{2M}
		{2r} ,
	\end{equation}
	where in the last step we used 
	$|z-\bar x|\ge 2r$. 
	Since $f_{\eps_j}(\bar x)\ge -M$ and $|w-\bar x|\le 2r\sqrt d$ we have $f_{\eps_j}(w)\ge -(1+2\sqrt d)M$.
	Passing to the smaller cube $Q_r(x)$ and using again monotonicity of the difference quotients we obtain
	$\mathrm{Lip}(f_{\eps_j}; Q_{r}(x))\le C'M$  for all $j$, so that $f_{\eps_j}$ converges uniformly in $Q_r(x)$ to a continuous convex function, which coincides almost everywhere with $f$. This concludes the proof.
\end{proof}

The following lemma builds upon Lemma~\ref{lemmaconvex1} and gives an integral characterization of convexity, which is more manageable, and follows from the rigidity in the inequality $|\Tr A|\le |A|_1$.
\begin{lemma}\label{lemmaconvex}
	Let $\Omega\subseteq\R^d$ be open and let $f\in L^1_\rmloc(\Omega)$  with bounded Hessian--Schatten variation in $\Omega$. Then
	\begin{equation}\label{eqlemmaconvex}
		|\DIFF^2_1f|(\Omega)\ge |\mathrm{Tr} \DIFF\nabla f(\Omega)|.
	\end{equation}
Assume now that equality in~\eqref{eqlemmaconvex} holds. Then
\begin{itemize}[label=$\bullet$]
	\item either $|\DIFF^2_1f|(\Omega)=\mathrm{Tr} \DIFF\nabla f(\Omega)$ and then $f$ has a representative which is continuous and convex,
		\item or $|\DIFF^2_1f|(\Omega)=-\mathrm{Tr} \DIFF\nabla f(\Omega)$ and then $f$ has a representative which is continuous and concave. 
\end{itemize}
\end{lemma}
\begin{proof}
	We can assume that $\mathrm{Tr}\DIFF\nabla f(\Omega)\ge 0$, otherwise one replaces $f$ by $-f$. 

Let now $A\in\R^{d\times d}$ be a symmetric matrix and let $\lambda_1,\ldots,\lambda_d$ denote its eigenvalues. By  item~\ref{zuppa1} of Proposition~\ref{zuppa},
\begin{equation}\notag
	|A|_1=\sum_{i=1}^d|\lambda_i|\ge\sum_{i=1}^d\lambda_i= \Tr A
\end{equation}
and equality holds if and only if $\lambda_i\ge 0$ for all $i$, which is the same as $A\ge0$ as a symmetric matrix.

	By Proposition~\ref{hessiananddiffgrad} (in particular, 
	$|\DIFF^2_1f|\ll |\DIFF\nabla f|$ and $\mathrm{Tr} \DIFF\nabla f\ll |\DIFF\nabla f|$), 
	\begin{equation}\notag
		|\DIFF^2_1f|(\Omega)
		=\int_\Omega \bigg|\dv{\DIFF\nabla f}{|\DIFF\nabla f|}\bigg|_1 \dd|\DIFF\nabla f|
		\ge\int_\Omega \Tr \dv{\DIFF\nabla f}{|\DIFF\nabla f|} \dd|\DIFF\nabla f|= \Tr \DIFF\nabla f(\Omega),
	\end{equation}
	which proves the bound~\eqref{eqlemmaconvex}.
	If equality holds, then 
	\begin{equation}\notag
	\bigg|\dv{\DIFF\nabla f}{|\DIFF\nabla f|}\bigg|_1= \Tr \dv{\DIFF\nabla f}{|\DIFF\nabla f|}\qquad|\DIFF\nabla  f|\text{-a.e.}
	\end{equation} so that
	\begin{equation}\notag
		\dv{\DIFF\nabla f}{|\DIFF\nabla f|}\ge0\qquad|\DIFF\nabla  f|\text{-a.e.}
	\end{equation} which means that $\DIFF\nabla f\ge0$ as a matrix-valued measure, so that the conclusion then follows by Lemma~\ref{lemmaconvex1}.
\end{proof}

\subsection{Extremality with respect to spherical averaging}\label{extremalityradialing}
In this section, we consider only the case $d\ge 2$. This is because this is an auxiliary section for the proof of Theorem \ref{coneexthm}, which holds only for $d\ge 2$.
We start by doing some explicit computation involving the Hessian--Schatten total variation of $f^\cone$. First, by Proposition~\ref{hessiananddiffgrad}, $f^\cone\in W^{1,1}(\RR^d)$ with $\nabla f^\cone\in\BV(\RR^d;\RR^d)$, 
more precisely
$$
\nabla f^\cone(x)=-\chi_{B_1}(x)\frac{x}{|x|}.
$$
This computation is easily justified by locality, as $f^\cone$ is smooth on $B_1\setminus\{0\}$ and on $\RR^d\setminus\bar B_1$. Now we claim that
\begin{equation}\label{eqd2fcone}
	\DIFF\nabla f^\cone(x)=
	-\frac{|x|^2	{\rm Id}-x\otimes x}{|x|^3}\LL^d\mres B_1 +( x\otimes x)\HH^{d-1}\mres \partial B_1.
\end{equation}
Taking into account that $\DIFF\nabla f^\cone$ does not charge points, this formula is easily justified on $\RR^d\setminus \partial B_1$ by locality, as above. For what concerns the singular part, on $\partial B_1$, it is enough to use the representation formula for the singular part of differentials of vector valued functions of bounded variation, e.g.\ \cite{AFP00}, notice indeed that the unit outer normal to $\partial B_1$ is $x$ and that the jump of $\nabla f^\cone$ at $x\in\partial B_1$ is exactly $x$. 

Taking traces, we have that 
\begin{equation}\notag
	\Tr\DIFF\nabla f^\cone (x)= \frac{(1-d)}{|x|}\LL^d\mres B_1+\HH^{d-1}\mres \partial B_1,
\end{equation}
so that 
\begin{equation}\label{tracecone}
	\int_{B_r}\dd\Tr\DIFF\nabla f^\cone = - d\omega_d r^{d-1}\chi_{(0,1]}(r)
	\qquad\text{$\forall r>0.$}
\end{equation}

\smallskip

Recall that by Lemma~\ref{radialbetter}, $|\DIFF^2_p f^\rad|(\RR^d)\le |\DIFF^2_p f|(\RR^d)$. The next lemma states that this inequality is somehow rigid.
\begin{lem}
		Let $p\in[1,\infty]$. Let $f \in L^1_\rmloc(\RR^d)$ with bounded Hessian--Schatten variation and assume that  
	\begin{equation}\label{eqd21fconedm}
		|\DIFF^2_p f^\rad|(\R^d)=|\DIFF^2_p f|(\R^d).
	\end{equation}
Then, for every $r>0$ one has
	\begin{equation}\label{splitted}
\begin{split}
		|\DIFF_p^2 f|(B_r)=|\DIFF_p^2 f^\rad|(B_r)&,\	|\DIFF_p^2 f|(\partial B_r)=|\DIFF_p^2 f^\rad|(\partial B_r)\\&\text{ and }|\DIFF_p^2 f|(\RR^d\setminus\bar  B_r)=|\DIFF_p^2 f^\rad|(\RR^d\setminus\bar  B_r).
\end{split}
\end{equation}
\end{lem}
\begin{proof}
First notice that thanks to Lemma~\ref{radialbetter}, for any $\epsilon>0$,
\begin{align*}
		|\DIFF^2_p f^\rad|(B_r)\le 	|\DIFF^2_p f|(B_r)&,\ 	
		|\DIFF^2_p f^\rad|(B_{r+\epsilon}\setminus\bar{B}_{r-\epsilon})\le 	|\DIFF^2_p f|(B_{r+\epsilon}\setminus\bar{B}_{r-\epsilon})\\&\text{and } 	|\DIFF^2_p f^\rad|(\RR^d\setminus\bar B_r)\le 	|\DIFF^2_p f|(\RR^d\setminus\bar B_r)
\end{align*}
	so that, by regularity of measures, letting $\epsilon\searrow0$,
\begin{equation}\notag
	\begin{split}
		|\DIFF^2_p f^\rad|(B_r)\le 	|\DIFF^2_p f|(B_r)&,\	|\DIFF^2_p f^\rad|(\partial B_{r})\le 	|\DIFF^2_p f|(\partial B_{r})\\&\text{and } 	|\DIFF^2_p f^\rad|(\RR^d\setminus\bar B_r)\le 	|\DIFF^2_p f|(\RR^d\setminus\bar B_r).
\end{split}
\end{equation}
	Then we  can compute, by the inequalities above and exploiting~\eqref{eqd21fconedm},
	\begin{align*}
		|\DIFF^2_p f|(\RR^d)&=|\DIFF^2_p f^\rad|(\RR^d)=|\DIFF^2_p f^\rad|(B_r)+|\DIFF^2_p f^\rad|(\partial B_r)+|\DIFF^2_p f^\rad|(\RR^d\setminus\bar B_r)\\&\le|\DIFF^2_p f|(B_r)+|\DIFF^2_p f|(\partial B_r)+|\DIFF^2_p f|(\RR^d\setminus\bar B_r)=|\DIFF_p^2 f|(\RR^d),
	\end{align*}
	so that equality holds throughout and  therefore we obtain~\eqref{splitted}.
\end{proof}

Now we state and prove the main results of this section, splitting the case $p=1$ and the case $p\in(1,\infty)$. Recall that $|\DIFF^2_1 f^\cone|(\RR^d\setminus\bar B_1)=0$ according to \eqref{eqd2fcone}.
\begin{proposition}\label{propconerad}
	Let $f \in L^1_\rmloc(\RR^d)$ with bounded Hessian--Schatten variation and assume that 
	\begin{equation}\label{eqd21fconed}
	f^\rad=f^\cone\qquad\text{and}\qquad	|\DIFF^2_1f|(\R^d)=|\DIFF^2_1f^\cone|(\R^d).
	\end{equation}
	Then $f$ is equal to $f^\cone$ up to a linear term: there exists $\alpha\in\R^d$ such that
	\begin{equation}\notag
		f(x)=f^\cone(x)+\alpha\cdot x
		\qquad\text{ for a.e.\ $x\in\R^d$.}
	\end{equation} 
\end{proposition}
\begin{proof}
	Let $r>0$ and let $U\in SO(\RR^d)$. By Lemma~\ref{radialbetter}, $f_U\defeq f(U\,\cdot\,)$ has finite Hessian--Schatten total variation. Also, for any radial function $g\in C_\rmc^\infty(\RR^d)$ one has
	$$
	\int_{\RR^d}f_U\Delta g\dd\LL^d=\int_{\RR^d}f (\Delta g)_{U^t}\dd\LL^d=\int_{\RR^d}f\Delta g\dd\LL^d,
	$$
so that, integrating both sides with respect to $\dd\mu_d(U)$ and using Fubini's Theorem,
	$$
\int_{\RR^d}f^\rad\Delta g\dd\LL^d=	\int_{\RR^d}f\Delta g\dd\LL^d.
$$
Then, as $f^\rad=f^\cone$ and integrating by parts,
		$$
	\int_{\RR^d}g \dd\Tr \DIFF\nabla f^\cone=	\int_{\RR^d}g \dd\Tr \DIFF\nabla f.
	$$
	Therefore, by an approximation argument, recalling the explicit computation~\eqref{tracecone}, we obtain that
	$$
		\int_{B_r} \dd\Tr \DIFF\nabla f=- d\omega_d r^{d-1}\chi_{(0,1]}(r)
		\qquad\forall r>0.
	$$
	In particular, taking into account \eqref{dcass} and~\eqref{splitted}
	$$
-\Tr \DIFF\nabla f(B_1)=d\omega_d=|\DIFF_1^2 f^\cone|(B_1)= |\DIFF_1^2 f|(B_1).
$$	
Now Lemma~\ref{lemmaconvex} can be applied, to obtain that the function $f$ has a continuous and concave representative in $B_1$  that, without loss of generality, {we} still denote by $f$. By~\eqref{splitted} again, $f$ is affine on $\RR^d\setminus\bar B_1$, say $f(x)=\alpha\,\cdot\,x+\beta$ for $x\in \RR^d\setminus\bar B_1$, for some $\alpha\in\RR^d$ and $\beta\in\RR$.	Now $f^\rad=f^\cone$ forces $\beta=0$. 

Setting also $\tilde f(x)\defeq f(x)-\alpha\,\cdot\, x$, we conclude the proof by showing $\tilde f=f^\cone$. Notice that still $\tilde f$ is continuous and concave on $B_1$ and $\tilde f^\rad=f^\cone$. Notice that this last fact implies $\tilde f(0)=1$.

Now, for any $\sigma\in \partial B_1$, define $\tilde f_\sigma(s)\defeq \tilde f(s \sigma)$ for $s\in [0,\infty)$,
 a function continuous and concave in $[0,1)$ with $\tilde f_\sigma(0)=1$.
 Notice that for $\HH^{d-1}$-a.e.\ $\sigma\in \partial B_1$, $\tilde f_\sigma\in W^{1,1}_\rmloc((0,\infty))$. This can be seen either with a change of coordinates and the characterization of Sobolev functions  on lines or by approximation, using repeatedly integration in polar coordinates. Hence, for $\HH^{d-1}$-a.e.\ $\sigma\in \partial B_1$, the function $\tilde f_\sigma$ has a continuous representative in $[1,\infty)$. Now, for $\HH^{d-1}$-a.e.\ $\sigma\in \partial B_1$, $\tilde f_\sigma$ vanishes a.e. in $(1,\infty)$ (as $\tilde f$ vanishes identically on $\RR^d\setminus \bar B_1$), therefore this implies
 $\tilde{f}_\sigma(s)\to 0$ as $s\uparrow 1$ and the continuous representative is the one null in $[1,\infty)$. 
 Then, exploiting continuity and concavity,  for $\HH^{d-1}$-a.e.\ $\sigma\in \partial B_1$, $\tilde f_\sigma(s)\ge (1-s)$ for $s\in [0,1]$. Then it holds that $\tilde f\ge f^\cone$ $\LL^d$-a.e.\ on $B_1$, whence, being $\tilde f^\rad=f^\cone$, $\tilde f=f^\cone$ on $B_1$. 
\end{proof}

\begin{proposition}\label{propconeradp}
		Let $p\in[1,\infty)$. Let $f \in L^1_\rmloc(\RR^d)$ with bounded Hessian--Schatten variation and assume that 
	\begin{equation}\label{eqd2pfd2pfcone}
	f^\rad=f^\cone\qquad\text{and}\qquad	|\DIFF^2_p f|(\R^d)=|\DIFF^2_p f^\cone|(\R^d).
	\end{equation}
	Then $f$ is equal to $f^\cone$ up to a linear term: there exists $\alpha\in\R^d$ such that
	\begin{equation}\notag
		f(x)=f^\cone(x)+\alpha\cdot x
		\qquad\text{ for a.e.\ $x\in\R^d$.}
	\end{equation} 
\end{proposition}
\begin{proof}
	We focus on the case $p>1$ as the case $p=1$ has already been proved in  Proposition~\ref{propconerad}.	Let now $g:=\frac12 (f+f^\cone)$.  Recalling~\eqref{splitted}, $|\DIFF^2_p g|(\RR^d\setminus \bar B_1)=0$.
	Still, $g^\rad=f^\cone$, so that, by Lemma~\ref{radialbetter} and~\eqref{eqd2pfd2pfcone},
	\begin{align*}
		|\DIFF^2_p f^\cone|(\RR^d)\le |\DIFF^2_p g|(\RR^d)\le \frac12|\DIFF^2_p f|(\RR^d)+\frac12|\DIFF^2_p f^\cone|(\RR^d)=|\DIFF^2_p f^\cone|(\RR^d),
	\end{align*}
hence equality holds throughout and therefore $g$ satisfies~\eqref{eqd2pfd2pfcone} in place of $f$.
	
	We next decompose $\DIFF\nabla f$ in absolutely continuous and singular part, use that the singular one has a rank one density with respect to the total variation, and show that the absolutely continuous one is proportional to the one of $\DIFF\nabla f^\cone$. We are going to use  the theory of functions of bounded variation throughout, see e.g.\ \cite{AFP00}. The superscript $s$ denotes the singular part of a measure with respect to $\LL^d$. 
	We have a $\LL^d$-negligible Borel set $N\subseteq B_1$ such that $|\DIFF^2_1 f|^s\mres B_1=|\DIFF^2_1 f|\mres N$. Also $|\DIFF^2_1 g|^s\mres B_1=|\DIFF^2_1 g|\mres N$, being $|\DIFF^2_1 f^\cone|\mres B_1\ll\LL^d$, by~\eqref{eqd2fcone}. In addition
	\begin{align*}
		|\DIFF^2_p g|\mres N&\le \frac12|\DIFF^2_p f|\mres N+\frac12|\DIFF^2_p f^\cone|\mres N=\frac12|\DIFF^2_p f|\mres N\\&\le |\DIFF^2_p g|\mres N+\frac12|\DIFF^2_p f^\cone|\mres N= |\DIFF^2_p g|\mres N
	\end{align*}
hence equality holds throughout and in particular, $|\DIFF^2_p g|\mres N=\frac 12|\DIFF^2_p f|\mres N$. Now, recall that $|\DIFF^2_p f|\mres (B_1\setminus N)\ll\LL^d$ and $|\DIFF^2_p g|\mres (B_1\setminus N)\ll\LL^d$, also $|\DIFF^2_p f^\cone|\mres B_1\ll\LL^d$, by~\eqref{eqd2fcone}. Therefore, by Proposition~\ref{hessiananddiffgrad},
\begin{align*}
	|\DIFF^2_p g|(B_1)&=|\DIFF^2_p g|(N)+\int_{B_1\setminus N}\bigg|\dv{\DIFF\nabla g}{|\DIFF\nabla g|}\bigg|_p\dd{|\DIFF\nabla g|}=|\DIFF^2_p g|(N)+\int_{B_1\setminus N}\bigg|\dv{\DIFF\nabla g}{\LL^d}\bigg|_p\dd{\LL^d}\\&=|\DIFF^2_p g|(N)+\frac12\int_{B_1\setminus N}\bigg|\dv{\DIFF\nabla f}{\LL^d}+\dv{\DIFF\nabla f^\cone}{\LL^d}\bigg|_p\dd{\LL^d}\\&\le 
\frac 12|\DIFF^2_p f|(N)+\frac 12\int_{B_1\setminus N}\bigg|\dv{\DIFF\nabla f}{\LL^d}\bigg|_p+\bigg|\dv{\DIFF\nabla f^\cone}{\LL^d}\bigg|_p\dd{\LL^d}\\&\le
\frac 12|\DIFF^2_p f|(B_1)+\frac 12|\DIFF^2_p f^\cone|(B_1)=	|\DIFF^2_p g|(B_1),
\end{align*}
where we also used~\eqref{eqd2pfd2pfcone} for $f$ and $g$ and \eqref{splitted} in the last equality. Hence equality holds throughout, so that 
$$
\bigg|\dv{\DIFF\nabla f}{\LL^d}+\dv{\DIFF\nabla f^\cone}{\LL^d}\bigg|_p=\bigg|\dv{\DIFF\nabla f}{\LL^d}\bigg|_p+\bigg|\dv{\DIFF\nabla f^\cone}{\LL^d}\bigg|_p\qquad\LL^d\text{-a.e.\ on $B_1$}.
$$
By strict convexity of the $p$-Schatten norm (item~\ref{zuppa6} of Proposition~\ref{zuppa}), and the fact (by~\eqref{eqd2fcone}) that the density of $\DIFF\nabla f^\cone$ with respect to $\LL^d$ is nonzero
 $\LL^d$-a.e.\ on $B_1$, we have that for some Borel map $t:B_1\rightarrow  [0,\infty)$,
\begin{equation}\label{vmdsok}
	\dv{\DIFF\nabla f}{\LL^d}=t\dv{\DIFF\nabla f^\cone}{\LL^d}\qquad\LL^d\text{-a.e.\ on $B_1$}. 
\end{equation}
Now, by~\eqref{eqd2fcone}, for $q\in[1,\infty]$,
\begin{equation}\label{vmdsok1}
	\bigg|\dv{\DIFF\nabla f^\cone}{\LL^d}{}(x)\bigg|_q=\bigg|-\frac{|x|^2	{\rm Id}-x\otimes x}{|x|^3}\bigg|_q=\frac{(d-1)^{1/q}}{|x|}\qquad\LL^d\text{-a.e.\ on $B_1$}.
\end{equation}

Then, by~\eqref{vmdsok} and~\eqref{vmdsok1} (with $q=1,p$),
\begin{align*}
	\bigg|\dv{\DIFF\nabla f}{\LL^d}{}(x)\bigg|_p&=t(x)\frac{(d-1)^{1/p}}{|x|}={(d-1)^{1/p-1}}t(x)\frac{d-1}{|x|}\\&={(d-1)^{1/p-1}}\bigg|\dv{\DIFF\nabla f}{\LL^d}{}(x)\bigg|_1\qquad\LL^d\text{-a.e.\ on $B_1$}.
\end{align*}
Therefore, by Proposition~\ref{hessiananddiffgrad},
\begin{equation}\label{eq222221}
	|\DIFF^2_p f|(B_1\setminus N)={(d-1)^{1/p-1}}|\DIFF^2_1 f|(B_1\setminus N).
\end{equation}
On the singular set $N$, by Proposition~\ref{hessiananddiffgrad} and Alberti's rank 1 Theorem together with item~\ref{zuppa5} of Proposition~\ref{zuppa}, 
\begin{equation}\label{eq22222}
		|\DIFF^2_p f|(N)=\int_{N}\bigg|\dv{\DIFF\nabla f}{|\DIFF\nabla f|}\bigg|_p\dd |\DIFF\nabla f|=\int_{N}\bigg|\dv{\DIFF\nabla f}{|\DIFF\nabla f|}\bigg|_1\dd |\DIFF\nabla f|=	|\DIFF^2_1 f|(N).
\end{equation}
Therefore, by~\eqref{eq222221},~\eqref{eq22222} and~\eqref{splitted}, taking into account that $d\ge 2$ and $p\ge 1$ (hence $1\le (d-1)^{1-1/p}$),
\begin{equation}\label{vcsdacds}
\begin{split}
	|\DIFF_1^2 f|(B_1)&=|\DIFF^2_1 f|(B_1\setminus N)+|\DIFF^2_1 f|(N)=(d-1)^{1-1/p}|\DIFF^2_p f|(B_1\setminus N)+|\DIFF^2_p f|(N)\\&\le (d-1)^{1-1/p}\big(|\DIFF^2_p f|(B_1\setminus N)+|\DIFF^2_p f|(N)\big)=(d-1)^{1-1/p}|\DIFF_p^2 f|(B_1)\\&=(d-1)^{1-1/p}|\DIFF_p^2f ^\cone|(B_1)=|\DIFF_1^2f ^\cone|(B_1)
\end{split}
\end{equation}
where the last equality follows from~\eqref{dcass}.
Recalling~\eqref{splitted} and arguing exactly as for~\eqref{eq22222} for the first and third equalities,
\begin{equation}\label{csadkncs}
	|\DIFF^2_1 f|(\partial B_1)=|\DIFF^2_p f|(\partial B_1)=|\DIFF^2_p f^\cone|(\partial B_1)=|\DIFF^2_1 f^\cone|(\partial B_1).
\end{equation}
Then, by~\eqref{splitted}, exploiting~\eqref{vcsdacds} and~\eqref{csadkncs}
\begin{align*}
|\DIFF_1^2 f|(\RR^d)=|\DIFF_1^2 f|(B_1)+|\DIFF_1^2 f|(\partial B_1)\le |\DIFF_1^2 f^\cone|(B_1)+|\DIFF_1^2 f^\cone|(\partial B_1)=|\DIFF_1^2 f^\cone|(\RR^d).
\end{align*}
Recalling Lemma~\ref{radialbetter} together with~\eqref{eqd2pfd2pfcone}, the inequality above yields that $f$ satisfies~\eqref{eqd21fconed}, so that the conclusion follows from Proposition~\ref{propconerad}.
\end{proof}
\subsection{Proof of the main result}\label{proofmainextremal}
\begin{proof}[Proof of Theorem~\ref{coneexthm}]
	Let $f_1$ and $f_2$ be as in the statement and recall~\eqref{defrad}, so that we can define $f^\rad_i$ for $i=1,2$. As $f^{\cone}$ is already a radial function, we still have $\lambda f^\rad_1+(1-\lambda)f^\rad_2=f^\cone$. 
Now we compute, using Lemma~\ref{radialbetter} and the assumption,
\begin{align*}
	|\DIFF^2_p f^\cone|(\RR^d)&=|\DIFF^2_p (\lambda f_1^\rad+(1-\lambda)f_2^\rad)|(\RR^d)\le \lambda |\DIFF_p^2 f_1^\rad|(\RR^d)+ (1-\lambda) |\DIFF^2_p f^\rad_2|(\RR^d)\\&\le \lambda |\DIFF_p^2 f_1|(\RR^d)+ (1-\lambda) |\DIFF^2_p f_2|(\RR^d) =\lambda 	|\DIFF_p^2 f^\cone|(\RR^d)+(1-\lambda)	|\DIFF_p^2 f^\cone|(\RR^d)\\&=		|\DIFF^2_p f^\cone|(\RR^d),
\end{align*}
hence equality holds throughout. Therefore, $$|\DIFF^2_p f^\rad_i|(\RR^d)=|\DIFF^2_p f_i|(\RR^d)\qquad\text{for }i=1,2,$$
and
$$
|\DIFF^2_p (\lambda f_1^\rad+(1-\lambda)f_2^\rad)|(\RR^d)=  |\DIFF_p^2 (\lambda f_1^\rad)|(\RR^d)+  |\DIFF^2_p ((1-\lambda) f^\rad_2)|(\RR^d)
$$
so that, by Lemma~\ref{rigidity},
\begin{equation}\label{ugualmisura}
|\DIFF^2_p f^\cone|=\lambda	|\DIFF^2_p f_1^\rad|+(1-\lambda)|\DIFF^2_p f_2^\rad|
\end{equation}
as measures on $\RR^d$.
As $f_1^\rad$ and $f_2^\rad$  are radial functions with bounded Hessian--Schatten variation, by Proposition~\ref{htvradial},
$f_i^\rad(x)=g_i(|x|)$ for $g_i\in W^{1,1}_{\rmloc}((0,\infty))$. Similarly, $f^\cone(x)=g^\cone(|x|)=(1-|x|)_+$, notice that $\lambda g_1+(1-\lambda) g_2=g^\cone$. Then, using  repeatedly the representation formula of Proposition~\ref{htvradial} and~\eqref{ugualmisura}, 
\begin{align*}
	|\DIFF_p^2 f^\cone|(B_1)&=d\omega_d\int_0^1  \Vert (0,g_\cone',\ldots,g_\cone')\Vert_{\ell^p} s^{d-2}\dd s\\&\le
d\omega_d\bigg(\lambda \int_0^1  \Vert (0,g_1',\ldots,g_1')\Vert_{\ell^p} s^{d-2}\dd s+(1-\lambda) \int_0^1  \Vert (0,g_2',\ldots,g_2')\Vert_{\ell^p} s^{d-2}\dd s\bigg)
	\\&\le \lambda |\DIFF_p^2 f^\rad_1|(B_1)+(1-\lambda)|\DIFF_p^2 f^\rad_2|(B_1)=|\DIFF_p^2 f^\cone|(B_1),
\end{align*}
hence equality holds throughout. In particular, 
as we have obtained
$$
d\omega_d\int_0^1  \Vert (0,g_i',\dots,g_i')\Vert_{\ell^p} s^{d-2}\dd s=|\DIFF^2_p f_i^\rad|(B_1)\qquad\text{for }i=1,\,\,2,
$$
 exploiting the representation formula of Proposition~\ref{htvradial}, 
we have that $g_1'$ and $g_2'$ are constant on $(0,1)$. Also, by~\eqref{ugualmisura},  and the representation formula of Proposition~\ref{htvradial} again,
$g_1'$ and $g_2'$ vanish identically on $(1,\infty)$. Recall also that  $g_i\in W^{1,1}_{\rmloc}((0,\infty))$, so that $g_i$ has a continuous representative, for $i=1,\,2$. Hence, there exist $\alpha_1,\alpha_2\in\RR$ and $\beta_1,\,\beta_2\in\RR$ such that 
$$
g_i(s)=\alpha_i (1-s)_++\beta_i.
$$
Now, $\lambda g_1+(1-\lambda) g_2=g^\cone$ forces $\lambda\alpha_1+(1-\lambda)\alpha_2=1$, whereas
$$
|\alpha_i||\DIFF_p^2 f^\cone|(\RR^d)=|\DIFF_p^2 f_i^\rad|(\RR^d)=|\DIFF_p^2 f^\cone|(\RR^d)\qquad\text{for }i=1,2
$$
forces $|\alpha_1|=|\alpha_2|=1$. Hence, $\alpha_1=\alpha_2=1$.

Therefore, to sum up, we have, for $i=1,\,2$,
$$
	f_i^\rad=f^\cone+\beta_i,$$
	so that$$
	|\DIFF^2_p f_i^\rad|(\RR^d)=|\DIFF_p^2 f^\cone|(\RR^d)=|\DIFF^2_p f_i|(\RR^d).
$$
Notice that $f_i^\rad-\beta_i=(f_i-\beta_i)^\rad$.
Now we use  Proposition~\ref{propconeradp}  to infer that  $$
f_i(x)-\beta_i=f^\cone(x)+a_i\,\cdot\,x\qquad\text{for a.e.\ }x\in\RR^d,
$$
hence the proof is concluded with $L_i(x)\defeq\alpha_i\,\cdot\, x+\beta_i$.	
\end{proof}

\section{Solutions of the minimization problem}\label{sectionsolutions}

 In this section we stick to the two dimensional case $d=2$. Recall that, by Proposition~\ref{sobo}, functions with bounded Hessian--Schatten variation are continuous, as we are in dimension $2$ and hence  the evaluation functionals in~\eqref{deffunct} below are meaningful (we will implicitly take  the continuous representative, whenever it is possible).
 
 Fix  $\Omega\subseteq\RR^2$ open, and fix $x_1,\ldots, x_N
 \in\Omega$ distinct test points and fix also $y_1,\ldots,y_N\in\RR$. For $\lambda\in [0,\infty]$ and $p,\,q\in [1,\infty]$ we consider the functional
\begin{equation}\label{deffunct}
	 \FF_\lambda^{p,q}:L^1_\rmloc(\Omega)\rightarrow [0,\infty]\qquad\text{defined as}\qquad \FF_\lambda^{p,q} (f)\defeq|\DIFF^2_p f|(\Omega)+\lambda\Vert (f(x_i)-y_i)_{i=1,\ldots,N}\Vert_{\ell^q},
\end{equation}
where we adopt the convention that $\infty\,\cdot\, 0=0$. Notice that if $p=q=1$, we have that $\FF^{1,1}_\lambda=\FF_\lambda$, where $\FF_\lambda$ is defined in~\eqref{defFF} in the Introduction.
 
 Our aim is to establish conditions  under which $\FF_\lambda^{p,q}$ has minimizers, i.e.\ we want to ensure the existence of a minimizer of 
 $$
\inf_{f\in L^1_\rmloc(\Omega)} \FF^{p,q}_\lambda(f).
 $$
 
It turns out that for many values of $\lambda,\,p,\,q$, minimizers indeed exist. Here we state our main results in this direction.
 
  \begin{thm}\label{csaa}
 	Let $p,\,q\in [1,\infty]$ and let $\lambda\in [0, 2^{1/p-1} 4\pi]$. Then there exists a minimizer of $\FF_\lambda^{p,q}$.
 \end{thm}

 \begin{thm}\label{csaaa}
Let $\lambda\in [0,\infty]$. Then there exists a minimizer of $\FF_\lambda^{1,1}$.
 \end{thm}

Theorem~\ref{csaa} and Theorem~\ref{csaaa} will follow easily from the results of  Section~\ref{auxexist}. We defer their proof of to Section~\ref{sectprofexi}.
 
 \subsection{Auxiliary results}\label{auxexist}
 
 For the following lemma, we recall  again that functions with bounded Hessian--Schatten variation in dimension $2$ are automatically continuous. Hence, the evaluation (at $0$) functional in the infimum above is meaningful. The spirit of this lemma is to provide us with \say{bump} functions whose Hessian--Schatten total variation is almost optimal.
 \begin{lem}\label{goodcutoff}
 	Let $p\in [1,\infty]$. Then it holds that 
 	\begin{equation}\label{cdasnojcnas}
 		\inf \left\{|\DIFF^2_p f|(\RR^2): \text{$f\in L^1_\rmloc(\RR^2)$ with 
		compact support and $f(0)=1$}\right\}=2^{1+1/p}\pi.
 	\end{equation}
 In particular, thanks to \eqref{dcass},  the infimum is attained by the cut cone $x\mapsto (1-|x|)^+$ when $p=1$.
 \end{lem}
 \begin{proof}
 	For $\epsilon\in (0,1)$, define $f_\epsilon(x)=(1-|x|^{\epsilon})\vee 0$.  By Proposition~\ref{htvradial},
 	$$
 	|\DIFF^2_p f_\epsilon|(\RR^2)=2\pi\bigg(\int_{0}^1 s^{\epsilon-1}\Vert(\epsilon(\epsilon-1),\epsilon)\Vert_{\ell^p}\dd s+\epsilon\bigg)\rightarrow 2^{1+1/p}\pi\qquad\text{as }\epsilon\searrow 0,
 	$$
 	so that we have $\le$ in~\eqref{cdasnojcnas}.
 	
 	We prove now the opposite inequality in~\eqref{cdasnojcnas}. Take then $f\in L^1_\rmloc(\RR^2)$, compactly supported, with bounded Hessian--Schatten variation and such that $f(0)=1$. We have to prove that $|\DIFF^2_p f|(\RR^2)\ge 2^{1+1/p}\pi$.  Using Lemma~\ref{mollif}, Lemma~\ref{radialbetter}, we see that we can assume with no loss of generality that $f\in C^\infty_\rmc(\RR^2)$ and $f$ is radial, say $f(x)=g(|x|)$,  with $g(0)=1$ and $g'_+(0)=0$. Now, by Proposition~\ref{hessiananddiffgrad} and the inequality $(|a|+|b|)\leq 2^{1-1/p}(|a|^p+|b|^p)^{1/p}$, we obtain that 
	$$|\DIFF^2_p f|(\RR^2)\ge 2^{1/p-1}|\DIFF^2_1 f|(\RR^2).$$
 	Hence, it is enough to show the claim in the case $p=1$, i.e.\ we have to show that $|\DIFF^2_1 f|(\RR^2)\ge 4\pi$. We compute now
 	$$
 	\int_0^\infty s|g''|\dd s\ge \int_0^\infty sg''\dd s=-\int_0^\infty g'\dd s=1
 	\qquad\text{and}\qquad
 	\int_0^\infty |g'|\dd s\ge -\int_0^\infty g'\dd s=1
 	$$
 	so that by by Proposition~\ref{htvradial},
 	\begin{equation*}
 		|\DIFF^2_1 f|(\RR^2)=2\pi \int_0^\infty s|g''|+|g'|\dd s\ge 4\pi.\qedhere
 	\end{equation*}
 \end{proof}
 
 The existence of \say{good bump functions} granted by Lemma~\ref{goodcutoff} allows us to prove, in
 Proposition~\ref{lambdabig} below, that for $\lambda$ large enough the infimum of $\FF_\lambda^{p,q}$ does not depend on $\lambda$, namely that minimizing $\FF_\lambda^{p,q}$  asymptotically promotes the perfect fit with the data.  
 
 \begin{prop}\label{lambdabig}
 	Let $p,\,q\in [1,\infty]$ and let $\lambda\in [2\pi 2^{1/p}{ N}^{1-1/q},\infty]$.
 	Then
 	$$\inf_{f\in L^1_\rmloc(\Omega)} \FF^{p,q}_\lambda(f)=\inf_{f\in L^1_\rmloc(\Omega)} \FF_{\infty}^{p,q}(f).$$
	In particular, in this range of $\lambda$, the infima are also independent of $q$.
 \end{prop}
 \begin{proof}
 	We let $r\in (0,\infty)$ small enough so that $\dist(x_i,x_j)> 3 r$ if $i\ne j$. Let $\epsilon\in(0,1)$. For $i=1,\ldots,N$, by Lemma~\ref{goodcutoff} and a scaling argument, we take $g_i\in C_\rmc(\RR^2)$ with $g(x_i)=1$, $\supp g_i\subseteq B_{r}(x_i)$ and $|\DIFF^2_p g_i|(\RR^2)\le 2^{1+1/p}\pi +\epsilon$.
 	
 	Then we consider $f\in L^1_\rmloc(\Omega)$ and we set\begin{equation}\label{modifyf}
 		 \tilde f\defeq f-\sum_i  (f(x_i)-y_i)g_i.
 	\end{equation}
 	Notice $\tilde f(x_i)=y_i$ for every $i=1,\ldots,N$ and that
 	\begin{align*}
 		|\DIFF^2_p \tilde f|(\Omega)&\le |\DIFF^2_p  f|(\Omega)+(2^{1+1/p}\pi+\epsilon) \sum_{i=1}^N |f(x_i)-y_i|\\&=|\DIFF^2_p f|(\Omega)+(2^{1+1/p}\pi+\epsilon) \Vert (f(x_i)-y_i)_i\Vert_{\ell^1}\\&\le  |\DIFF^2_p f|(\Omega)+ (2^{1+1/p}\pi+\epsilon) N^{1-1/q}\Vert (f(x_i)-y_i)_i\Vert_{\ell^q}.
 	\end{align*}
 	Therefore, being $\epsilon\in (0,1)$ arbitrary and $f\in L^1_\rmloc(\Omega)$ arbitrary, we have that
 	$$\inf_{f\in L^1_\rmloc(\Omega)} \FF^{p,q}_{\infty}(f)\le
\inf_{f\in L^1_\rmloc(\Omega)} \FF_{\lambda}^{p,q}(f)
	\qquad{\text{whenever $\lambda\geq 2\pi 2^{1/p}N^{1-1/q}$.}}$$
 	As also $\FF_{\infty}^{p,q}(\tilde f)\ge \FF^{p,q}_{\lambda}(f)$, we have proved the claim, thanks to our choice of $\lambda$.
 \end{proof}

 The following lemma estimates how much the evaluation functional at $x$ differs from the average  functional on $B_r(x)$, hence allows us to quantify the error we make replacing the evaluation functional  with another functional that has the advantage of being continuous with respect to weaker notion of convergence.
 
 \begin{lem}\label{trick}
 	Let $f\in L^1_\rmloc(\Omega)$ with bounded Hessian--Schatten variation in $\Omega$. Let also $B=B_r(x)\subseteq\Omega$ such that $2 B:=B_{2r}(x)\subseteq\Omega$.  Then, if $p\in [1,\infty]$, 
 	\begin{equation}\label{trick1}
 		\abs{f(x_0)-\dashint_B f}\le 2^{1-1/p}\left( \frac{1}{4\pi}|\DIFF_p^2 f|(B)+\frac{1}{2 \pi}
		|\DIFF^2_p f  |(2 B\setminus B)\right).
 	\end{equation}
 \end{lem}
 \begin{proof}
 	We can assume with no loss of generality that $x=0$.
 	By approximation of $r$ from below, we can also assume that  $|\DIFF^2_1 f|(\partial B)=0$.
 	Hence, using Proposition~\ref{myersserrin} and Lemma~\ref{radialbetter}, we can assume in addition that $f$ is radial and $f\in C^{\infty}(2B)$, say  $f(\,\cdot\,)=g(|\,\cdot\,|)$. Notice that $g'_{+}(0)=0$.
 	We then compute
 	\begin{align*}
 		f(0)-\dashint_B f&=g(0)-\frac{2}{r^2} \int_0^r s g(s)\dd s =\frac{2}{ r^2} \int_0^r s(g(0)-g(s))\dd s\\&=-\frac{2}{r^2} \int_0^r s\int_0^s g'(\tau)\dd\tau\dd s=-\frac{2}{r^2} \int_0^r g'(\tau)\biggl( \int_\tau^r s\dd s\biggr) \dd\tau,
 	\end{align*}
 	so that 
 	\begin{equation}\label{asosvndsoa}
 		\abs{f(0)-\dashint_B f }\le \frac{2}{r^2}\int_0^r |g'| r^2/2\dd s=\int_0^r |g'|\dd s.
 	\end{equation}
 	We stick for the moment to the case $p=1$. We use Proposition~\ref{htvradial} to compute 
 	\begin{equation}\label{htvballs}
 		\begin{split}
  			|\DIFF^2_1 f|(2B\setminus B)&=2\pi \int_r^{2 r} s|g''|+|g'|\dd s,\\
  			|\DIFF^2_1 f|( B)&=2\pi\int_0^r s|g''|+|g'|\dd s
 		\end{split}
 	\end{equation}
and we take $\xi\in (r,2 r)$ such that
 \begin{equation}\label{casnjssaa}
r | g'|(\xi)\le \int_r^{2 r} |g'|\dd s.
 	\end{equation}

 	Now we write $\{g'>0\}\cap (0,\xi)=\bigcup_k I_k$ and $\{g'<0\}\cap (0,\xi)=\bigcup_k J_k$, where $I_k$ and $J_k$ are countably many pairwise disjoint open intervals. Notice that if $p\in\partial I_k$  for some $k$, then either $p=\xi$ or $g'(p)=0$. Then, if we take $I_k$ such that $\xi\in\partial I_k$,
 	$$
 	\int_{I_k} s|g''|\dd s\ge -\int_{I_k} s g''\dd s=\int_{I_k} g'\dd s- \xi g'(\xi)=\int_{I_k} |g'|\dd s- \xi |g'|(\xi),
 	$$
 	whereas if we take $I_k$ such that $\xi\notin\partial I_k$,
 	$$
 	\int_{I_k} s|g''|\dd s\ge -\int_{I_k} s g''\dd s=\int_{I_k} g'\dd s=\int_{I_k} |g'|\dd s.
 	$$
 	Similar inequalities hold in the case of  an interval of the type $J_k$. 
 	Therefore, summing over all intervals $I_k$ and $J_k$,
 	$$
 	\int_0^{2 r}s|g''|\dd s\ge \int_0^\xi s |g''|\dd s\ge \int_0^\xi |g'|\dd s-\xi |g'|(\xi),
 	$$
 	so that, by the choice of $\xi$ due to~\eqref{casnjssaa},
 	\begin{align*}
 		\int_0^r |g'|\dd s\le	\int_0^\xi |g'|\dd s \le \int_0^{2r}s |g''|\dd s+\xi |g'|(\xi)\le  \int_0^{2r}s |g''|\dd s+
2\int_r^{2 r} |g'|\dd s.
 	\end{align*}
	
 	Then, using also~\eqref{asosvndsoa} and~\eqref{htvballs},
 		\begin{align*}
 		2 \abs{f(0)-\dashint_B f}&\le 2\int_0^r |g'|\dd s\le \int_0^r |g'|\dd s+\int_0^r s|g''|\dd s+\int_r^{2 r} s|g''|\dd s+
 		2\int_r^{2 r} |g'|\dd s
\\
&\le \frac{1}{2\pi} |\DIFF ^2_1 f|(B)+\frac{1}{\pi}|\DIFF ^2_1 f|(2B \setminus B),
 	\end{align*}
 	whence the claim for $p=1$. 
 	For the general case, simply notice that $|\DIFF^2_1f|(B)\le 2^{1-1/p}|\DIFF^2_p f|(B)$ and the same holds for $2B\setminus B$, by $\ell_1-\ell_p$ inequality and Proposition~\ref{hessiananddiffgrad}.
 \end{proof}
 \begin{rem}
 	Notice that the constant ${1}/{(4\pi)}$ in front of $|\DIFF^2_p f|(B)$ in~\eqref{trick1} is somehow optimal. We can realize this considering the sequence of functions $f_\epsilon$ used to prove Lemma~\ref{goodcutoff}.\fr
 \end{rem}
 
 By Lemma~\ref{trick}, there is no surprise in knowing that, given a weakly convergent sequence $f_k\rightharpoonup f$, in duality with the space $L^\infty_\rmc(\Omega)$ of $L^\infty$ function with compact (essential) support, we can estimate how much the evaluation functional fails to converge in terms of concentration of Hessian--Schatten total variation at $x$.

 \begin{lem}\label{excess}
 	Let $f\in L^1_\rmloc(\Omega)$ and let $(f_k)\subseteq L^1_\rmloc(\Omega)$ such that $f_k\rightharpoonup f$ in duality with $L^\infty_\rmc(\Omega)$
 	with $\sup_k|\DIFF^2_p f_k|(A)<\infty$ for any open set $A\Subset\Omega$. Then, $f$ has locally bounded Hessian--Schatten variation in $\Omega$ and
	for any $x\in\Omega$ one has
 	\begin{equation}\label{csanoca}
 		\limsup_k |f(x)-f_k(x)|\le \frac{2^{1-1/p}}{4\pi}\lim_{r\searrow 0}\limsup_k |\DIFF^2_p f_k|(B_{r}(x)).
 	\end{equation}
 \end{lem}
 \begin{proof}
 	First, take a non relabelled subsequence so that $\lim_k |f(x)-f_k(x)|$ exists and equals the $\limsup_k$ at the left hand side of~\eqref{csanoca}.
 	
 	We assume that there exists $r_1>0$ small enough so that $B_{r_1}(x)\subseteq\Omega$ and moreover 
	that $\limsup_k|\DIFF^2 f_k|(B_{r_1}(x))<\infty$, otherwise there is nothing to show. 
 By lower semicontinuity this implies that $f$ has bounded Hessian-Schatten variation in $B_{r_1}(x)$.
 	We extract a further non relabelled subsequence such that, for some finite measure $\mu$ on $B_{r_1}(x)$, $|\DIFF^2_p f_k|\rightharpoonup\mu $ in duality with $C_\rmc(B_{r_1}(x))$. 
 	
 	Let now $r\in (0,r_1/2)$. Then,
 	\begin{align*}
 		|f(x)-f_k(x)|	\le& \abs{f(x)-\dashint_{B_r(x)} f}+ \abs{\dashint_{B_r(x)} f-\dashint_{B_r(x)} f_k}+ \abs{f_k(x)-\dashint_{B_r(x)} f_k}.
 	\end{align*}
 	Now notice that by continuity of $f$ the first summand converges to $0$ as $r\searrow 0$, whereas, by the convergence assumption the second summand converges to $0$ as $k\rightarrow\infty$.
 	Also, by Lemma~\ref{trick}, we bound the third summand as follows
 	$$
 	\abs{f_k(x)-\dashint_{B_r(x)} f_k}\le  2^{1-1/p}\left( \frac{1}{4\pi}|\DIFF_p^2 f_k|(B_r(x))+\frac{1}{2 \pi}|\DIFF^2_p f_k  |( B_{2 r}(x)\setminus B_r(x))\right).
 	$$
 	To conclude, it is enough notice that that
 	\begin{equation*}
\limsup_{r\searrow 0}\limsup_k|\DIFF^2_p f_k | (B_{2 r}(x)\setminus B_r(x))\le \lim_{r\searrow 0} \mu(\bar B_{2 r}(x)\setminus B_r(x))=0.\qedhere
 	\end{equation*}
 \end{proof}
 
By using the results above, we can prove the lower semicontinuity of $\FF_\lambda^{p,q}$.  In the case $q=1$, notice that
the argument used in the proof of Proposition~\ref{lambdabig} together with the next result can be used to show that $\FF^{p,1}_\lambda$ is precisely the relaxed functional of $\FF_\infty^{p,1}$ when $\lambda=2^{1+1/p}\pi$.

 \begin{lem}\label{lowersemicont}
 	Let $p,\,q\in [1,\infty]$ and let $\lambda\in [0,2^{1/p-1} 4\pi ]$. Then $\FF_\lambda^{p,q}$ is lower semicontinuous with respect to weak convergence in duality with $L^\infty_\rmc(\Omega)$.
 \end{lem}

 \begin{proof}
 	Let $(f_k)\subseteq L^1_\rmloc(\Omega)$ be such that $f_k\rightharpoonup f$ in duality with $L^\infty_\rmc(\Omega)$, for some $f\in L^1_\rmloc(\Omega)$. We have to prove that
 	$$
 	\FF_\lambda^{p,q}(f)\le\liminf_k  \FF_\lambda^{p,q}(f_k).
 	$$
 	First, extract a non relabelled subsequence such that $\FF_\lambda^{p,q}(f_k)$ has a limit, as $k\rightarrow \infty$, which equals the right hand side of the inequality above.
 	Then, we can assume that $\liminf_k |\DIFF^2_p f_k|(\Omega)<\infty$, otherwise there is nothing to show. Hence $f$ has bounded Hessian--Schatten variation in $\Omega$ and, up to the extraction of a non relabelled subsequence, we can assume that $|\DIFF^2_p f_k|\rightharpoonup\mu$ in duality with $C_\rmc(\Omega)$ for some finite measure $\mu$ on $\Omega$. Even though $\mu$ depends on $p$, we do not make this dependence explicit.
 	Also, we extract a  non relabelled subsequence such that for every $i=1,\ldots,N$,
 	$|f(x_i)-f_k(x_i)|$ has a (finite) limit as $k\rightarrow \infty$.
 	
 	Notice that for every $z\in\Omega$ one has
\begin{equation}\label{dvcsaascas}
	  		 \mu(\{z\})\leq	\lim_{r\searrow 0}\limsup_k |\DIFF_p^2 f_k|(\bar B_r(z)){\leq}
	\lim_{r\searrow 0}\mu(\bar B_r(z))=\mu(\{z\}).
\end{equation}
 	We compute, as $|\DIFF^2_p f|(\{z\})=0$ for every $z\in\Omega$,
 	\begin{equation}\label{cmasd1}
 		\FF_\lambda^{p,q}(f)=|\DIFF^2_p f|(\Omega)+\lambda \Vert (f(x_i)-y_i)_i\Vert_{\ell^q}=\lim_{r\searrow 0}|\DIFF^2_p f|\bigg(\Omega\setminus\bigcup_{i=1}^N\bar B_{r}(x_i)\bigg)+\lambda \Vert (f(x_i)-y_i)_i\Vert_{\ell^q}.
 	\end{equation}
 	By lower semicontinuity,
 	\begin{equation}\notag
 		|\DIFF^2_p f|\bigg(\Omega\setminus\bigcup_{i=1}^N\bar B_{r}(x_i)\bigg)\le \liminf_k |\DIFF^2_p f_k|\bigg(\Omega\setminus\bigcup_{i=1}^N\bar B_{r}(x_i)\bigg)
 	\end{equation}
 	so that by \eqref{dvcsaascas}
 	\begin{equation}\label{cmasd2}
 		\lim_{r\searrow 0}|\DIFF^2_p f|\bigg(\Omega\setminus\bigcup_{i=1}^N\bar B_{r}(x_i)\bigg)\le \liminf_k |\DIFF^2_p f_k|(\Omega)-\sum_{i=1}^N\mu(\{x_i\}).
 	\end{equation}
 	
 	Also, by Lemma~\ref{excess} and  \eqref{dvcsaascas},
 	\begin{equation}\notag
 		\lim_k \Vert (f(x_i)-f_k(x_i))_i\Vert_{\ell^q}\le 	\lim_k \Vert (f(x_i)-f_k(x_i))_i\Vert_{\ell^1}\le   \frac{2^{1-1/p}}{4\pi} \sum_{i=1}^N \mu(\{x_i\}),
 	\end{equation}
 	so that 
 	\begin{equation}\label{cmasd3}
 		\Vert (f(x_i)-y_i)_i\Vert_{\ell^q}\le \frac{2^{1-1/p}}{4\pi} \sum_i \mu(\{x_i\})+\liminf_k  \Vert (f_k(x_i)-y_i)_i\Vert_{\ell^q}.
 	\end{equation}
 	Inserting~\eqref{cmasd2} and~\eqref{cmasd3} into~\eqref{cmasd1} we obtain, by the super additivity of the $\liminf$,
 	\begin{align*}
 		\FF_\lambda^{p,q}(f)\le \liminf_k \FF_\lambda^{p,q}(f_k)+\bigg(\lambda  \frac{2^{1-1/p}}{4\pi}-1\bigg) \sum_i \mu(\{x_i\}),
 	\end{align*}
 	whence the claim by the choice of $\lambda$.
 \end{proof}
 
 Weak relative compactness of minimizing sequences for $\FF_\lambda^{p,q}$ is obtained through a classical argument, the only (slight) technical difficulty relies in possibly irregular domains $\Omega$.
 \begin{lem}\label{compacntess}
 	Let $p,\,q\in [1,\infty]$ and let $\lambda\in[0,\infty]$. Then there exist a minimizing sequence $(f_k)$ for $\FF_\lambda^{p,q}$ and a function $f\in L^1_\rmloc(\Omega)$ such that $f_k\rightharpoonup f$ in duality with $L^\infty_{\rmc}(\Omega)$.
 \end{lem}
 \begin{proof}
 	We assume $\lambda>0$, the case $\lambda=0$ being trivial. We also assume that $\Omega$
	is connected, as we can do the modifications independently in each connected component of $\Omega$.
	Let now $(f_k)\subseteq L^1_\rmloc(\Omega)$ be a minimizing sequence for $\FF_\lambda^{p,q}$. In particular, the sequence $(|\DIFF^2 f_k|(\Omega))$ is bounded as well as the sequence $(|f_k(x_i)|)$, for every $i=1,\ldots,N$. 
 	Now we are going to modify $(f_k)$ to obtain a new sequence $(\tilde{f}_k)\subseteq L^1_\rmloc(\Omega)$ that is still minimizing but in $\Omega$ is locally uniformly bounded. 
 	
 	There are two cases to be considered:
 	\begin{enumerate}[label={\rm(\alph{enumi})}]
 		\item\label{asssss1} $N\ge 3$ and there are three points $x_{i_1},x_{i_2},x_{i_3}\in\{x_1,\ldots, x_N\}$ such that $x_{i_2}-x_{i_1}$ and $x_{i_3}-x_{i_1}$ are linearly independent.
 		\item \label{asssss2} either $N=0$ or all the points $x_i$ are on a line $\{t v+c:t\in\RR\} \subseteq\RR^2$, for some $v\in\RR^2\setminus\{0\}$ and $c\in\RR$.
 	\end{enumerate}
 	We treat the two cases separately.
 	\medskip
\\\textbf{Case \ref{asssss1}}. 
In this case no modification is needed, indeed we show that $(f_k)$ is locally uniformly bounded in $\Omega$. Take a compact set $K\subseteq\Omega$. For $\eps:=\frac12 \dist(K,\partial\Omega)$ we select points $y_0, y_1, \ldots, y_M\in K$ such that $K\subseteq \cup_j B_\eps(y_j)$, then curves $\gamma_j\subseteq\Omega$ joining $y_j$ to $y_0$, and finally curves $\hat\gamma_i\subseteq\Omega$ joining $x_i$ to $y_0$. 
Let 
$$
K':=\bigcup_{j=0}^M \overline B_\eps(y_j) \cup \bigcup_{j=1}^M \gamma_j \cup 
\bigcup_{i=1}^N \hat\gamma_i.
$$
Then $\cup_i\{x_i\}\cup K\subseteq K'\subseteq \Omega$, and $K'$ is compact and connected. 
Therefore, to prove uniform
boundedness of $(f_k)$ on $K$, we can assume with no loss of generality that all points $x_i$ belong to $K$ and that
$K$ is connected.
	
	Now we take $\delta\in(0,1)$ small enough so that 
	$\Omega'\defeq {B_{{2}\delta}}(K)$ satisfies $\overline{\Omega'}\subseteq\Omega$. Hence $\Omega'$ is a connected domain. 
We show now that $\Omega'$ is a (bounded) John domain, then $\Omega'$ satisfies Poincar\'e inequalities, by  \cite[Lemma 3.1 and Theorem 5.1]{Bojar} and the trivial inequality
$$
	\bigg(\dashint_{\Omega '} \Big|f-\dashint_{{\Omega'}} f\Big|^{q}\dd\LL^2\bigg)^{1/q}\le 2 	\bigg(\dashint_{\Omega '} |f-a|^{q}\dd\LL^2\bigg)^{1/q}\qquad\text{for every }a\in\RR
	$$
	that holds for every $f\in L^1(\Omega')$ and $q\in [1,\infty)$.
Fix any $p_0\in K$. We have to show that there exist $0<\alpha\le \beta$ such that for every $p\in\Omega'$, there exists a rectifiable curve $\gamma:[0,l(\gamma)]\rightarrow\Omega'$, parametrized by arc length, joining $p$ to $p_0$ and such that 
$l(\gamma)\le \beta$ and 
\begin{equation}\label{cdascsadas}
	\dist(\gamma(t),\partial\Omega')\ge \frac{\alpha t}{l(\gamma)}\qquad\text{for every }t\in [0,l(\gamma)].
\end{equation}
To prove this, notice first that there exists $\beta'>0$ such that for every $p\in K$, there exists rectifiable curve $\gamma$, parametrized by arc length, joining $p$ to $p_0$, with   image contained in $B_{\delta}(K)\subseteq\Omega'$ and length bounded by $\beta'$. This 
follows from the connectedness of $B_\delta(K)$ and the 
compactness of $K$ (simply take a finite covering of $K$ of balls of radius $\delta$ and centre in $K$ and consider the rectifiable curves with image in $B_{\delta}(K)$ joining the centres of these balls); also, $\gamma$ satisfies \eqref{cdascsadas} with $\alpha\defeq\delta $. Then the claim for arbitrary $p\in \Omega'$ follows: indeed, for any $p\in \Omega'\setminus K$, $p\in B_{2\delta}(q)$ with $q\in K$, then we join the radial curve connecting $p$ to $q$ to the curve connecting $q$ to $p_0$ obtained as before and we have that $l(\gamma)\le{2}\delta+ \beta'\eqdef \beta$ and moreover
$\gamma$ still satisfies \eqref{cdascsadas} (with $\alpha=\delta$ as before): indeed, for $t\in [0,|p-q|]$,
\begin{align*}
	\dist(\gamma(t),\partial\Omega')\ge {2}\delta-|p-q|+t\ge {2}\delta \frac{t}{|p-q|}\ge {2}\delta\frac{t}{l(\gamma)},
\end{align*}
whereas  for $t\in [|p-q|,l(\gamma)]$, \eqref{cdascsadas} follows as before.

	Take also $\psi\in C_\rmc^\infty(\RR^2)$ such that $\supp\psi\subseteq\Omega'$ and $\psi=1$ on a neighbourhood of $K$. By Proposition~\ref{sobo} and standard calculus rules, the sequence $(|\DIFF^2 (\psi \hat f_k)|(\RR^2))$ is bounded, where $\hat f_k=f_k-g_k$ with $g_k$ suitable affine perturbation. Therefore, by \cite[Proposition 3.1]{Deme} and the compactness of
support of $\psi\hat{f}_k$, we have that $\psi\hat{f}_k$ are uniformly bounded in $L^\infty(\RR^2)$, in
	particular $\hat{f}_k$ are uniformly bounded in $L^\infty(K)$.
	Now, as $|g_k(x_i)|=|\hat{f}_k(x_i)-f(x_i)|$ are bounded for every $i=i_1,\,i_2,\,i_3$, it is easy to infer, by the assumption in \ref{asssss1} that the perturbations $g_k$ are uniformly bounded. Hence $\Vert f_k\Vert_{L^\infty(K)}$ is bounded and, since $K$ is arbitrary, the claim follows by weak compactness.
 	\medskip
	\\\textbf{Case \ref{asssss2}}. If $N\le 2$, there is an affine function $f_*$ with 
	$f_*(x_i)=y_i$ for all $i$, and therefore $\FF_\lambda^{p,q}(f_*)=0$.
	 We can therefore assume $N\ge 3$.
	Let $v^\perp$ be a unit vector orthogonal to $v$, and choose $\eps\in(0,1)$ sufficiently small that $x_0:=x_1+\eps v^\perp\in\Omega$. Define
	 \begin{equation*}
	  \tilde{f}_k(x):=f_k(x)-\frac1\eps f_k(x_0) (x-x_1)\cdot v^\perp .
	 \end{equation*}
	 As $\FF_\lambda^{p,q}(\tilde{f}_k)=
	 \FF_\lambda^{p,q}(f_k)$, this is also a minimizing sequence, with the additional property that $\tilde{f}_k(x_0)=0$ for all $k$. The conclusion follows then from the argument of the previous case.	 
 \end{proof}

\subsection{Proof of the main results}\label{sectprofexi}
Having proved the results in Section~\ref{auxexist}, Theorem~\ref{csaa} and Theorem~\ref{csaaa} follow in a  immediate, classical way.
  \begin{proof}[Proof of Theorem~\ref{csaa}]
 	The statement is proved by the direct method of calculus of variations, by Lemma~\ref{lowersemicont} and Lemma~\ref{compacntess}.
 \end{proof}
 \begin{proof}[Proof of Theorem~\ref{csaaa}]
 Let $\lambda_c{\defeq}4\pi$.
 We argue as in Proposition~\ref{lambdabig}, starting from a minimizer $f$ of $\FF^{1,1}_{\lambda_c}$ granted by Theorem~\ref{csaa}.
 We modify $f$ subtracting $\sum_i (f(x_i)-y_i)g_i$ where this time $g_i$ are rescaled cut cones (see \eqref{modifyf}), in such a way that
 $$\tilde{f}\defeq f-\sum_i (f(x_i)-y_i)g_i$$ has a perfect fit with the data. Since $|\DIFF^2_1 g_i|(\R^2)=4\pi$ (recall e.g.\  Lemma \ref{goodcutoff}), one has
 \begin{equation*}
 \FF^{1,1}_\infty(\tilde f)\leq|\DIFF^2_1\tilde{f}|(\Omega)\leq 
 |\DIFF^2_1 f|(\Omega)+\sum_i |\DIFF^2_1 g_i|(\R^2)|f(x_i)-y_i|\\
 =\FF^{1,1}_{\lambda_c}(f).
 \end{equation*}
 This, taking the inequality $\FF^{1,1}_\lambda\leq\FF^{1,1}_\infty$ into account, 
 proves that $\tilde{f}$ is a minimizer of $\FF^{1,1}_\lambda$ for any $\lambda\geq\lambda_c$.
 \end{proof}
  \section*{Acknowledgments}
  The first two authors wish to thank Shayan Aziznejad, Michele Benzi and Michael Unser for inspiring conversations around the topic of this note. The third author wishes to thank Matteo Focardi and Flaviana Iurlano for interesting discussions on  Section~\ref{sectcpwl}. The authors wish to thank Gian Paolo Leonardi for comments leading to the investigation contained in Remark \ref{remnotclosed}.

 \bibliography{schatten}
 \bibliographystyle{alpha}
 
\end{document}